\documentclass[a4paper,reqno]{amsart}

\usepackage[utf8]{inputenc}
\usepackage{dsfont}
\usepackage[all]{xy}
\usepackage{mathrsfs}
\usepackage[hidelinks]{hyperref}
\usepackage[all]{xy}
\usepackage{amssymb}
\usepackage[a4paper]{geometry}
\usepackage[backend=biber,style=alphabetic,sorting=nyvt]{biblatex}
\usepackage{graphicx}
\usepackage{enumitem,nicefrac}
\usepackage{xcolor}
\usepackage{float}
\usepackage[foot]{amsaddr}

\DeclareNolabel{
  \nolabel{\regexp{[\p{Z}\p{P}\p{S}\p{C}]+}}
}

\newcommand*{\MCG}{\textup{MCG}}

\newcommand*{\Homeo}{\textup{Homeo}}
\newcommand*{\Id}{\textup{Id}}
\renewcommand*{\d}{\textup{d}}
\newcommand*{\GL}{\textup{GL}}
\newcommand*{\sign}{\textup{sign}}

\newcommand*{\Stab}{\textup{Stab}}
\newcommand*{\itop}{\overset{i}{\simeq}}
\newcommand*{\WP}{\textup{WP}}
\newcommand*{\PSL}{\textup{PSL}}
\newcommand*{\SL}{\textup{SL}}

\newcommand*{\Tr}{\textup{Tr}}
\newcommand*{\Tw}{\textup{Tw}}

\renewcommand{\hat}[1]{\widehat{#1}}
\newcommand{\argch}[1]{\textup{argch}#1}

\renewcommand{\cosh}[1]{\textup{ch}#1}
\renewcommand{\sinh}[1]{\textup{sh}#1}
\renewcommand{\coth}[1]{\textup{coth}#1}

\theoremstyle{plain}
    \newtheorem{Theo}{Theorem}[section]
    \newtheorem{Cor}[Theo]{Corollary}
    \newtheorem{Lem}[Theo]{Lemma}
    \newtheorem{Prop}[Theo]{Proposition}
\theoremstyle{definition}
    \newtheorem{Def}[Theo]{Definition}
\theoremstyle{remark}
    \newtheorem*{Rem}{Remark}
    \newtheorem*{Rems}{Remarks}
    \newtheorem*{Ex}{Example}

\title[Integrals of general geometric random variables]{Integrals of general geometric random variables on the moduli space of hyperbolic surfaces}

\author{Victor Le Guilloux\textsuperscript{1}}
\address[1]{IMJ-PRG site Sophie Germain, 8 Place Aurélie Nemours, 75013 Paris}
\email{victor.le-guilloux@u-pariscite.fr}


\begin{document}

\maketitle

\begin{abstract}
    In this article we provide an integration formula enabling us to integrate random variables defined on the moduli space of hyperbolic surfaces which involve the lengths of closed geodesics of any prescribed topological type, at fixed genus and number of boundaries. We then compute the general form of length functions of closed loops in Fenchel-Nielsen coordinates. Using this expression together with our integration formula, we express the expectations of geometric random variables as integrals over $\mathbf{R}$ with almost polynomial density. As a consequence, given an arbitrary closed loop $\gamma$, we obtain an improvement of Mirzakhani's asymptotic equivalent of the Weil-Petersson expectation $\mathbb{E}[N_\gamma(a)]$, when $a\to\infty$, of the number of geodesics with same topological type as $\gamma$ of length at most $a$. We finally provide, for a loop $\gamma$ filling a one-holed torus or a four-holed sphere, an interpretation of the density function associated with $\gamma$ as the period of the Hamiltonian flow of the length function of $\gamma$.
\end{abstract}

\tableofcontents

\section{Introduction}

\subsection{Main results}

Let $S_{g,n}$ be a fixed surface of genus $g$ with $n$ labelled boundaries $\alpha_1,\ldots,\alpha_n$. Assuming that $2g-2+n>0$, we can consider the moduli space $\mathcal{M}_{g,n}(\mathbf{L}_n)$ of hyperbolic metrics on $S_{g,n}$ with geodesic boundaries of lengths given by $\mathbf{L}_n=(L_1,\ldots,L_n)\in\mathbf{R}_{\geq0}^n$ (a boundary of length 0 is a cusp). In this article we study random variables on $\mathcal{M}_{g,n}(\mathbf{L}_n)$ whose expression involves the lengths of closed geodesics of any fixed topological type. In order to consider functions defined on $\mathcal{M}_{g,n}(\mathbf{L}_n)$ as random variables, we equip the moduli space with the Weil-Petersson probability measure, and we denote by $\mathbb{E}_{g,n}^{\mathbf{L}_n}$ the expectation for this measure and by $V_{g,n}(\mathbf{L}_n)$ the Weil-Petersson volume of $\mathcal{M}_{g,n}(\mathbf{L}_n)$.

Given a closed loop $\gamma$ on $S_{g,n}$, we denote by $\mathcal{O}_\gamma$ the mapping class group orbit of the free homotopy class $[\gamma]$ of $\gamma$, and if $f\colon S_{g,n}\longrightarrow X$ is an orientation-preserving homeomorphism from $S_{g,n}$ to a hyperbolic surface $X$, we denote for any test function $F\colon \mathbf{R}\longrightarrow\mathbf{C}$ (e.g. compactly supported):
\[
    F^\gamma(X)=\sum_{[\gamma']\in\mathcal{O}_\gamma}F(\ell_{f(\gamma')}(X)),
\]
where $\ell_{f(\gamma')}(X)$ is the hyperbolic length of the unique geodesic of $X$ freely homotopic to $f(\gamma')$. With this definition, $F^\gamma(X)$ does not depend on the choice of the homeomorphism $f$. As a consequence, $F^\gamma$ is a well-defined random variable on $\mathcal{M}_{g,n}(\mathbf{L}_n)$.

Our main result is the following.

\begin{Theo}\label{S1:Theo:MainIntro}
    Let $\gamma$ be a closed loop on $S_{g,n}$. There exists a non-negative function $V_{\mathcal{O}_\gamma}(~.~,\mathbf{L}_n)\in L_{loc}^1(\mathbf{R}_{\geq0})$, depending on the orbit of $\gamma$, such that for every test function $F\colon \mathbf{R}\longrightarrow\mathbf{C}$, we have:
    \[
        \mathbb{E}_{g,n}^{\mathbf{L}_n}[F^\gamma]=\int_0^\infty F(\ell)V_{\mathcal{O}_\gamma}(\ell,\mathbf{L}_n)\frac{\d\ell}{V_{g,n}(\mathbf{L}_n)}.
    \]
    Moreover there exist a polynomial $P_\gamma^{\mathbf{L}_n}$ of degree $6g-7+2n$ and constants $c_0,c>0$ and $0<\lambda_\gamma<1$ such that for every $m\geq1$ we have:
    \[
        \int_0^me^\ell|V_{\mathcal{O}_\gamma}(\ell,\mathbf{L}_n)-P_\gamma^{\mathbf{L}_n}(\ell)|\d\ell\leq c_0(1+m)^ce^{(1-\lambda_\gamma)m}.
    \]
\end{Theo}

As a consequence, we obtain:

\begin{Cor}\label{S1:Cor:MainIntro}
    Let $\gamma$ be a closed loop on $S_{g,n}$. Given $a\geq0$ and a hyperbolic metric $X\in\mathcal{M}_{g,n}(\mathbf{L}_n)$, we denote by $N_\gamma(a)(X)$ the number of geodesics of $X$ in the same mapping class group orbit as $\gamma$ with length at most $a$. There exist a polynomial $Q_\gamma^{\mathbf{L}_n}$ of degree $6g-6+2n$ and constants $c>0$ and $0<\lambda_\gamma<1$ such that:
    \[
        \mathbb{E}_{g,n}^{\mathbf{L}_n}[N_\gamma(a)]=Q_\gamma^{\mathbf{L}_n}(a)+\underset{a\to\infty}{O}\big((1+a)^ce^{-\lambda_\gamma a}\big).
    \]
\end{Cor}

\begin{Rem}
    The method we use to prove Theorem \ref{S1:Theo:MainIntro} does not necessarily provide an optimal coefficient $\lambda_\gamma$. Indeed, if we apply the method of this article to an eight-shaped loop $\gamma$, depending on the orbit of $\gamma$ we could obtain $\lambda_\gamma=\frac{1}{4}$ instead of the $\lambda_\gamma=\frac{1}{2}$ we computed in \cite[Paragraph 4.2.2]{LeG25}.
\end{Rem}

For filling loops on specific \emph{small} hyperbolic surfaces, we provide a geometric interpretation of the density function $V_{\mathcal{O}_\gamma}$.

\begin{Theo}
    Let $(g,n)\in\{(1,1),(0,4)\}$ and $\gamma$ be a filling loop on $S_{g,n}$. The Hamiltonian flow of the length function $\ell_\gamma$ is periodic, and its period on the level set $\ell_\gamma^{-1}(\ell)$ is $V_{\mathcal{O}_\gamma}(\ell,\mathbf{L}_n)$.
\end{Theo}

In what follows, we will work with fixed values of $g,n$ and $\mathbf{L}_n$, hence we shall only denote by $\mathbb{E}$ the Weil-Petersson expectation on $\mathcal{M}_{g,n}(\mathbf{L}_n)$.

\subsection{The Weil-Petersson model for random hyperbolic surfaces}

A \emph{marked hyperbolic surface} is a couple $(X,f)$, where $X$ is a hyperbolic surface with labelled geodesic boundaries $\beta_1,\ldots,\beta_n$ of lengths $L_1,\ldots,L_n$ respectively and $f\colon S_{g,n}\longrightarrow X$ is a homeomorphism which preserves the orientation and the labelling of the boundaries. Two marked hyperbolic surfaces $(X,f)$ and $(Y,g)$ are said to be \emph{equivalent} is $g\circ f^{-1}$ is isotopic to an isometry. The Teichmüller space $\mathcal{T}_{g,n}(\mathbf{L}_n)$ is the set of equivalence classes of marked hyperbolic surfaces. This space is a manifold diffeomorphic to $(\mathbf{R}_{>0}\times\mathbf{R})^{3g-3+n}$ through \emph{Fenchel-Nielsen coordinates} and carries a natural symplectic form $\omega_{\WP}$ called the \emph{Weil-Petersson form}. We recall more precisely the definition of Fenchel-Nielsen coordinates in Paragraph \ref{S1:Subsec3}.

The Teichmüller space is additionnally equipped with an action of the \emph{mapping class group}, which is defined by:
\[
    \MCG_{g,n}=\Homeo^+(S_{g,n},\partial S_{g,n})/\Homeo_\circ^+(S_{g,n},\partial S_{g,n}),
\]
where $\Homeo^+(S_{g,n},\partial S_{g,n})$ is the group of orientation preserving homeomorphisms that restrict to the identity on $\partial S_{g,n}$, and $\Homeo_\circ^+(S_{g,n},\partial S_{g,n})$ is its neutral component, i.e. the subgroup of homeomorphisms isotopic to the identity relatively to $\partial S_{g,n}$. If we consider a surface $S$ whose genus and number of boundaries or punctures is not specified, we denote by $\MCG(S)$ its mapping class group. The \emph{moduli space} is the quotient space:
\[
    \mathcal{M}_{g,n}(\mathbf{L}_n)=\mathcal{T}_{g,n}(\mathbf{L}_n)/\MCG_{g,n}.
\]
Unlike the Teichmüller space, $\mathcal{M}_{g,n}(\mathbf{L}_n)$ is not a manifold, but an \emph{orbifold}. However the action of $\MCG_{g,n}$ on $\mathcal{T}_{g,n}(\mathbf{L}_n)$ is symplectic, hence the Liouville measure associated to the Weil-Petersson symplectic form $\omega_{\WP}$ induces a well-defined measure $\mu_{\WP}$ on $\mathcal{M}_{g,n}(\mathbf{L}_n)$, called the \emph{Weil-Petersson measure}. The moduli space $\mathcal{M}_{g,n}(\mathbf{L}_n)$ has a finite volume $V_{g,n}(\mathbf{L}_n)$. More precisely we have:

\begin{Theo}[\cite{Mir07-1}]
    The Weil-Petersson volume $V_{g,n}(\mathbf{L}_n)$ is a polynomial in the variables $L_1,\ldots,L_n$ of total degree $6g-6+2n$, and we can write:
    \[
        V_{g,n}(\mathbf{L}_n)=\sum_{i_1+\cdots+i_n\leq 3g-3+n}C_{i_1\ldots i_n}L_1^{2i_1}\cdots L_n^{2i_n},
    \]
    where $C_{i_1\ldots i_n}\in\pi^{6g-6+2n-2(i_1+\cdots+i_n)}\mathbf{Q}_{>0}$.
\end{Theo}

When equipped with the renormalized measure $\frac{1}{V_{g,n}(\mathbf{L}_n)}\mu_{\WP}$, the moduli space is a probability space.

\subsection{Loops on a surface and Fenchel-Nielsen coordinates}\label{S1:Subsec3}

A \emph{loop} is a continuous map $\gamma$ from $[0,1]$ to a surface $X$ such that $\gamma(0)=\gamma(1)$. Two loops $\gamma_1,\gamma_2\colon[0,1]\longrightarrow X$ are \emph{freely homotopic} if there exists a continuous map $h\colon[0,1]^2\longrightarrow X$ called a \emph{homotopy}, such that $h_{|\{0\}\times[0,1]}=\gamma_1$ and $h_{|\{1\}\times[0,1]}=\gamma_2$ and such that for all $t\in[0,1]$, $h_{|\{t\}\times[0,1]}$ is a loop. The free homotopy class of a loop $\gamma$ will be denoted by $[\gamma]$. A loop $\gamma$ is said to be \emph{non-primitive} if it is homotopic to $\gamma^{\prime m}$ for some other loop $\gamma'$ and $m\geq2$, and $\gamma$ is \emph{primitive} otherwise.

Let $\gamma$ be a loop on a surface $X$. We say that $\gamma$ is in \emph{minimal position} if it is an immersion, if its self-intersection points are double points, and if $\gamma$ has the minimal amount of self-intersections among all such loops in the free homotopy class of $\gamma$. The loop $\gamma$ is \emph{essential} if it is non-contractible and not homotopic to any of the boundaries and cusps of $X$.

A \emph{multi-loop} is a family a loops $\Gamma=(\gamma_1,\ldots,\gamma_k)$. Given a multi-loop $\Gamma=(\gamma_1,\ldots,\gamma_k)$ made of essential loops on a hyperbolic surface $X$, there exists a unique family $\Gamma'=(\gamma'_1,\ldots,\gamma'_k)$ of geodesics on $X$ such that $\Gamma$ is freely homotopic to $\Gamma'$ (i.e. such that $\gamma_i$ is homotopic to $\gamma'_i$ for all $i\in\{1,\ldots,k\}$). We denote by $\ell_\Gamma(X)$ the vector $(\ell_{\gamma'_1}(X),\ldots,\ell_{\gamma'_k}(X))$ made of the lengths of the geodesic representatives of the free homotopy classes $[\gamma_1],\ldots,[\gamma_k]$. This definition of $\ell_\Gamma(X)$ holds even if some of the loops $\gamma_i$ are not essential (with the convention that the length of a cusp or a contractible loop is 0). A simple loop, i.e. a loop with no self-intersection, will also be called a \emph{curve}. A \emph{multi-curve} is a family of curves $\Gamma=(\gamma_1,\ldots,\gamma_k)$ where $\gamma_1,\ldots,\gamma_k$ are disjoint essential curves, such that for any $i\neq j$, $\gamma_i$ is not homotopic to $\gamma_j$ or $\gamma_j^{-1}$.

Given a multi-loop $\Gamma=(\gamma_1,\ldots,\gamma_k)$ on the base surface $S_{g,n}$, we define the \emph{length function associated to $\Gamma$}, denoted $\ell_\Gamma\colon\mathcal{T}_{g,n}(\mathbf{L}_n)\longrightarrow\mathbf{R}^k$, the following way: given a marked hyperbolic surface, $\ell_\Gamma(X,f)$ is the length vector of the geodesic representative of $f(\Gamma)$ on $X$ (i.e. $\ell_\Gamma(X,f)=\ell_{f(\Gamma)}(X)$).

A maximal multi-curve on $S_{g,n}$, i.e. a multi-curve which contains the maximal possible number of curves, contains exactly $3g-3+n$ curves and is called a \emph{pants decomposition of $S_{g,n}$}. Every multi-curve can be completed into a pants decomposition of $S_{g,n}$. Given a pants decomposition $\mathscr{P}$ of $S_{g,n}$, the complement $S_{g,n}\backslash\mathscr{P}$ is a disjoint union of three-holed (or punctured) spheres, called \emph{pairs of pants}. If $P$ is such a pair of pants and $(X,f)$ is a marked hyperbolic surface, the hyperbolic metric inside $f(P)$ is uniquely determined by the length vector $\ell_{\partial P}(X,f)\in\mathbf{R}_{\geq0}^3$. To recover the hyperbolic metric on the whole surface $(X,f)$, we additionally require a gluing parameter $\tau_\alpha(X,f)$ for each curve $\alpha\in\mathscr{P}$, called the \emph{twist parameter along $\alpha$}, defined to be the oriented length of the arc of $f\circ\alpha$ between two common perpendiculars of $X$ joining $f\circ\alpha$ and neighboring geodesics of $f\circ\alpha$ in $f\circ\mathscr{P}$. This twist belongs to $\mathbf{R}$ in order to distinguish from $(X,f)$ any marked hyperbolic surface $(X,f')$ obtained from $(X,f)$ by performing a Dehn twist along $\alpha$ (the underlying hyperbolic surface of such a marked hyperbolic surface $(X,f')$ would be isometric to the one of $(X,f)$, however the marking homeomorphisms $f$ and $f'$ are not isotopic and the marked hyperbolic surfaces $(X,f)$ and $(X,f')$ therefore represent distinct points in $\mathcal{T}_{g,n}(\mathbf{L}_n)$). The data of $(\ell_\alpha,\tau_\alpha)_{\alpha\in\mathscr{P}}$ is called a set of \emph{Fenchel-Nielsen coordinates on $\mathcal{T}_{g,n}(\mathbf{L}_n)$} and provides a diffeomorphism between the Teichmüller space $\mathcal{T}_{g,n}(\mathbf{L}_n)$ and $(\mathbf{R}_{>0}\times\mathbf{R})^{3g-3+n}$. Wolpert proved in \cite{Wol85} that this diffeomorphism is actually a symplectomorphism between the Weil-Petersson structure of $\mathcal{T}_{g,n}(\mathbf{L}_n)$ and the standard symplectic structure of $(\mathbf{R}_{>0}\times\mathbf{R})^{3g-3+n}$.

\subsection{Integration of geometric random variables and average geodesic counting}

The mapping class group also acts on the set of free homotopy classes of multi-loops on $S_{g,n}$ by $[\varphi]\cdot[\Gamma]=[\varphi\circ\Gamma]$, and if $\Gamma$ is a multi-loop, we denote by $\mathcal{O}_\Gamma$ the orbit of the free homotopy class $[\Gamma]$, and by $\Stab([\Gamma])$ the stabilizer of $[\Gamma]$, i.e. the subgroup of $\MCG_{g,n}$ made of the elements $[\varphi]$ such that $\varphi\circ\gamma$ is homotopic to $\gamma$.

In this article we will be interested in random variables on $\mathcal{M}_{g,n}(\mathbf{L}_n)$ defined the following way. Given a hyperbolic surface $X$, each compact set of $\mathbf{R}$ contains only a finite number of lengths of closed geodesics on $X$. Hence if $F\colon \mathbf{R}^k\longrightarrow\mathbf{C}$ is a test function (e.g. compactly supported), we can consider the following sum:
\[
    \sum_{[\Gamma']\in\mathcal{O}_\Gamma}F(\ell_{\Gamma'}(X,f)),
\]
where $(X,f)$ is a marked hyperbolic surface. Since the sum is made over the whole orbit of $\Gamma$, it does not depend on the choice of the marking homeomorphism $f$. This sum therefore induces a well-defined random variable on $\mathcal{M}_{g,n}(\mathbf{L}_n)$, which will be denoted by:
\begin{equation}\label{S1:Eq:DefGeomRandVar}
    F^\Gamma\colon  X\mapsto\sum_{[\Gamma']\in\mathcal{O}_\Gamma}F(\ell_{\Gamma'}(X)),
\end{equation}
i.e. we omit the marking homeomorphism $f$.

\begin{Rems}
    \begin{itemize}
        \item The notations in the sum (\ref{S1:Eq:DefGeomRandVar}) are abusive. Indeed $\Gamma'$ is not a multi-loop on $X$, hence we would only be able to consider $\ell_{\Gamma'}(X,f)$, for some marking homeomorphism $f\colon S_{g,n}\longrightarrow X$, but as we said before, since the sum is made on the whole orbit, the value of $F^\Gamma(X)$ does not depend on $f$.
        \item Let $\gamma$ be a closed loop on $S_{g,n}$. If we take $F=\mathds{1}_{[0,a]}$ in (\ref{S1:Eq:DefGeomRandVar}), where $a\geq0$, we retrieve the counting random variable $N_\gamma(a)$.
    \end{itemize}
\end{Rems}

Mirzakhani provided an integration formula which enables us to integrate random variables of the form $F^\Gamma$, when $\Gamma$ is a multi-curve.

\begin{Theo}[Mirzakhani's integration formula, \cite{Mir07-2}]\label{S1:Theo:MirzInt}
    Let $\Gamma=(\gamma_1,\ldots,\gamma_k)$ be a multi-curve on $S_{g,n}$. If $F\colon \mathbf{R}^k\longrightarrow\mathbf{C}$ is a test function, we have:
    \[
        \mathbb{E}[F^\gamma]=\frac{1}{2^{M(\gamma)}}\int_{\mathbf{R}_{>0}^k}F(x_1,\ldots,x_k)x_1\cdots x_kV_\Gamma(x_1,\ldots,x_k,\mathbf{L}_n)\frac{\d x_1\cdots\d x_k}{V_{g,n}(\mathbf{L}_n)},
    \]
    where $M(\gamma)$ is the number of connected components of the complement $S_{g,n}\backslash\Gamma$ of the multicurve $\Gamma$ homeomorphic to $S_{1,1}$ and $V_\Gamma(x_1,\ldots,x_k,\mathbf{L}_n)$ is the volume of the moduli space of the complement\footnote{By \emph{moduli space of the complement}, we mean the moduli space of hyperbolic structures on $S_{g,n}\backslash\Gamma$.} $S_{g,n}\backslash\Gamma$, where the two boundary components of $S_{g,n}\backslash\Gamma$ corresponding to the curve $\gamma_i$ have length $x_i$. If $S_{g,n}\backslash\Gamma$ is disconnected, then $V_\Gamma(x_1,\ldots,x_k,\mathbf{L}_n)$ is the product of the volumes of the moduli spaces of the connected components\footnote{We refer to \cite{Mir07-2} for a more expanded expression of $V_\Gamma(x_1,\ldots,x_k,\mathbf{L}_n)$.}. In particular $V_\Gamma$ is a polynomial.
\end{Theo}

\begin{Cor}
    If $\gamma$ is a closed curve (i.e. $\gamma$ is a simple loop) on $S_{g,n}$, then $\mathbb{E}[N_\gamma(a)]$ is a polynomial of degree $6g-6+2n$ in the variable $a$.
\end{Cor}

Mirzakhani also computed an asymptotic equivalent of $N_\gamma(a)$ and $\mathbb{E}[N_\gamma(a)]$ when $a\to\infty$, for any loop $\gamma$ on $S_{g,n}$.

\begin{Theo}[\cite{Mir16}]\label{S1:Theo:MirzAsympt}
    Let $\gamma$ be a closed loop on $S_{g,n}$. There exists a constant $n_\gamma\in\mathbf{Q}$ and a smooth and integrable function $B\colon \mathcal{M}_{g,n}(\mathbf{L}_n)\longrightarrow\mathbf{R}_{\geq0}$ such that for all $X\in\mathcal{M}_{g,n}(\mathbf{L}_n)$, we have:
    \[
        N_\gamma(a)(X)\underset{a\to\infty}{\sim}n_\gamma\frac{B(X)}{||B||_{L^1}}a^{6g-6+2n}.
    \]
\end{Theo}

Mirzakhani proved in the same paper that this result can be integrated in order to obtain the following.

\begin{Cor}[\cite{Mir16}]\label{S1:Cor:MirzAsympt}
    Let $\gamma$ be a closed loop on $S_{g,n}$. We have:
    \[
        \mathbb{E}[N_\gamma(a)]\underset{a\to\infty}{\sim}\frac{n_\gamma}{V_{g,n}(\mathbf{L}_n)}a^{6g-6+2n}.
    \]
\end{Cor}

Erlandsson and Souto obtained similar results for geodesics on a once-holed torus in \cite{ErSo16} by using a different approach, and Erlandsson, Parlier and Souto in \cite{ErPaSo20}, Erlandsson and Souto in \cite{ErSo24} and Trin in \cite{Tri24} proved similar results for more general notions of lengths of loops. Given $a>0$ and an integer $k\geq0$, Sapir computed in \cite{Sap16-01,Sap16-02} bounds on the number of closed geodesics of length at most $a$ and with at most $k$ self-intersections on any hyperbolic surface.

We also obtained in \cite{LeG25} a more precise asymptotic expansion of $\mathbb{E}[N_\gamma(a)]$ when $a\to\infty$ than the one from Corollary \ref{S1:Cor:MirzAsympt} for eight-shaped loops $\gamma$. However the tools we used to study the figure eight do not work anymore for more complex loops. The aim of this article is therefore to develop a new method suited to the study of arbitrary loops.

Even though our Corollary \ref{S1:Cor:MainIntro} improves the asymptotic equivalent of Corollary \ref{S1:Cor:MirzAsympt}, we do not know if it is possible to deduce a better pointwise estimate than Theorem \ref{S1:Theo:MirzAsympt} from our Theorem \ref{S1:Theo:MainIntro}.

\subsection{Links with the work of Anantharaman and Monk}

The lengths of closed geodesics on a hyperbolic surface $X$ are linked to the eigenvalues of the Laplace operator on $X$ by the Selberg trace formula \cite{Sel56}. Therefore, having a better understanding of the behavior of random variables of the form $F^\Gamma$ could help us understanding the spectral geometry of random hyperbolic surfaces.

Anantharaman and Monk exploited this fact in \cite{AnMo24,AnMo25} in order to compute the optimal spectral gap for random hyperbolic surfaces of large genus. Since they consider closed surfaces of growing genus $g$, the notion of $\MCG_{g,n}$-orbits of loops is not well-suited to their study. To remedy this problem they defined the notion of \emph{local topological type} of a loop. Given a local topological type $T$, and a test function $F$, the \emph{$T$-average of $F$ over surfaces of genus $g$} is defined to be:
\[
    \langle F\rangle^T_g=\mathbb{E}_{g,0}\bigg[\sum_{\substack{\gamma\sim T\\\gamma\in\mathcal{G}(X)}}F(\ell_\gamma(X))\bigg],
\]
where $\mathcal{G}(X)$ is the set of closed geodesics on $X\in\mathcal{M}_g=\mathcal{M}_{g,0}$ and $\gamma\sim T$ means that $\gamma$ has local topological type $T$. Anantharaman and Monk proved in \cite{AnMo24} that there exists $V_g^T\in L_{loc}^1(\mathbf{R}_{\geq0})$, called the \emph{volume function associated with local topological type $T$ on surfaces of genus $g$}, such that for any test function $F$, we have:
\[
    \langle F\rangle_g^T=\int_0^\infty F(\ell)V_g^T(\ell)\d\ell.
\]
Contrary to Anantharaman and Monk, we fix the topology of the base surface $S_{g,n}$ (i.e. we fix $g$ and $n$), hence we are more interested in studying quantities of the form $\mathbb{E}[F^\gamma]$ rather than $\langle F\rangle_g^T$, i.e. we sort the loops according to their $\MCG_{g,n}$-orbit and not their local topological type. However we can adapt some of the tools developed in \cite{AnMo24,AnMo25} to our framework. We can therefore obtain the existence of functions $V_{\mathcal{O}_\gamma}(~.~,\mathbf{L}_n)$ in Theorem \ref{S1:Theo:MainIntro} by using similar arguments as for the existence of the volume functions $V_g^T$. In this article, we will be studying the large $\ell$ behavior of $V_{\mathcal{O}_\gamma}(\ell,\mathbf{L}_n)$. The proof of the main result of our previous article \cite{LeG25} resorted to explicit computations, thus since we cannot rely anymore on such explicit calculations for arbitrary loops, the new method we develop in this article is different from the one of \cite{LeG25}. We hope that our estimates will help us in the future having a better understanding of the spectral behavior of random hyperbolic surfaces of fixed topology.

\subsection{Length-type functions and pseudo-convolution}

The density function $V_{\mathcal{O}_\gamma}(~.~,\mathbf{L}_n)$ in Theorem \ref{S1:Theo:MainIntro} is a linear combination of \emph{pseudo-convolution of polynomials}, obtained by disintegrating the measure $x_1\ldots x_n V_\Gamma(x_1,\ldots,x_n,\mathbf{L}_n)\d x_1\cdots\d x_n$ along the level sets of the length function $h_\gamma$ of the loop $\gamma$ (written in appropriate Fenchel-Nielsen coordinates on $\mathcal{T}_{g,n}(\mathbf{L}_n)$). One of the contributions of this paper is that we define the class of \emph{length-type functions} and prove that $h_\gamma$ is a length-type function. Length-type functions behave similarly to piecewise linear forms up to a $\mathscr{C}^1$-small error term. We then study the properties of length-type functions and we compute the asymptotic behavior of the pseudo-convolution of polynomials with respect to length-type functions.

The process of pseudo-convolution was introduced by Anantharaman and Monk in \cite{AnMo25} to compute the behavior at high genus of the volume functions $V_g^T$, and we used it in \cite{LeG25} to study eight-shaped loops. The usual convolution of locally integrable functions $f_1,\ldots,f_n$ consists in disintegrating the measure $f_1\cdots f_n\d x_1\cdots\d x_n$ along the level sets of the linear form $(x_1,\ldots,x_n)\mapsto x_1+\cdots+x_n$, and in a similar manner, the \emph{pseudo-convolution of $f_1,\ldots,f_n$} consists in disintegrating $f_1\cdots f_n\d x_1\cdots\d x_n$ along the level set of a more general function $h\colon\mathbf{R}^n\longrightarrow\mathbf{R}$ which behaves similarly to a (piecewise) linear form.

\subsection{Acknoledgements}

I would like to thank Nalini Anantharaman for her invaluable help with the trigonometric calculations of Paragraph \ref{S3:Subsec:GenCase} and for her helpful comments during the preparation of this article and Anton Zorich for his advices. This research has been partially funded by the European Research Council
(ERC) under the European Union’s Horizon 2020 research and innovation programme (Grant agreement
No. 101096550 InSpeGMoS and No. 101141508 UniGeoDyM).

\section{An integration formula for general geometric random variables}

As soon as a loop $\gamma$ on $S_{g,n}$ is topologically complex enough, the expression of its length function in Fenchel-Nielsen coordinates will involve twist parameters. Hence Mirzakhani's integration formula (Theorem \ref{S1:Theo:MirzInt}) is not suited anymore. The goal of this section is to generalize Mirzakhani's formula in order to be able to integrate general geometric random variables.

\subsection{The surface filled by a loop}

The topological complexity of a loop $\gamma$ we mentionned above is encoded by the \emph{surface filled by $\gamma$}. We begin by recalling the definition of the surface filled by a loop.

\begin{Def}\label{S2:Def:SurfFilled}
    Let $\gamma$ be a closed loop in minimal position on a hyperbolic surface $X$. The \emph{surface filled by $\gamma$}, denoted by $S(\gamma)$ is defined the following way:
    \begin{enumerate}[label=\textup{(\roman*)}]
        \item we consider $\varepsilon>0$ small enough so the regular neighborhood $\mathcal{V}_\varepsilon(\gamma)=\{x\in X|d(x,\gamma)<\varepsilon\}$ retracts to $\gamma$;
        \item denoting by $C_1,\ldots,C_k$ the connected components of $X\backslash\mathcal{V}_\varepsilon(\gamma)$ homeomorphic to a disk, we define:
        \[
            S(\gamma)= \mathcal{V}_\varepsilon(\gamma)\cup\bigcup_{i=1}^kC_i.
        \]
    \end{enumerate}
\end{Def}

\begin{Rem}
    We need the surface $X$ to be endowed with a hyperbolic metric only to be able to consider the regular neighborhood $\mathcal{V}_\varepsilon(\gamma)$. But the surface filled by $\gamma$ is well-defined up to isotopy, and it does not depend on the choice of the metric. Hence we can consider the surface filled by a loop $\gamma$ on the base surface $S_{g,n}$.
\end{Rem}

\begin{Ex}
    The surface filled by a figure-eight is a pair of pants.
\end{Ex}

We will prove in Paragraph \ref{S3:Subsec:GenCase} that as soon as $S(\gamma)$ is neither an annulus nor a pair of pants, then the length function of the loop $\gamma$ will involve twist parameters when written in Fenchel-Nielsen coordinates.

\subsection{Integration on the moduli space}

Let $\mathcal{D}\subset\mathcal{T}_{g,n}(\mathbf{L}_n)$ be a fundamental domain of $\mathcal{M}_{g,n}(\mathbf{L}_n)$ and let $f\colon\mathcal{T}_{g,n}(\mathbf{L}_n)\longrightarrow\mathbf{C}$ be an $\MCG_{g,n}$-invariant test function. The function $f$ therefore induces a well-defined map on $\mathcal{M}_{g,n}(\mathbf{L}_n)$ which we still denote by $f$. We have:
\[
    \int_{\mathcal{M}_{g,n}(\mathbf{L}_n)}f\d\mu_\WP=\int_\mathcal{D}f\d\mu_\WP.
\]
If $f$ is only invariant under the action of a subgroup $G\subset\MCG_{g,n}$, denote:
\begin{equation}\label{S2:Eq:FoncInvSubGrp}
    f_G\colon(X,h)\longmapsto\sum_{\psi\in G\backslash\MCG_{g,n}}f(\psi\cdot(X,h)),
\end{equation}
where the sum over $\psi\in G\backslash\MCG_{g,n}$ means that we take exactly one element in each right coset modulo $G$. The function $f_G$ is invariant under the action of $\MCG_{g,n}$, and if $\mathcal{D}'\subset\mathcal{T}_{g,n}(\mathbf{L}_n)$ is a fundamental domain for the action of $G$, we have:
\begin{equation}\label{S2:Eq:IntQuot}
    \int_{\mathcal{M}_{g,n}(\mathbf{L}_n)}f_G\d\mu_\WP = \int_{\mathcal{D}'}f\d\mu_\WP.
\end{equation}

\subsection{Stabilizer of a loop in the mapping class group}

Let $\gamma$ be a loop on $S_{g,n}$ and $S(\gamma)$ be the surface filled by $\gamma$. We denote $G(\gamma)=\MCG(S_{g,n}\backslash S(\gamma))\subset\MCG_{g,n}$.

\begin{Rem}
    The group $G(\gamma)$ contains the Dehn twists along $\partial S(\gamma)$.
\end{Rem}

Let $F\colon\mathbf{R}\longrightarrow\mathbf{C}$ be a test function, we denote $\widetilde{F}=F\circ\ell_\gamma\colon\mathcal{T}_{g,n}(\mathbf{L}_n)\longrightarrow\mathbf{R}$. In order to get an integration formula for $F^\gamma$, we shall compare $F^\gamma$ and $\widetilde{F}_{G(\gamma)}$ (where $\widetilde{F}_{G(\gamma)}$ is defined as (\ref{S2:Eq:FoncInvSubGrp}) with $f=\widetilde{F}$ and $G=G(\gamma)$)

Let us first introduce some notations. Given two loops $\gamma,\gamma'\colon[0,1]\longrightarrow S_{g,n}$, we denote by $\gamma\itop\gamma'$ if $\gamma$ and $\gamma'$ are isotopic, i.e. if there exists an isotopy $(\psi_t)$ of $S_{g,n}$ such that $\psi_1\circ\gamma=\gamma'$. Following the definition of \cite{HaSc99}, we say that $\gamma$ and $\gamma'$ \emph{have the same configuration} if there is an isotopy of $S_{g,n}$ sending $\gamma([0,1])$ to $\gamma'([0,1])$. Recall that we denote by $[\gamma]$ the free homotopy class of $\gamma$, and we shall note $(\gamma)$ its isotopy class.

\begin{Lem}\label{S2:Lem:ConfigItop}
    The configuration class of a loop is partitionned into a finite number of isotopy classes.
\end{Lem}

\begin{proof}
    Let $\gamma$ and $\gamma'$ be two loops on $S_{g,n}$ which have the same configuration, and let $(\psi_t)$ be an isotopy such that $\psi_1(\gamma([0,1]))=\gamma'([0,1])$. Then $\psi_1\circ\gamma$ is a loop with same image as $\gamma'$, and since $\psi_1$ is a homeomorphism, the loop $\psi_1\circ\gamma$ is a reparametrization of $\gamma'$ and the ordered self-intersections of $\psi_1\circ\gamma$ are obtained from those of $\gamma'$ by applying a cyclic permutation. Two reparametrizations of $\gamma'$ with same self-intersections cyclic order are isotopic, and $\gamma'$ has a finite number of self-intersections, hence $\gamma$ can belong to only finitely many isotopy classes.
\end{proof}

\begin{Lem}
    The free homotopy class of a loop is partitionned into a finite number of isotopy classes.
\end{Lem}

\begin{proof}
    According to \cite{HaSc99}, each free homotopy class is partitionned into a finite number of configuration classes. The conclusion is immediate following Lemma \ref{S2:Lem:ConfigItop}.
\end{proof}

\begin{Lem}\label{S2:Lem:HomFillLoop}
    Let $\gamma$ be a filling loop on $S_{g,n}$ and let $\varphi\in\Homeo^+(S_{g,n},\partial S_{g,n})$ be such that $\varphi\circ\gamma\itop\gamma$. Then $\varphi$ is homotopic to a product of Dehn twists along $\partial S_{g,n}$.
\end{Lem}

\begin{proof}
    By composing with an isotopy, we can assume without loss of generality that $\varphi\circ\gamma=\gamma$. Let $V$ be a regular neighborhood of $\gamma([0,1])$, small enough so that $V$ retracts to $\gamma$ via retraction map $(r_t)$. Then $\varphi(V)$ is also a neighborhood of $\gamma([0,1])$ which retracts onto $\gamma([0,1])$, let $W\varsubsetneq V\cap\varphi(V)$ be another small regular neighborhood of $\gamma([0,1])$. There exists an isotopy $(\psi_t)$ such that $\psi_{t|W}=\Id_W$ for all $t\in[0,1]$ sending $\varphi(V)$ to $V$, hence we can assume without loss of generality that $\varphi(V)=V$. Moreover $(\varphi\circ r_t)$ is a homotopy between $\varphi\circ r_0=\varphi$ and $\varphi\circ r_1=r_1$. Hence $\varphi_{|V}$ is homotopic to $\Id_V$. Since $\varphi_{|V}$ and $\Id_V$ are two orientation preserving homeomorphisms of $V$, $\varphi$ does not permute the connected components of $\partial V$.

    On the other hand every connected component of $S_{g,n}\backslash V$ is a disc or an annulus around a boundary of $S_{g,n}$. Thus, in a disc of the complement of $V$, $\varphi$ is homotopically trivial, and in an annulus of the complement $V$, $\varphi$ is homotopic to a Dehn twist along the corresponding boundary of $\partial S_{g,n}$ (relatively to $\partial S_{g,n}$). This yields the conclusion.
\end{proof}

\begin{Lem}\label{S2:Lem:IndexStabTwist}
    Let $\gamma$ be a filling loop on $S_{g,n}$ and let us denote by $\Tw(\partial S_{g,n})$ the group of Dehn twists along $\partial S_{g,n}$. The index $[\Stab([\gamma]):\Tw(\partial S_{g,n})]$ is finite.
\end{Lem}

\begin{proof}
   Let $[\varphi]\in\Stab([\gamma])$. We partition the free homotopy class $[\gamma]$ into the isotopy classes $(\gamma_1),\ldots,(\gamma_k)$. There exists a unique $j\in\{1,\ldots,k\}$ such that $\varphi\circ\gamma\itop\gamma_j$. If $[\psi]\in\Stab([\gamma])$ also satisfies $\psi\circ\gamma\itop\gamma_j$, then $\psi^{-1}\circ\varphi\circ\gamma\itop\gamma$. Hence, according to Lemma \ref{S2:Lem:HomFillLoop} we can conclude that $\psi^{-1}\circ\varphi$ is homotopic to a product of Dehn twists along $\partial S_{g,n}$, i.e. $[\psi]^{-1}[\varphi]\in\Tw(\partial S_{g,n})$. It follows that $[\Stab([\gamma]):\Tw(\partial S_{g,n})]=k$.
\end{proof}

\begin{Lem}\label{S2:Lem:StabIndex}
    Given any loop $\gamma$ on $S_{g,n}$, the index $m(\gamma)=[\Stab([\gamma]):G(\gamma)]$ is finite.
\end{Lem}

\begin{proof}
    Let $\beta_1,\ldots,\beta_N$ be the essential boundary components of $\partial S(\gamma)$. If $[\varphi]\in\Stab([\gamma])$, then $\varphi$ possibly permutes the $\beta_i$'s, i.e. there exists $\sigma_\varphi\in\mathfrak{S}_N$ such that $\varphi(\beta_i)$ is isotopic to $\beta_{\sigma_\varphi(i)}$ for all $i\in\{1,\ldots,N\}$. This defines a group morphism $\Stab([\gamma])\longrightarrow\mathfrak{S}_N$, let us denote by $K$ its kernel.

    For each $i\in\{1,\ldots,N\}$, we fix a curve $\widetilde{\beta}_i$ isotopic to $\beta_i$, contained inside $S(\gamma)$. Since $\widetilde{\beta}_i$ and $\beta_i$ are isotopic, the Dehn twists along $\beta_i$ and $\widetilde{\beta}_i$ are the same element in the mapping class group. But by performing the Dehn twists along the $\widetilde{\beta}_i$'s instead of the $\beta_i$'s, we can assume without loss of generality that an element $[\varphi]\in K$ can be represented by a homeomorphism $\varphi$ such that $\varphi_{|\partial S(\gamma)}=\Id_{\partial S(\gamma)}$. Given $[\varphi]\in K$, represented by $\varphi$ which restricts to the identity on $\partial S(\gamma)$, we can write $\varphi=\varphi_1\circ\varphi_2$, where $\varphi_{1|S(\gamma)}=\varphi_{|S(\gamma)}$ and $\varphi_{1|S_{g,n}\backslash S(\gamma)}=\Id_{S_{g,n}\backslash S(\gamma)}$ (and therefore $\varphi_{2|S(\gamma)}=\Id_{S(\gamma)}$ and $\varphi_{2|S_{g,n}\backslash S(\gamma)}=\varphi_{|S_{g,n}\backslash S(\gamma)}$). Hence we have $[\varphi]=[\varphi_1]\cdot[\varphi_2]$ and $[\varphi_2]\in G(\gamma)$. Moreover, denoting by $\Stab([\gamma],S(\gamma))$ the stabilizer of $[\gamma]$ in $\MCG(S(\gamma))$, on the one hand we have $[\varphi_1]\in\Stab([\gamma],S(\gamma))$, and on the other hand according to Lemma \ref{S2:Lem:IndexStabTwist} there exist finitely many $[\psi_1],\ldots,[\psi_k]\in\Stab([\gamma],S(\gamma))$ (independent of $\varphi$) such that $\Stab([\gamma],S(\gamma))=[\psi_1]\Tw(\partial S(\gamma))\sqcup\cdots\sqcup[\psi_k]\Tw(\partial S(\gamma))$. Hence $[\varphi]$ belongs to finitely many left cosets modulo $G(\gamma)$. This proves that the index $[K:G(\gamma)]$ is finite.

    Finally we have:
    \[
        m(\gamma)=[\Stab([\gamma]):G(\gamma)]=[\Stab([\gamma]):K][K:G(\gamma)]<\infty.
    \]
\end{proof}

\begin{Cor}\label{S2:Cor:IndexRandVar}
    We have $F^\gamma=\frac{1}{m(\gamma)}\widetilde{F}_{G(\gamma)}$.
\end{Cor}

\begin{proof}
    Let $\mathcal{Y}\subset\MCG_{g,n}$ be a complete system of representatives of $\Stab([\gamma])\backslash\MCG_{g,n}$ and $\mathcal{X}\subset\Stab([\gamma])$ be a complete system of representatives of $G(\gamma)\backslash\Stab([\gamma])$. On the one hand we have a bijection:
    \[
        \begin{array}{ccc}
            \mathcal{X}\times\mathcal{Y} & \longrightarrow & G(\gamma)\backslash\MCG_{g,n} \\
            \big([\varphi],[\psi]\big) & \longmapsto & G(\gamma)[\varphi\circ\psi]
        \end{array},
    \]
    thus for every marked hyperbolic surface $(X,f)$ we have:
    \begin{align*}
        \widetilde{F}_{G(\gamma)}([X,f]) = & \sum_{[\varphi]\in\mathcal{X}}\sum_{[\psi]\in\mathcal{Y}}F\circ\ell_\gamma([\varphi\circ\psi]\cdot[X,f]) \\
        = & \sum_{[\varphi]\in\mathcal{X}}\sum_{[\psi]\in\mathcal{Y}}F\circ\ell_{\psi^{-1}\circ\varphi^{-1}\circ\gamma}([X,f]) \\
        = & m(\gamma)\sum_{[\psi]\in\mathcal{Y}}F\circ\ell_{\psi^{-1}\circ\gamma}([X,f]).
    \end{align*}
    On the other hand $\mathcal{Y}^{-1}=\{[\psi^{-1}]|[\psi]\in\mathcal{Y}\}$ is a complete system of representatives of $\MCG_{g,n}/\Stab([\gamma])$, which yields:
    \begin{align*}
        \widetilde{F}_{G(\gamma)}([X,f]) = & m(\gamma)\sum_{[\varphi]\in\MCG_{g,n}/\Stab([\gamma])}F\circ\ell_{\varphi\circ\gamma}([X,f]) \\
        = & m(\gamma)\sum_{\gamma'\in\mathcal{O}_\gamma}F\circ\ell_{\gamma'}([X,f]) \\
        = & m(\gamma)F^\gamma([X,f]).
    \end{align*}
\end{proof}

\subsection{Integration formula}

From now on, for convenience of notations, unless otherwise stated we shall use the bold font $\mathbf{x}_n$ to denote the vector $(x_1,\ldots,x_n)\in\mathbf{R}^n$, and $\d\mathbf{x}_n=\d x_1\cdots\d x_n$.

\begin{Theo}\label{S2:Theo:IntFormGeneral}
    Let $\gamma$ be a loop on $S_{g,n}$, $\beta_1,\ldots,\beta_M$ be the essential connected components of $\partial S(\gamma)$ and $2k$ be the dimension of the Teichmüller space of $S(\gamma)$. For each pair $(\beta_i,\beta_j)$ with $j\neq i$ such that $[\beta_i]=[\beta_j^{\pm1}]$, we only keep one curve so that after having possibly re-ordered the $\beta_i$'s, there exists a maximal $N\leq M$ such that $\Gamma=(\beta_1,\ldots,\beta_N)$ is a multi-curve. Complete $\Gamma$ into a pants decomposition of $S_{g,n}$ and denote by $h_\gamma$ the length function of $\gamma$ in the set of Fenchel-Nielsen coordinates associated to this pants decomposition. For every test function $F\colon\mathbf{R}\longrightarrow\mathbf{C}$, we have:
    \[
        \mathbb{E}[F^\gamma]=\frac{1}{m(\gamma)2^{M(\gamma)}}\int_{\mathbf{R}_{>0}^k\times\mathbf{R}^k\times\mathbf{R}_{>0}^N}F\circ h_\gamma(\mathbf{x}_k,\mathbf{t}_k,\mathbf{y}_N,\mathbf{L}_n)y_1\cdots y_NV_\Gamma(\mathbf{y}_N,\mathbf{L}_n)\frac{\d\mathbf{x}_k\d\mathbf{t}_k\d\mathbf{y}_N}{V_{g,n}(\mathbf{L}_n)},
    \]
    where $M(\gamma)$ is the number of connected components of the complement $S_{g,n}\backslash S(\gamma)$ homeomorphic to $S_{1,1}$, and $V_\Gamma(\mathbf{y}_N,\mathbf{L}_n)$ is the Weil-Petersson volume of the moduli space of $S_{g,n}\backslash S(\gamma)$, where the length of the boundary component(s)\footnotemark associated to the curve $\beta_i$ has (have) length $y_i$.
\end{Theo}

\footnotetext{If $\beta_i$ is homotopic to another boundary component of $S(\gamma)$, then one of the connected components of $S_{g,n}\backslash S(\gamma)$ is an annulus, and its two boundary components have length $y_i$. In this case the volume $V_\Gamma(\mathbf{y}_N,L)$ does not depend on $y_i$.}

\begin{proof}
    By using Corollary \ref{S2:Cor:IndexRandVar}, we have $\mathbb{E}[F^\gamma]=\frac{1}{m(\gamma)}\mathbb{E}[\widetilde{F}_{G(\gamma)}]$. Our goal will be to use (\ref{S2:Eq:IntQuot}) to compute $\mathbb{E}[\widetilde{F}_{G(\gamma)}]$.

    Let us identify the Teichmüller space $\mathcal{T}_{g,n}(\mathbf{L}_n)$ to $\mathbf{R}_{>0}^{3g-3+n}\times\mathbf{R}^{3g-3+n}$ by using Fenchel-Nielsen coordinates associated to a pants decomposition $\mathscr{P}$ of $S_{g,n}$ as follows: the $k$ first curves of $\mathscr{P}$ are a pants decomposition of $S(\gamma)$, the $N$ following curves are $\beta_1,\ldots,\beta_N$, and the $p=3g-3+n-k-N$ remaining curves are a pants decomposition of the complement $S_{g,n}\backslash S(\gamma)$.

    Given $\mathbf{y}_N\in\mathbf{R}_{>0}^N$, let us denote by $\mathcal{T}(S_{g,n}\backslash S(\gamma),\mathbf{y}_N)$ the Teichmüller space of $S_{g,n}\backslash S(\gamma)$ where for all $i\in\{1,\ldots,N\}$, the boundary component(s) corresponding to $\beta_i$ has (have) length $y_i$. Let $\mathcal{D}(\mathbf{y}_N)$ be a fundamental domain of the action of $\MCG(S_{g,n}\backslash S(\gamma))$ on $\mathcal{T}(S_{g,n}\backslash S(\gamma),\mathbf{y}_N)$. We assume that the numbering of the $\beta_i$'s is chosen so that the curves $\beta_1,\ldots,\beta_{M(\gamma)}$ disconnect a one-holed torus from $S_{g,n}$, and the $N-M(\gamma)$ other $\beta_i$'s do not. Since every hyperbolic metric on a one-holed torus $S_{1,1}$ carries an isometric involution, the set:
    \[
        \mathcal{D}=\left\{(\mathbf{x}_k,\mathbf{y}_N,\mathbf{z}_p,\mathbf{t}_k,\mathbf{s}_N,\mathbf{u}_p)\in\mathbf{R}_{>0}^{3g-3+n}\times\mathbf{R}^{3g-3+n}~\Big|~
            (\mathbf{z}_p,\mathbf{u}_p)\in\mathcal{D}(\mathbf{y}_N),
            \left\{\begin{array}{l}
                \forall j\leq M(\gamma),0\leq s_j< \frac{y_j}{2} \\
                \forall j> M(\gamma), 0\leq s_j<y_j
            \end{array}\right.\right\}
    \]
    is a fundamental domain of the action of $G(\gamma)$ on $\mathcal{T}_{g,n}(\mathbf{L}_n)$.

    Thus we have:
    \begin{align*}
        \mathbb{E}[\widetilde{F}_{G(\gamma)}] =  \int_{\mathbf{R}_{>0}^{3g-3+n}\times\mathbf{R}^{3g-3+n}}F\circ h_\gamma\Big( & \mathbf{x}_k,\mathbf{t}_k,\mathbf{y}_N,\mathbf{L}_n\Big)\mathds{1}_{\mathcal{D}(\mathbf{y}_N)}(\mathbf{z}_p,\mathbf{u}_p)\\
        \times &\prod_{j=1}^{M(\gamma)}\mathds{1}_{\{0\leq s_j<\frac{y_j}{2}\}}\prod_{j=M(\gamma)+1}^N\mathds{1}_{\{0\leq s_j<y_j\}}\frac{\d\mathbf{x}_k\d\mathbf{t}_k\d\mathbf{y}_N\d\mathbf{s}_N\d\mathbf{z}_p\d\mathbf{u}_p}{V_{g,n}(\mathbf{L}_n)} \\
        = \frac{1}{2^{M(\gamma)}}\int_{\mathbf{R}_{>0}^k\times\mathbf{R}^k\times\mathbf{R}_{>0}^N} F\circ h_\gamma & \Big(\mathbf{x}_k,\mathbf{t}_k,\mathbf{y}_N,\mathbf{L}_n\Big)y_1\cdots y_N\mu_\WP\Big(\mathcal{M}(S_{g,n}\backslash S(\gamma),\mathbf{y}_N)\Big)\frac{\d\mathbf{x}_k\d\mathbf{t}_k\d\mathbf{y}_N}{V_{g,n}(\mathbf{L}_n)},
    \end{align*}
    where $\mathcal{M}(S_{g,n}\backslash S(\gamma),\mathbf{y}_N)=\mathcal{T}(S_{g,n}\backslash S(\gamma),\mathbf{y}_N)/\MCG(S_{g,n}\backslash S(\gamma))$. We conclude by noticing that we have $V_\Gamma(\mathbf{y}_N,\mathbf{L}_n)=\mu_\WP(\mathcal{M}(S_{g,n}\backslash S(\gamma),\mathbf{y}_N))$.
\end{proof}

By disintegrating the Lebesgue measure on the level sets of $h_\gamma$, we can consider the expectation $\mathbb{E}[F^\gamma]$ as the integral of $F$ against a density function which is a linear combination of pseudo-convolutions of polynomials with respect to $h_\gamma$. We therefore need to have a better understanding of the behavior of the length function $h_\gamma$.

\section{Length function of a closed loop}

It follows from Theorem \ref{S2:Theo:IntFormGeneral} that the next key step to prove Theorem \ref{S1:Theo:MainIntro} is to have a good understanding of the behavior of length functions of closed loops in Fenchel-Nielsen coordinates. We first describe an algorithm revolving on results of Roger and Yang from \cite{RoYa14} to compute the length of a non-simple closed geodesic by ``opening'' its self-intersections. We will apply this algorithm to prove monotony and properness results on length functions of arbitrary loops and to compute the length functions of loops filling a pair of pants, and we shall explain why it does not make us able to compute the length function of more general loops. 

\subsection{A first approach: resolving a non-simple loop at a self-intersection point}

In order to describe the length of a non-simple loop $\gamma$ in Fenchel-Nielsen coordinates, we first try to open the self-intersections of $\gamma$. This process is called the \emph{resolution of $\gamma$ at its self-intersecting points}. We recall from \cite{RoYa14} the properties of the resolution process and its applications in hyperbolic geometry.

\begin{Def}[\cite{RoYa14}]\label{S3:Def:ResLoop}
    Let $\gamma$ be a closed loop self-intersecting transversally at a point $p$. The \emph{resolutions of $\gamma$ at $p$} are the loops obtained by opening the self-intersection at $p$ by going along $\gamma$ towards $p$ and turning left (respectively right) at $p$.
\end{Def}

Since the resolutions of a loop $\gamma$ at a self-intersecting point can be taken in an arbitrarily small neighborhood of $\gamma$, its resolutions are homotopic to loops on the surface $S(\gamma)$ filled by $\gamma$.

\begin{Lem}[\cite{RoYa14}]
    Let $\gamma$ be a closed loop self-intersecting transversally at $p$. One way of resolving the self-intersection at $p$ splits $\gamma$ into two loops $\alpha,\alpha'$ and the other way of resolving $\gamma$ gives a single loop $\beta$.

    The loops $\alpha,\alpha'$ are called the \emph{separating resolutions of $\gamma$ at $p$} and $\beta$ is called the \emph{non-separating resolution of $\gamma$ at $p$}.
\end{Lem}

\begin{figure}[H]
    \centering
    \includegraphics{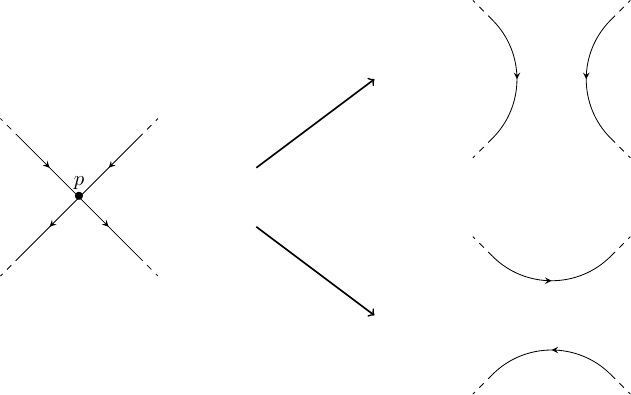}
    \caption{The resolutions of a loop at the self-intersection point $p$.}
    \label{S3:Fig:Resol}
\end{figure}

\begin{figure}[H]
    \centering
    \includegraphics{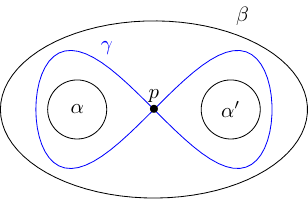}
    \caption{The resolutions of an eight-shaped loop. The curves $\alpha$ and $\alpha'$ are the separating resolutions and $\beta$ is the non-separating one.}
    \label{S3:Fig:ResEight}
\end{figure}

\begin{Def}[\cite{RoYa14}]
    Let $\gamma$ be a loop, $p$ be a self-intersection point of $\gamma$ and let $t_1,t_2\in(0,1)$, $t_1<t_2$, such that $\gamma(t_1)=\gamma(t_2)=p$. If the subloop $\gamma_{|[t_1,t_2]}$ is homotopically trivial, we call it a \emph{curl of $\gamma$} and $p$ is its \emph{vertex}.

    A \emph{Reidemeister move I on $\gamma$ at $p$} consists in removing a curl of $\gamma$ of vertex $p$.

    If $\gamma$ is homotopically non-trivial, the \emph{curling number of $\gamma$}, denoted by $c(\gamma)$ is the maximal number of Reidemeister moves I possible on $\gamma$. If $\gamma$ is homotopically trivial, $c(\gamma)$ is defined by adding 1 to the former definition.
\end{Def}

\begin{figure}[H]
    \centering
    \includegraphics{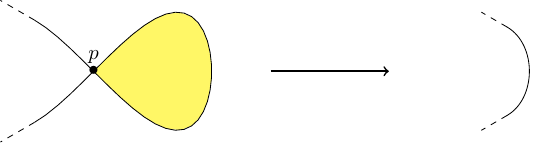}
    \caption{A Reidemeister move I removing a curl of vertex $p$.}
    \label{S3:Fig:ReideMove}
\end{figure}

The length of a non-simple geodesic on a hyperbolic surface is closely related to the length of its resolutions. We can therefore describe an algorithm consisting in successively resolving a closed geodesic at its self-intersecting points in order to find an expression of its length in terms of lengths of simple closed geodesics.

\begin{Lem}[\cite{RoYa14}]\label{S3:Lem:CurlNum}
    Let $\gamma$ be a closed geodesic on a hyperbolic surface.
    \begin{enumerate}[label=\textup{(\roman*)}]
        \item $c(\gamma)=0$.
        \item The curling number of the non-separating resolution of $\gamma$ at a self-intersection $p$ is at most 1, and it is 1 if and only if there exists another self-intersection $p'$ of $\gamma$ such that the domain bounded by $\gamma$ between $p$ and $p'$ is a contractible triangle, with $p$ as a double vertex. The only possible case for it to happen is shown on Figure \ref{S3:Fig:CurlIs1}.
    \end{enumerate}
\end{Lem}

\begin{figure}[H]
    \centering
    \includegraphics{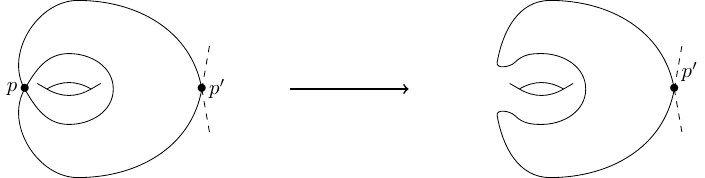}
    \caption{The only possible case for the curling number of the non-separating resolution of $\gamma$ at $p$ to be 1.}
    \label{S3:Fig:CurlIs1}
\end{figure}

We shall denote by $\cosh$ and $\sinh$ the hyperbolic cosinus and sinus instead of the usual notations $\textup{cosh}$ and $\textup{sinh}$. This slight gain of space will become convenient for our further calculations in Paragraph \ref{S3:Subsec:GenCase}.

\begin{Lem}[\cite{RoYa14}]\label{S3:Lem:LenResol}
    Let $\gamma$ be a non-simple closed geodesic on a hyperbolic surface $X$ and let us denote by $\alpha,\alpha'$ the separating resolutions of $\gamma$ at one of its self-intersecting points $p$ and $\beta$ its non-separating resolution at $p$.
    \begin{enumerate}[label=\textup{(\roman*)}]
        \item If $c(\beta)=0$, then:
        \[
            \cosh\Big(\frac{\ell_\gamma(X)}{2}\Big) = 2\cosh\Big(\frac{\ell_\alpha(X)}{2}\Big)\cosh\Big(\frac{\ell_{\alpha'}(X)}{2}\Big)+\cosh\Big(\frac{\ell_\beta(X)}{2}\Big).
        \]
        \item If $c(\beta)=1$, then:
        \[
            \cosh\Big(\frac{\ell_\gamma(X)}{2}\Big) = 2\cosh\Big(\frac{\ell_\alpha(X)}{2}\Big)\cosh\Big(\frac{\ell_{\alpha'}(X)}{2}\Big)-\cosh\Big(\frac{\ell_\beta(X)}{2}\Big).
        \]
    \end{enumerate}
\end{Lem}

By using iteratively Lemma \ref{S3:Lem:LenResol}, we can therefore obtain an expression of the length function of a loop which depends only on the length of simple (possibly non-disjoint) curves.

\begin{Lem}\label{S3:Lem:LenPolyPos}
    Let $\gamma$ be a non-simple closed geodesic on a hyperbolic surface $X$. There exist simple closed geodesics $\alpha_1,\ldots,\alpha_k$, integers $q_1,\ldots,q_k\in\mathbf{N}_{>0}$ and a polynomial $P\in\mathbf{R}[X_1,\ldots,X_k]$ with positive coefficients such that:
    \[
        \cosh\Big(\frac{\ell_\gamma(X)}{2}\Big)=P\bigg(\cosh\Big(\frac{q_1\ell_{\alpha_1}(X)}{2}\Big),\ldots,\cosh\Big(\frac{q_k\ell_{\alpha_k}(X)}{2}\Big)\bigg).
    \]
    Given a boundary $\delta$ of $S(\gamma)$, we can always choose $\alpha_1,\ldots,\alpha_k$ such that $\delta$ is homotopic to one of the $\alpha_i$'s.
    
    Moreover if $\varphi\colon X\longrightarrow X$ is an orientation preserving homeomorphism, we have:
    \[
        \cosh\Big(\frac{\ell_{\varphi\circ\gamma}(X)}{2}\Big)=P\bigg(\cosh\Big(\frac{q_1\ell_{\varphi\circ\alpha_1}(X)}{2}\Big),\ldots,\cosh\Big(\frac{q_k\ell_{\varphi\circ\alpha_k}(X)}{2}\Big)\bigg)
    \]
\end{Lem}

\begin{proof}
    We first prove the existence of the polynomial $P$ by induction on the number of self-intersections of $\gamma$. If $\gamma$ has one self-intersection then we can use Lemma \ref{S3:Lem:LenResol} to conclude. Let us now assume that $\gamma$ has $N\geq2$ self-intersections and that the conclusion is true for any closed geodesic with at most $N-1$ self-intersections.
    
    Let $p$ be a self-intersecting point of $\gamma$ and let us denote by $\alpha$ and $\alpha'$ its separating resolutions at $p$, and $\beta$ its non-separating resolution. Two scenarios can occur.
    
    If $c(\beta)=0$, then by using Lemma \ref{S3:Lem:LenResol}, since the geodesics homotopic to $\alpha,\alpha'$ and $\beta$ have at most $N-1$ self-intersections, we can conclude by induction.
    
    If $c(\beta)=1$, then according to Lemma \ref{S3:Lem:CurlNum}, $p$ is a double vertex of a contractible triangle bounded by $\gamma$. Let us denote by $p'$ its other vertex. Resolving $\gamma$ at $p'$ instead of $p$ makes us able to conclude.
    
    Let us now prove that, given a boundary $\delta$ of $S(\gamma)$, we can choose $\alpha_1,\ldots,\alpha_k$ such that we have up to homotopy $\delta\in\{\alpha_1,\ldots,\alpha_k\}$. By definition of $S(\gamma)$, the curve $\delta$ is a non-contractible boundary of a regular neighborhood $V$ of $\gamma$ which retracts onto $\gamma$. Let $p_1,\ldots,p_M$ be the self-intersecting points of $\gamma$ such that, when considering a small enough neighborhood $U_i\subset X$ of $p_i$, the intersection $U_i\cap\partial V$ is a disjoint union of four arcs $\widetilde{\delta}_1,\widetilde{\delta}_2,\widetilde{\delta}_3,\widetilde{\delta}_4$ and at least one of the $\widetilde{\delta}_j$'s is an arc of $\delta$.
    
    When resolving $\gamma$ at each $p_i$ consecutively, we obtain a curve homotopic to $\delta$.
    
    \begin{figure}[H]
        \centering
        \includegraphics{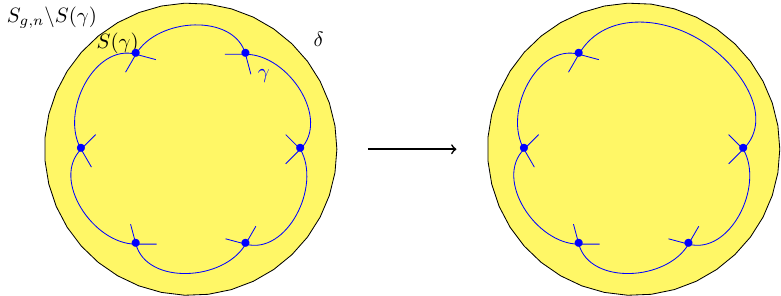}
        \caption{Resolutions of $\gamma$ at self-intersecting points near the curve $\delta$.}
    \end{figure}
    
    However it is possible that the non-separating resolution at one of the self-intersecting points $p_i$ has curling number 1. In this case we denote by $p_i'$ the simple vertex of the contractible triangle bounded by $\gamma$ with double vertex $p_i$. Let $\alpha_i,\alpha_i'$ be the separating resolutions of $\gamma$ at $p_i$ and $\beta_i$ be the non-separating resolution (with $c(\beta_i)=1$). Then one of $\alpha_i,\alpha_i'$ is simple, say $\alpha_i'$, and $p_i'$ is a self-intersection of $\alpha_i$. The separating resolutions of $\alpha_i$ at $p_i'$ are homotopic to $\alpha_i'$ and $\beta_i$, let us denote by $\beta_i'$ the non-separating resolution of $\alpha_i$ at $p_i'$ (see Figure \ref{S3:Lem:LenPolyPos:DemoFig}). Since $\delta$ is non-contractible, we have $c(\beta_i')=0$, and using Lemma \ref{S3:Lem:LenResol} we have:
    \begin{align*}
        \cosh\Big(\frac{\ell_\gamma(X)}{2}\Big) = & 2\cosh\Big(\frac{\ell_{\alpha_i'}(X)}{2}\Big)\cosh\Big(\frac{\ell_{\alpha_i}(X)}{2}\Big)-\cosh\Big(\frac{\ell_{\beta_i}(X)}{2}\Big) \\
        = & 2\cosh\Big(\frac{\ell_{\alpha_i'}(X)}{2}\Big)\bigg[2\cosh\Big(\frac{\ell_{\alpha_i'}(X)}{2}\Big)\cosh\Big(\frac{\ell_{\beta_i}(X)}{2}\Big)+\cosh\Big(\frac{\ell_{\beta_i'}(X)}{2}\Big)\bigg]-\cosh\Big(\frac{\ell_{\beta_i}(X)}{2}\Big) \\
        = & \cosh\Big(\frac{\ell_{\beta_i}(X)}{2}\Big)\bigg[4\cosh^2\Big(\frac{\ell_{\alpha_i'}(X)}{2}\Big)-1\bigg]+2\cosh\Big(\frac{\ell_{\alpha_i'}(X)}{2}\Big)\cosh\Big(\frac{\ell_{\beta_i'}(X)}{2}\Big) \\
        = & \cosh\Big(\frac{\ell_{\beta_i}(X)}{2}\Big)\Big[2\cosh\big(\ell_{\alpha_i'}(X)\big)+1\Big]+2\cosh\Big(\frac{\ell_{\alpha_i'}(X)}{2}\Big)\cosh\Big(\frac{\ell_{\beta_i'}(X)}{2}\Big)
    \end{align*}
    Hence by induction we can still choose the polynomial $P$ with positive coefficients.
    
    \begin{figure}[H]
        \centering
        \includegraphics{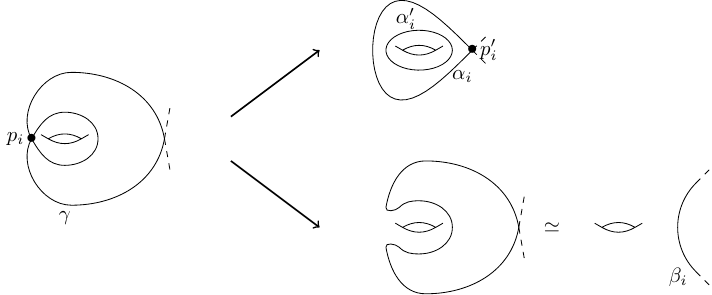}
        \caption{Resolutions at a point where the non-separating resolution has curling number $1$.}
        \label{S3:Lem:LenPolyPos:DemoFig}
    \end{figure}
\end{proof}

\begin{Rem}
    In general the curves $\alpha_1,\ldots,\alpha_k$ are not disjoint.
\end{Rem}

\begin{figure}[H]
    \centering
    \includegraphics{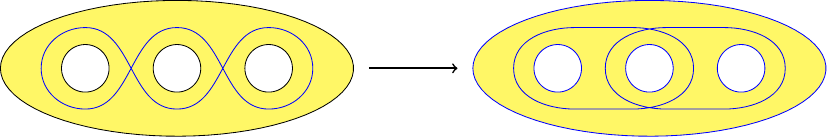}
    \caption{Non-disjoint simple resolutions of a loop filling a four-holed sphere.}
\end{figure}

A first consequence of Lemma \ref{S3:Lem:LenPolyPos} is that we can deduce from it a monotony result for the length functions of loops with respect to the lengths of the boundaries of $S_{g,n}$. Indeed, while proving the main result of \cite{Par05}, Parlier also proved the following monotony property.

\begin{Theo}[\cite{Par05}]\label{S3:Theo:LenIncBound}
    Let $\mathscr{P}$ a pants decomposition of $S_{g,n}$ and $\gamma$ be a simple closed curve on $S_{g,n}$. We denote by $\ell_\gamma^{\mathbf{L}_n}\colon\mathcal{T}_{g,n}(\mathbf{L}_n)\longrightarrow\mathbf{R}_{\geq0}$ its length function on $\mathcal{T}_{g,n}(\mathbf{L}_n)$ and $(\mathbf{x}_{3g-3+n},\mathbf{t}_{3g-3+n})\mapsto h_\gamma(\mathbf{x}_{3g-3+n},\mathbf{t}_{3g-3+n},\mathbf{L}_n)$ its expression in Fenchel-Nielsen coordinates defined along $\mathscr{P}$. Then $h_\gamma$ is an increasing function of its $n$ last variables, i.e. for all $\mathbf{y}_n\in\mathbf{R}_{\geq0}^n$ and $(\mathbf{x}_{3g-3+n},\mathbf{t}_{3g-3+n},\mathbf{L}_n)\in\mathbf{R}_{>0}^{3g-3+n}\times\mathbf{R}^{3g-3+n}\times\mathbf{R}_{\geq0}^n$, we have:
    \[
        h_\gamma(\mathbf{x}_{3g-3+n},\mathbf{t}_{3g-3+n},\mathbf{L}_n)\leq h_\gamma(\mathbf{x}_{3g-3+n},\mathbf{t}_{3g-3+n},\mathbf{L}_n+\mathbf{y}_n).
    \]
\end{Theo}

Using Lemma \ref{S3:Lem:LenPolyPos}, we can generalize Theorem \ref{S3:Theo:LenIncBound}.

\begin{Cor}\label{S3:Cor:LenIncBound}
    The conclusion of Theorem \ref{S3:Theo:LenIncBound} holds for any closed loop $\gamma$ on $S_{g,n}$.
\end{Cor}

\begin{proof}
    Using Lemma \ref{S3:Lem:LenPolyPos}, we can write $\ell_\gamma$ as an increasing function of the lengths of simple closed curves. The conclusion immediately follows from Theorem \ref{S3:Theo:LenIncBound}.
\end{proof}

The following result is also a direct consequence to Theorem \ref{S3:Theo:LenIncBound} and Lemma \ref{S3:Lem:LenPolyPos}.

\begin{Cor}\label{S3:Cor:LenProperBound}
    Let $\gamma$ be a closed loop on $S_{g,n}$ and let $\delta$ be a connected component of $\partial S(\gamma)$. We have $\ell_\gamma\to\infty$ as $\ell_\delta\to\infty$.
\end{Cor}

\begin{proof}
    Following Lemma \ref{S3:Lem:LenPolyPos}, we choose simple curves $\alpha_1,\ldots,\alpha_k$ such that $\delta$ is one of the $\alpha_i$'s, integers $q_1,\ldots,q_k$ and a polynomial $P\in\mathbf{R}[X_1,\ldots,X_k]$ with positive coefficients such that:
    \[
        \cosh\Big(\frac{\ell_\gamma(X)}{2}\Big)=P\bigg(\cosh\Big(\frac{q_1\ell_{\alpha_1}(X)}{2}\Big),\ldots,\cosh\Big(\frac{q_k\ell_{\alpha_k}(X)}{2}\Big)\bigg).
    \]
    According to Theorem \ref{S3:Theo:LenIncBound}, each $\ell_{\alpha_i}$ is an increasing function of $\ell_\delta$, hence we have $\ell_\gamma\to\infty$ as $\ell_\delta\to\infty$.
\end{proof}

In addition to this fact, the length function of a filling loop is proper on the Teichmüller space.

\begin{Lem}
    Let $\gamma$ be a filling loop on $S_{g,n}$. Then $\ell_\gamma\colon\mathcal{T}_{g,n}(\mathbf{L}_n)\longrightarrow\mathbf{R}$ is a proper map.
\end{Lem}

\begin{proof}
    Kerckhoff proved in \cite{Ker83} that if $\Gamma=(\gamma_1,\ldots,\gamma_k)$ is a multi-loop made of non-disjoint curves on a compact surface $S_g$ such that $S_g\backslash\Gamma$ is a union of discs, then $\ell_{\gamma_1}+\cdots+\ell_{\gamma_k}$ is a proper map. We follow his arguments to prove our result.
    
    It is sufficient to prove that the length function of every essential curve $\alpha$ is bounded on the set $\{\ell_\gamma\leq K\}\subset\mathcal{T}_{g,n}(\mathbf{L}_n)$, for every $K\geq0$. If $\alpha$ is an essential curve on $S_{g,n}$, then by definition of the surface filled by $\gamma$, the curve $\alpha$ is homotopic to a loop contained in an arbitrarily small regular neighborhood of $\gamma$, hence it is homotopic to a loop $\widetilde{\alpha}$ made of a minimal amount of arcs of $\gamma$. Denote by $N_\alpha$ the maximal multiplicity of a point of $\widetilde{\alpha}$. We have $\ell_\alpha=\ell_{\widetilde{\alpha}}\leq N_\alpha\ell_\gamma$, hence $\ell_\alpha$ is bounded on every set of the form $\{\ell_\gamma\leq K\}$.
\end{proof}

\subsection{Consequence for loops filling a pair of pants and limitations for more complex loops}

We immediately deduce the following result from Lemma \ref{S3:Lem:LenPolyPos}.

\begin{Prop}\label{S3:Prop:LenPantsFill}
    Let $\gamma$ be a loop on $S_{g,n}$ such that $S(\gamma)$ is a pair of pants. We denote by $\alpha_1,\alpha_2,\alpha_3$ the curves which bound $S(\gamma)$. There exist a finite set $G$ of linear forms on $\mathbf{R}^3$ and $\{a_T\}_{T\in G}\subset\mathbf{R}_{>0}$ such that:
    \[
        \ell_\gamma=2\argch\bigg[\sum_{T\in G}a_Te^{T(\ell_{\alpha_1},\ell_{\alpha_2},\ell_{\alpha_3})}\bigg].
    \]
\end{Prop}

\begin{Rem}
    The fact that the coefficients $a_T$'s are positive is a crucial point for our future computations.
\end{Rem}

The length function of a loop filling a pair of pants is our first encounter of a \emph{length-type function}. This class of functions will be defined and studied in Paragraph \ref{S3:Subsec:LenType}, and gathers the useful properties that will make us able to prove our main result. 

Mirzakhani's method from \cite{Mir16} to prove Theorem \ref{S1:Theo:MirzAsympt} consists in using Lemma \ref{S3:Lem:LenPolyPos} and making a finite number of change of Fenchel-Nielsen variables in once-holed tori and four-holed spheres, using the formul{\ae} provided by Okai in \cite{Oka93}.

\begin{Theo}[Change of Fenchel-Nielsen coordinates in a once-holed torus, \cite{Oka93}]\label{S3:Theo:ChangeFNTor}
    Let $a$ and $b$ be two simple closed curves on $S_{1,1}$ such that $a$ and $b$ intersect once, $[a,b]$ is homotopic to $\partial S_{1,1}$ and $\langle a,b\rangle = \pi_1(S_{1,1})$. Let $(\ell_a,\tau_a)$ and $(\ell_b,\tau_b)$ be the Fenchel-Nielsen coordinates defined along $a$ and $b$ respectively, with the twists $\tau_a$ and $\tau_b$ be determined so that $\tau_a=\tau_b=0$ if and only if $a$ and $b$ intersect orthogonally. If $\ell_{\partial S_{1,1}}=L$, we have:
    \[
        \cosh\Big(\frac{\ell_b}{2}\Big) = \frac{\cosh\big(\frac{\tau_a}{2}\big)}{\sinh\big(\frac{\ell_a}{2}\big)}\sqrt{\frac{\cosh(\ell_a)+\cosh\big(\frac{L}{2})}{2}},
    \]
    and $\tau_b$ is determined by the sign convention $\tau_a\tau_b\leq0$ and:
    \[
        \cosh\Big(\frac{\tau_b}{2}\Big) = \cosh\Big(\frac{\ell_a}{2}\Big)\sqrt{\frac{\cosh(\frac{\tau_a}{2})^2(\cosh(\ell_a)+\cosh(\frac{L}{2}))-2\sinh(\frac{\ell_a}{2})^2}{\cosh(\frac{\tau_a}{2})^2(\cosh(\frac{\ell_a}{2})^2+\cosh(\frac{L}{2}))+\sinh(\frac{\ell_a}{2})^2\cosh(\frac{L}{2})}}.
    \]
\end{Theo}

\begin{figure}[H]
    \centering
    \includegraphics{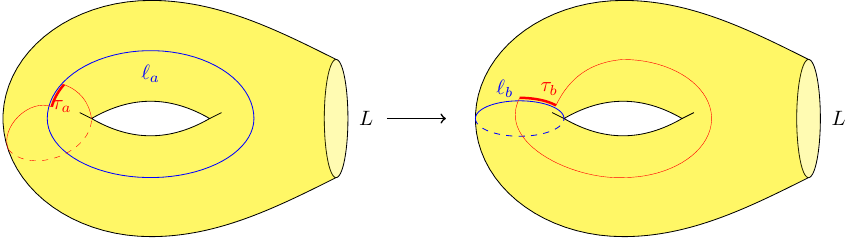}
    \caption{Change of Fenchel-Nielsen coordinates in a once-holed torus.}
\end{figure}

Using this method, Mirzakhani obtained appropriate $\mathscr{C}^0$ estimates on the length functions of arbitrary closed loops. However our method requires $\mathscr{C}^1$ estimates of these functions. Unfortunately we cannot control the simple curves appearing in Lemma \ref{S3:Lem:LenPolyPos}, hence we could have to compose several different changes of coordinates provided by Theorem \ref{S3:Theo:ChangeFNTor}. This, in addition to the fact that obtaining appropriate $\mathscr{C}^1$ estimates out of Theorem \ref{S3:Theo:ChangeFNTor} brings heavy computations (and that the change of coordinates is even more complicated in four-holed spheres), motivated us to choose another approach for obtaining a usable expression of the length function of a closed loop.

\subsection{The general case}\label{S3:Subsec:GenCase}

Our strategy to compute the expression of the length function of an arbitrary loop $\gamma$ will be to represent the homotopy class of $\gamma$ by a right-angled piecewise geodesic loop. This way we will be able to find an appropriate matrix in $\SL_2(\mathbf{R})$ corresponding to the deck transformation associated to $\gamma$. Using the trace formula:
\[
    \Tr(AB)=\Tr(A)\Tr(B)-\Tr(A^{-1}B),
\]
valid for $A,B\in\SL_2(\mathbf{R})$, we will be able to compute an appropriate expression of $\ell_\gamma$. We will rely on the following result to compute a representative of the homotopy class of $\gamma$ in $\SL_2(\mathbf{R})$.

\begin{Lem}[\cite{Oka93}]\label{S3:Lem:HoloRepGeneral1}
    Let $\gamma$ be a piecewise geodesic closed loop on a hyperbolic surface starting at a point $p_0$ on a geodesic $\delta$. We lift $p_0$ to $i\in\mathbf{H}=\{z\in\mathbf{C}|\Im(z)>0\}$, $\delta$ to the imaginary axis and $\gamma$ to a curve $\widetilde{\gamma}$. Suppose that $\widetilde{\gamma}$ is the following curve: make an angle $\theta_1$ then go straight (i.e. geodesically) with length $\ell_1$, turn with an angle $\theta_2$, go straight with length $\ell_2$, $\ldots$, turn with angle $\theta_r$ and go straight with length $\ell_r$ to reach another lift of $p_0$ with an angle $\varphi$ with another lift of $\delta$. The deck transformation corresponding to $\gamma$ is given by the matrix:
    \[
        R(\theta_1)a(\ell_1)R(\theta_2)a(\ell_2)\cdots R(\theta_r)a(\ell_r)R(\varphi),
    \]
    where:
    \[
        R(\theta)=\left(\begin{array}{cc}
            \cos(\frac{\theta}{2}) & \sin(\frac{\theta}{2}) \\
            -\sin(\frac{\theta}{2}) & \cos(\frac{\theta}{2})
        \end{array}\right),~a(\ell)=\left(\begin{array}{cc}
            e^{\frac{\ell}{2}} & 0 \\
            0 & e^{-\frac{\ell}{2}}
        \end{array}\right).
    \]
\end{Lem}

\begin{Def}
    Let $\gamma$ be a closed loop on $S_{g,n}$ and let $P\subset S(\gamma)$ be a pair of pants. An \emph{incursion of $\gamma$ in $P$} is a connected (possibly non-closed) subloop of $\gamma([0,1])\cap P$.
\end{Def}

\begin{Rem}
    Given a pants decomposition $\mathscr{P}$ of $S(\gamma)$, the loop $\gamma$ can be recovered by concatenating its incursions in the pairs of pants associated to $\mathscr{P}$.
\end{Rem}

Let $\gamma$ be a closed loop on $S_{g,n}$, $\mathscr{P}$ a pants decomposition of $S(\gamma)$ and $(X,f)\in\mathcal{T}_{g,n}(\mathbf{L}_n)$. We represent the homotopy class of $f\circ\gamma$ by the piecewise geodesic loop $\widetilde{\gamma}$ defined the following way.

\begin{figure}[H]
    \centering
    \includegraphics[width=8cm]{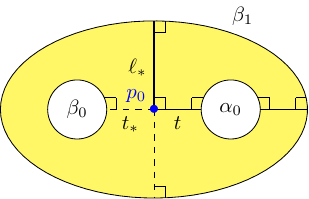}
    \caption{Definitions of the lengths in the pair of pants bounded by $\alpha_0$, $\beta_0$ and $\beta_1$ with $\ell_{\alpha_0}=x$, $\ell_{\beta_0}=y$ and $\ell_{\beta_1}=z$.}
    \label{S3:Fig:Incursion}
\end{figure}

Let $P\subset S(\gamma)$ be a pair of pants in the complement of $\mathscr{P}$ and $\widetilde{\gamma}_P$ be an incursion of $\widetilde{\gamma}$ in $P$. Let $P_0$ be the pair of pants from which $\widetilde{\gamma}_P$ enters in $P$ and $P_1$ be the pair of pants in which $\widetilde{\gamma}_P$ enters after exiting $P$. We note $\alpha_0,\alpha_1\in\mathscr{P}$ the boundaries of $P$ such that the starting point of $\widetilde{\gamma}_P$ is on $\alpha_0$ and the ending point of $\widetilde{\gamma}_P$ is on $\alpha_1$. If $\alpha_0\neq\alpha_1$, let $\beta_0=\alpha_1$, and if $\alpha_0=\alpha_1$ let $\beta_0$ be any connected component of $\partial P$ different from $\alpha_0$. Let $\beta_1$ be the last connected component of $\partial P$. We denote by $T$ the common perpendicular geodesic arc between $\alpha_0$ and $\beta_0$, and $D$ the geodesic arc perpendicular to $\beta_1$ at its extremities. The geodesic arcs $D$ and $T$ intersect orthogonally at a point $p_0$, which is the middle of $D$. Let $\ell_*=\ell_*(\ell_{\alpha_0},\ell_{\beta_0},\ell_{\beta_1})$ be the distance between $p_0$ and $\beta_1$ (which is the half length of $D$), $t=t(\ell_{\alpha_0},\ell_{\beta_0},\ell_{\beta_1})$ be the distance between $p_0$ and $\alpha_0$ and $t_*=t_*(\ell_{\alpha_0},\ell_{\beta_0},\ell_{\beta_1})$ be the distance between $p_0$ and $\beta_0$ (hence the length of $T$ is $t+t_*$). Let $b$ be the piecewise geodesic arc starting from $p_0$, following $T$ with length $t_*$ until reaching $\beta_0$, then turning with angle $\frac{\pi}{2}$, going straight with length $\ell_{\beta_0}$, turning with angle $\frac{\pi}{2}$, going straight with length $t_*$ and finally turning with angle $\pi$. We define $a$ similarly: start at $p_0$ oriented towards $\beta_0$, turn with angle $\pi$, go straight with length $t$ until reaching $\alpha_0$, turn with angle $\frac{\pi}{2}$, go straight with length $\ell_{\alpha_0}$, turn with angle $\frac{\pi}{2}$ and go straight with length $t$.

\begin{Lem}
    With the same notations as above, $a$ and $b$ are freely homotopic to $\alpha_0$ and $\beta_0$ respectively and the fundamental group $\pi_1(P,p_0)$ is generated by the homotopy classes of $a$ and $b$ based at $p_0$.
\end{Lem}

\begin{Rem}
    Starting at $p_0$ and following any word in the letters $a$ and $b$, the corresponding piecewise geodesic path is oriented towards $\beta_0$.
\end{Rem}

By gliding along $\alpha_0$ the ending point of the previous incursion of $\widetilde{\gamma}$ in $P_0$, we can assume without loss of generality that the starting point of $\widetilde{\gamma}_P$ is $\alpha_0\cap T$. Then $\widetilde{\gamma}_P$ follows $T$ with length $t$ until reaching $p_0$. Next $\widetilde{\gamma}_P$ makes a loop (which may be trivial) on $P$, wich can be represented by a word $w(a,b)$ in the letters $a$ and $b$ defined above. By looping enough times around $\alpha_0$ before going inside $P$, we can assume that the word $w(a,b)$ begins with the letter $b$. After having followed the word $w(a,b)$, the arc $\widetilde{\gamma}_P$ follows the following path.

\begin{enumerate}[label=\textup{(\roman*)}]
    \item\label{S3:GeodPath:Item1} If $\alpha_0\neq\alpha_1$, i.e. if the endpoint of $\widetilde{\gamma}_P$ is on $\beta_0$, then $\widetilde{\gamma}_P$ goes straight with length $t_*$ to reach $\beta_0$ and turns with angle $\frac{\pi}{2}$. By looping enough times around $\beta_0$ at this step, we can assume that the word $w(a,b)$ ends with the letter $a$. Next go straight until reaching the starting point of the next incursion of $\widetilde{\gamma}$ in $P_1$. Depending on the choice of the pair of common perpendiculars used to define the twist $\tau_{\beta_0}$, this corresponds to going straight with length $\pm\tau_{\beta_0}+\frac{m}{2}\ell_{\beta_0}$, for some $m\in\mathbf{Z}$. Finally turn with angle $-\frac{\pi}{2}$.
    \item\label{S3:GeodPath:Item2} If $\alpha_0=\alpha_1$, i.e. if the endpoint of $\widetilde{\gamma}_P$ is on $\alpha_0$, then $\widetilde{\gamma}_P$ turns with angle $\pi$, goes straight with length $t$ to reach $\alpha_0$ and turns with angle $\frac{\pi}{2}$. By looping enough times around $\alpha_0$ at this step, we can assume that the word $w(a,b)$ ends with the letter $b$. Next go straight until reaching the starting point of the next incursion of $\widetilde{\gamma}$ in $P_1$. Similarly to the previous case, this corresponds to going straight with length $\pm\tau_{\alpha_0}+\frac{m}{2}\ell_{\alpha_0}$ for some $m\in\mathbf{Z}$. Finally turn with angle $-\frac{\pi}{2}$.
\end{enumerate}

By repeating this process on every incursion of $\gamma$ inside the pairs of pants of $S(\gamma)\backslash\mathscr{P}$, we are able to represent the free homotopy class of $\gamma$ by the right-angled piecewise geodesic loop $\widetilde{\gamma}$. Using Lemma \ref{S3:Lem:HoloRepGeneral1}, we obtain the following.

\begin{Lem}\label{S3:Lem:HoloRepGeneral2}
    With the same notations as above, the deck transformation corresponding to $\widetilde{\gamma}$ is given by a product of the following matrices:
    \begin{equation}\label{S3:Lem:HoloRepGeneral2:Eq1}
        a(t(\ell_{\alpha_0},\ell_{\beta_0},\ell_{\beta_1}))B^{p_1}A^{q_1}\cdots B^{p_r}A^{q_r}a(t_*(\ell_{\alpha_0},\ell_{\beta_0},\ell_{\beta_1}))R\Big(\frac{\pi}{2}\Big)a\Big(\pm\tau_{\beta_0}+\frac{m}{2}\ell_{\beta_0}\Big)R\Big(-\frac{\pi}{2}\Big),
    \end{equation}
    and:
    \begin{equation}\label{S3:Lem:HoloRepGeneral2:Eq2}
        a(t(\ell_{\alpha_0},\ell_{\beta_0},\ell_{\beta_1}))B^{p_1}A^{q_1}\cdots B^{p_r}A^{q_r}B^{p_{r+1}}R(\pi)a(t(\ell_{\alpha_0},\ell_{\beta_0},\ell_{\beta_1}))R\Big(\frac{\pi}{2}\Big)a\Big(\pm\tau_{\alpha_0}+\frac{m}{2}\ell_{\alpha_0}\Big)R\Big(-\frac{\pi}{2}\Big),
    \end{equation}
    where $p_1,q_1,\ldots,p_r,q_r,p_{r+1}\in\mathbf{Z}$, $\alpha_0,\beta_0,\beta_1$ are either elements of $\mathscr{P}$ or connected components of $\partial S(\gamma)$, and:
    \begin{align*}
        A = & \left(\begin{array}{cc}
            \cosh(\frac{\ell_{\alpha_0}}{2}) & e^{-t(\ell_{\alpha_0},\ell_{\beta_0},\ell_{\beta_1})}\sinh(\frac{\ell_{\alpha_0}}{2}) \\
            e^{t(\ell_{\alpha_0},\ell_{\beta_0},\ell_{\beta_1})}\sinh(\frac{\ell_{\alpha_0}}{2}) & \cosh(\frac{\ell_{\alpha_0}}{2})
        \end{array}\right) \\
        B = & \left(\begin{array}{cc}
            \cosh(\frac{\ell_{\beta_0}}{2}) & -e^{t_*(\ell_{\alpha_0},\ell_{\beta_0},\ell_{\beta_1})}\sinh(\frac{\ell_{\beta_0}}{2}) \\
            -e^{-t_*(\ell_{\alpha_0},\ell_{\beta_0},\ell_{\beta_1})}\sinh(\frac{\ell_{\beta_0}}{2}) & \cosh(\frac{\ell_{\beta_0}}{2})
        \end{array}\right).
    \end{align*}
\end{Lem}

\begin{proof}
    The expression (\ref{S3:Lem:HoloRepGeneral2:Eq1}) comes from applying Lemma \ref{S3:Lem:HoloRepGeneral1} to the case \ref{S3:GeodPath:Item1}, and (\ref{S3:Lem:HoloRepGeneral2:Eq2}) comes from applying the same result to the case \ref{S3:GeodPath:Item2}. The matrices $A,B\in\SL_2(\mathbf{R})$ are the matricial representations of following the piecewise geodesic loops $a$ and $b$ respectively inside each pair of pants of $S(\gamma)\backslash\mathscr{P}$.
\end{proof}

For technical reason we will introduce the following matrices:
\[
    S_a^n(x,y,z) = \varepsilon_n\left(\begin{array}{cc}
        \sinh(t(x,y,z))\sinh(\frac{nx}{2}) & \cosh(t(x,y,z))\sinh(\frac{nx}{2}) + \cosh(\frac{nx}{2}) \\
        \cosh(t(x,y,z))\sinh(\frac{nx}{2}) - \cosh(\frac{nx}{2}) & \sinh(t(x,y,z))\sinh(\frac{nx}{2})
    \end{array}\right),
\]
where $\varepsilon_n$ is the sign of $n$, and:
\[
    S_b^n(x,y,z) = \varepsilon_n\left(\begin{array}{cc}
        \sinh(t_*(x,y,z))\sinh(\frac{ny}{2}) & \cosh(t_*(x,y,z))\sinh(\frac{ny}{2}) + \cosh(\frac{ny}{2}) \\
        \cosh(t_*(x,y,z))\sinh(\frac{ny}{2}) - \cosh(\frac{ny}{2}) & \sinh(t_*(x,y,z))\sinh(\frac{ny}{2})
    \end{array}\right).
\]
By a direct computation, we have:
\begin{align*}
    S_a^n(\ell_{\alpha_0},\ell_{\beta_0},\ell_{\beta_1}) = &  \varepsilon_nR\Big(\frac{\pi}{2}\Big)A^nR\Big(\frac{\pi}{2}\Big) \\
    S_b^n(\ell_{\alpha_0},\ell_{\beta_0},\ell_{\beta_1})
    = & \varepsilon_nR\Big(-\frac{\pi}{2}\Big)B^nR\Big(-\frac{\pi}{2}\Big).
\end{align*}
Let us also denote:
\[
    w(\ell)=R\Big(-\frac{\pi}{2}\Big)a(\ell)R\Big(\frac{\pi}{2}\Big) =\left(\begin{array}{cc}
        \cosh(\frac{\ell}{2}) & \sinh(\frac{\ell}{2}) \\
        \sinh(\frac{\ell}{2}) & \cosh(\frac{\ell}{2})
    \end{array}\right).
\]
The length function of $\gamma$ is given by the (absolute value of the) trace of a product of matrices of the form (\ref{S3:Lem:HoloRepGeneral2:Eq1}) and (\ref{S3:Lem:HoloRepGeneral2:Eq2}). Since the trace is invariant under cyclic permutation, we can assume without loss of generality that words of the form (\ref{S3:Lem:HoloRepGeneral2:Eq1}) and (\ref{S3:Lem:HoloRepGeneral2:Eq2}) begin with the matrix $R(-\frac{\pi}{2})$ instead of ending with it. With this new expression, words of the form (\ref{S3:Lem:HoloRepGeneral2:Eq1}) can be represented in $\PSL_2(\mathbf{R})$ by:
\begin{equation}\label{S3:Eq:HoloRepGen1}
    w(t(\ell_{\alpha_0},\ell_{\beta_0},\ell_{\beta_1}))\bigg(\prod_{i=1}^rS_b^{p_i}(\ell_{\alpha_0},\ell_{\beta_0},\ell_{\beta_1})S_a^{q_i}(\ell_{\alpha_0},\ell_{\beta_0},\ell_{\beta_1})\bigg)w(t_*(\ell_{\alpha_0},\ell_{\beta_0},\ell_{\beta_1}))a\Big(\pm\tau_{\beta_0}+\frac{m}{2}\ell_{\beta_0}\Big),
\end{equation}
and words of the form (\ref{S3:Lem:HoloRepGeneral2:Eq2}) can be represented by:
\begin{align}
    \nonumber w(t(\ell_{\alpha_0},\ell_{\beta_0},\ell_{\beta_1}))\bigg(\prod_{i=1}^rS_b^{p_i}(\ell_{\alpha_0}, & \ell_{\beta_0},\ell_{\beta_1})S_a^{q_i}(\ell_{\alpha_0},\ell_{\beta_0},\ell_{\beta_1})\bigg) \\
    \label{S3:Eq:HoloRepGen2} & \times S_b^{p_{r+1}}(\ell_{\alpha_0},\ell_{\beta_0},\ell_{\beta_1})w(t(\ell_{\alpha_0},\ell_{\beta_0},\ell_{\beta_1}))a\Big(\pm\tau_{\alpha_0}+\frac{m}{2}\ell_{\alpha_0}\Big).
\end{align}

\begin{Lem}\label{S3:Lem:MatrixCoeffPos1}
    With the same notations as above, if $n\neq0$ and $x,y>0,z\geq0$, the entries of the matrix $S_a^n(x,y,z)\in\SL_2(\mathbf{R})$ are non-negative.
\end{Lem}

\begin{proof}
    We have:
    \[
        S_a^n(x,y,z) = \left(\begin{array}{cc}
            \sinh(t)\sinh(\frac{|n|x}{2}) & \cosh(t)\sinh(\frac{|n|x}{2})+\varepsilon_n\cosh(\frac{nx}{2}) \\
            \cosh(t)\sinh(\frac{|n|x}{2})-\varepsilon_n\cosh(\frac{nx}{2}) & \sinh(t)\sinh(\frac{|n|x}{2})
        \end{array}\right),
    \]
    hence the diagonal entries of $S_a^n(x,y,z)$ are clearly non-negative.
    
    Using hyperbolic trigonometric formul{\ae} in right-angled pentagons and the notations introduced above (see Figure \ref{S3:Fig:Incursion}), we have:
    \[
        \cosh(t)=\coth\Big(\frac{x}{2}\Big)\coth(\ell_*),
    \]
    yielding:
    \begin{align*}
        \cosh(t)\sinh\Big(\frac{nx}{2}\Big)\pm\cosh\Big(\frac{nx}{2}\Big) = & \pm\cosh\Big(\frac{nx}{2}\Big)+\coth(\ell_*)\frac{\sinh(\frac{nx}{2})\cosh(\frac{x}{2})}{\sinh(\frac{x}{2})} \\
        = & \cosh\Big(\frac{nx}{2}\Big)(\coth(\ell_*)\pm1)+\coth(\ell_*)\frac{\sinh(\frac{(n-1)x}{2})}{\sinh(\frac{x}{2})},
    \end{align*}
    which is clearly positive for $n>0$. The matrix $S_a^n(x,y,z)$ therefore has non-negative entries.
\end{proof}

For convenience of notations we will write $S_a^n$ and $S_b^n$ instead of $S_a^n(x,y,z)$ and $S_b^n(x,y,z)$. We have:
\begin{align}
    \label{S3:Eq:MatSa} & \varepsilon_nw(t)a(nx)Ww(t) \\
    \label{S3:Eq:MatSb} & \varepsilon_nw(t_*)a(ny)Ww(t_*).
\end{align}
Let:
\[
    W=\left(\begin{array}{cc}
        0 & 1 \\
        -1 & 0
    \end{array}\right).
\]
\begin{Rem}
    We have $W^2=-I_2$ and if $A\in\SL_2(\mathbf{R})$:
    \[
        WAW^{-1}=\,^tA^{-1}.
    \]
    It follows that:
    \begin{align}
        \label{S3:Eq:MatWa} Wa(\theta) & = a(-\theta)W \\
        \label{S3:Eq:MatWw} Ww(\ell) & = w(-\ell)W,
    \end{align}
    and using (\ref{S3:Eq:MatSa}) the fact that $W^2=-I_2$, we deduce that:
    \begin{equation}\label{S3:Eq:InvMatS}
        (S_a^n)^{-1}=WS_a^{-n}W^{-1}.
    \end{equation}
    Since $S_b^n(x,y,z)=S_a^n(y,x,z)$, we deduce that $(S_b^n)^{-1}=WS_b^{-n}W^{-1}$.
\end{Rem}

\begin{Lem}\label{S3:Lem:TraceHoloRep1}
    With the same notations as above, there exists $k\in\mathbf{N}$ such that $2\cosh(\frac{\ell_\gamma}{2})$ is a linear combination with coefficients which are products of elements of the form:
    \[
        \Tr\big(S_b^{u_1}S_a^{v_1}\cdots S_b^{u_s}S_a^{v_s}\big),
    \]
    with $u_i,v_i\geq 1$ for all $i$, of terms of the form:
    \[
        \Tr\big(M_1a(\theta_1)\cdots M_ka(\theta_k)\big),
    \]
    where each $\theta_j$ is a linear combination with rational coefficients of Fenchel-Nielsen coordinates, and each $M_j$ is either of the form $w(t+t_*)$ or $w(t)S_b^pw(t)$ for some $p\in\mathbf{Z}_{\neq0}$.
\end{Lem}

\begin{proof}
    By concatenating all the incursions of $\widetilde{\gamma}$ in the pairs of pants of the complement $S(\gamma)\backslash\mathscr{P}$, we can represent the loop $\widetilde{\gamma}$ by a product of the form:
    \[
        M_1a(\theta_1)\cdots M_ka(\theta_k),
    \]
    where each $\theta_j$ is a linear combination with rational coefficients of Fenchel-Nielsen coordinates defined along $\mathscr{P}$, and each $M_j$ is of the form (\ref{S3:Eq:HoloRepGen1}) or (\ref{S3:Eq:HoloRepGen2}) (with exponents $p_i$'s and $q_i$'s of arbitrary signs). Hence since the entries of the matrices $S_b^{p_i},S_a^{q_i},a(\theta)$ and $w(\ell)$ are all non-negative (when $\ell\geq0$), we have:
    \[
        2\cosh\Big(\frac{\ell_\gamma}{2}\Big)=\Tr\Big(M_1a(\theta_1)\cdots M_ka(\theta_k)\Big).
    \]
    If there is at most one matrix $S_b^{p_i}$ in each matrix $M_j$, there is nothing to prove. Else, since the trace is invariant under cyclic permutation of the matrices, we can assume without loss of generality that the expression of the matrix $M_1$ contains at least one product $S_b^pS_a^q$. Our strategy to prove the result is to inductively reduce the number of such products in the expression of $M_1$ by iterating the trace formula:
    \[
        \Tr(AB)=\Tr(A)\Tr(B)-\Tr(A^{-1}B).
    \]
    The first step is, by induction, to replace $M_1$ by a matrix of the form:
    \begin{equation}\label{S3:Lem:TraceHoloRep1:Demo1}
        w(t)S_b^{p_1}S_a^{q_1}\cdots S_b^{p_r}S_a^{q_r}\widetilde{M}_1,
    \end{equation}
    with $p_1,q_1,\cdots,p_r,q_r$ of the same sign, and the form of the matrix $\widetilde{M}_1$ is either $w(t+t_*)$ or $w(t)S_b^pw(t)$, depending on the expression (\ref{S3:Eq:HoloRepGen1}) or (\ref{S3:Eq:HoloRepGen2}) of $M_1$.
    
    Let us first assume that $p_1,q_1,\ldots,p_r,q_r$ do not have the same signs. At this step two scenarios can occur.
    \begin{enumerate}[label=\textup{(\roman*)}]
        \item\label{S3:Lem:TraceHoloRep1:DemoItem1} The exponents $p_1,q_1,\ldots,p_s,q_s$ have the same sign and $p_{s+1}$ has the opposite sign. Let us denote:
        \[
            M_1=w(t)S_b^{p_1}\cdots S_a^{q_s}S_b^{p_{s+1}}M_0,
        \]
        and $M=a(\theta_1)M_2\cdots a(\theta_{k-1})M_k$. We have:
        \begin{align*}
            2\cosh\Big(\frac{\ell_\gamma}{2}\Big) = & \Tr\Big(w(t)S_b^{p_1}\cdots S_a^{q_s}S_b^{p_{s+1}}M_0Ma(\theta_k)\Big) \\
            = & \Tr\Big(S_b^{p_1}\cdots S_a^{q_s}\Big)\Tr\Big(w(t)S_b^{p_{s+1}}M_0Ma(\theta_k)\Big) \\
            & -\Tr\Big(w(t)(S_a^{q_s})^{-1}(S_b^{p_s})^{-1}\cdots(S_a^{q_1})^{-1}(S_b^{p_1})^{-1}S_b^{p_{s+1}}M_0Ma(\theta_k)\Big)
        \end{align*}
        If the exponents $p_1,q_1,\cdots,p_s,q_s$ are positive, then the coefficient $\Tr(S_b^{p_1}\cdots S_a^{q_s})$ is an admissible coefficient in the linear combination we want to obtain. If they are negative, using (\ref{S3:Eq:InvMatS}) we have:
        \[
            \Tr\Big(S_b^{p_1}\cdots S_a^{q_s}\Big) = \Tr\Big((S_a^{q_s})^{-1}\cdots(S_b^{p_1})^{-1}\Big)=\Tr\Big(S_a^{-q_s}\cdots S_b^{-p_1}\Big),
        \]
        which has the correct form. Next, if $M_1'=w(t)S_b^{p_{s+1}}M_0$ does not have the form (\ref{S3:Lem:TraceHoloRep1:Demo1}), it falls into the scenarios \ref{S3:Lem:TraceHoloRep1:DemoItem1} or \ref{S3:Lem:TraceHoloRep1:DemoItem2} with less products $S_b^pS_a^q$ than $M_1$ and by induction we can assume without loss of generality that this matrix has the form (\ref{S3:Lem:TraceHoloRep1:Demo1}).
        
        It still remains to deal with the term:
        \[
            -\Tr\Big(w(t)(S_a^{q_s})^{-1}(S_b^{p_s})^{-1}\cdots(S_a^{q_1})^{-1}(S_b^{p_1})^{-1}S_b^{p_{s+1}}M_0Ma(\theta_k)\Big).
        \]
        Using (\ref{S3:Eq:MatWa}), (\ref{S3:Eq:MatWw}), (\ref{S3:Eq:InvMatS}) and the fact that $p_1q_s>0$ and $p_1p_{s+1}<0$, we have:
        \begin{align*}
            -\Tr\Big(w(t)(S_a^{q_s})^{-1}(S_b^{p_s})^{-1}\cdots & (S_a^{p_1})^{-1}(S_b^{p_1})^{-1}S_b^{p_{s+1}}M_0Ma(\theta_k)\Big) \\
            = & \Tr\Big(a(q_sx)w(t)S_b^{-p_s}S_a^{-p_{s-1}}\cdots S_b^{-p_2}S_a^{-q_1}S_b^{p_{s+1}-p_1}M_0Ma(\theta_k)\Big) \\
            = & \Tr\Big(w(t)S_b^{-p_s}S_a^{-p_{s-1}}\cdots S_b^{-p_2}S_a^{-q_1}S_b^{p_{s+1}-p_1}M_0Ma(\theta_k+q_sx)\Big).
        \end{align*}
        If the matrix involved in this expression does not have the correct form (\ref{S3:Lem:TraceHoloRep1:Demo1}), it falls into the scenarios \ref{S3:Lem:TraceHoloRep1:DemoItem1} or \ref{S3:Lem:TraceHoloRep1:DemoItem2} with less products $S_b^pS_a^q$, hence we can proceed by induction.
        \item\label{S3:Lem:TraceHoloRep1:DemoItem2} The exponents $p_1,q_1,\ldots,p_s$ have the same sign and $q_s$ has the opposite sign. In this case we denote:
        \[
            M_1=w(t)S_b^{p_1}\cdots S_a^{q_s}M_0,
        \]
        and $M=a(\theta_1)M_2\cdots a(\theta_{k-1})M_k$. We have:
        \begin{align*}
            2\cosh\Big(\frac{\ell_\gamma}{2}\Big) = & \Tr\Big(w(t)S_b^{p_1}\cdots S_a^{q_s}M_0Ma(\theta_k)\Big) \\
            = & \Tr\Big(S_b^{p_1}S_a^{q_1}\cdots S_b^{p_s}W^{-1}\Big)\Tr\Big(w(t)WS_a^{q_s}M_0Ma(\theta_k)\Big) \\
            & - \Tr\Big(w(t)W(S_b^{p_s})^{-1}\cdots(S_a^{q_1})^{-1}(S_b^{p_1})^{-1}WS_a^{q_s}M_0Ma(\theta_k)\Big).
        \end{align*}
        Since $p_1p_s>0$, we have:
        \[
            \Tr\Big(S_b^{p_1}S_a^{q_1}\cdots S_b^{p_s}W^{-1}\Big)=\varepsilon_{p_1}\Tr\Big(S_b^{p_1+p_s}S_a^{q_1}\cdots S_b^{p_{s-1}}S_a^{q_{s-1}}\Big),
        \]
        where $\varepsilon_{p_1}$ is the sign of $p_1$, and the coefficient $\Tr(S_b^{p_1+p_s}S_a^{q_1}\cdots S_b^{p_{s-1}}S_a^{q_{s-1}})$ has an admissible form. Let $\varepsilon_{q_s}$ be the sign of $q_s$. We have:
        \begin{align*}
            \Tr\Big(w(t)WS_a^{q_s}M_0Ma(\theta_k)\Big) = & -\varepsilon_{q_s}\Tr\Big(a(-q_sx)w(t)M_0Ma(\theta_k)\Big) \\
            = & -\varepsilon_{q_s}\Tr\Big(w(t)M_0Ma(\theta_k-q_sx)\Big),
        \end{align*}
        hence, since $\varepsilon_{q_s}\varepsilon_{p_1}=-1$, by induction the term:
        \[
            \Tr\Big(S_b^{p_1}S_a^{q_1}\cdots S_b^{p_s}W^{-1}\Big)\Tr\Big(w(t)WS_a^{q_s}M_0Ma(\theta_k)\Big)
        \]
        has the correct form. Finally we have:
        \begin{align*}
            - \Tr\Big(w(t)W(S_b^{p_s})^{-1}\cdots(S_a^{q_1})^{-1}(S_b^{p_1})^{-1}WS_a^{q_s}M_0 & Ma(\theta_k)\Big) \\
            = & \Tr\Big(w(t)S_b^{-p_s}S_a^{-q_{s-1}}\cdots S_b^{-p_1}S_a^{q_s}M_0Ma(\theta_k)\Big),
        \end{align*}
        hence if the matrix involved in this term has not the form (\ref{S3:Lem:TraceHoloRep1:Demo1}), it falls into the scenarios \ref{S3:Lem:TraceHoloRep1:DemoItem1} or \ref{S3:Lem:TraceHoloRep1:DemoItem2} with a longer sequence of consecutive exponents $p_i,q_i$ of the same sign, and we can conclude by induction.
    \end{enumerate}
    After this first induction, and by performing cyclic permutations of the matrices $M_ja(\theta_j)$, we can express $2\cosh(\frac{\ell_\gamma}{2})$ as a linear combination (with coefficients which are products of terms $\Tr(S_b^{u_1}\cdots S_a^{v_s})$ with $u_1,v_1,\ldots,u_s,v_s\geq1$) of terms of the form:
    \[
        \Tr\Big(M_1a(\theta_1)\cdots M_ka(\theta_k)\Big),
    \]
    where each $M_j$ has the form (\ref{S3:Lem:TraceHoloRep1:Demo1}). Let us now perform another induction to replace the matrix $M_1$ by either $w(t+t_*)$ or $w(t)S_b^pw(t)$ for some $p\in\mathbf{Z}_{\neq0}$. Once again several scenarios can occur.
    \begin{itemize}[label=$\bullet$]
        \item The matrix $M_1$ has the form (\ref{S3:Eq:HoloRepGen1}), i.e. we have:
        \[
            M_1=w(t)S_b^{p_1}\cdots S_a^{q_r}w(t_*),
        \]
        with $p_1,q_1,\ldots,p_r,q_r$ of the same sign. We denote $M'=M_2a(\theta_2)\cdots a(\theta_{k-1})M_k$. We have:
        \begin{align*}
            \Tr\Big(M_1a(\theta_1)\cdots M_ka(\theta_k)\Big) = & \Tr\Big(w(t)S_b^{p_1}\cdots S_a^{q_r}w(t_*)a(\theta_1)M'a(\theta_k)\Big) \\
            = & \Tr\Big(S_b^{p_1}\cdots S_a^{q_r}\Big)\Tr\Big(w(t+t_*)a(\theta_1)M'a(\theta_k)\Big) \\
            & - \Tr\Big(w(t)(S_a^{q_r})^{-1}\cdots(S_b^{p_1})^{-1}w(t_*)a(\theta_1)M'a(\theta_k)\Big).
        \end{align*}
        The coefficient $\Tr(S_b^{p_1}\cdots S_a^{q_r})$ is an admissible coefficient for our linear combination, and in the term $\Tr(w(t+t_*)a(\theta_1)M'a(\theta_k))$, the matrix $M_1$ has been replaced by $w(t+t_*)$. Finally, using (\ref{S3:Eq:MatWa}), (\ref{S3:Eq:MatWw}) and (\ref{S3:Eq:InvMatS}) we have:
        \begin{align*}
            -\Tr\Big(w(t)(S_a^{q_r})^{-1}\cdots & (S_b^{p_1})^{-1}w(t_*)a(\theta_1)M'a(\theta_k)\Big) \\
            = & -\Tr\Big(w(t)(S_a^{q_r})^{-1}WS_b^{-p_r}S_a^{-q_{r-1}}\cdots S_b^{-p_2}S_a^{-q_1}W^{-1}(S_b^{p_1})^{-1}w(t_*)a(\theta_1)M'a(\theta_k)\Big) \\
            = & \Tr\Big(a(q_rx)w(t)S_b^{-p_r}S_a^{-q_{r-1}}\cdots S_b^{-p_2}S_a^{-q_1}w(t_*)a(-p_1y)a(\theta_1)M'a(\theta_k)\Big) \\
            = & \Tr\Big(w(t)S_b^{-p_r}S_a^{-q_{r-1}}\cdots S_b^{-p_2}S_a^{-q_1}w(t_*)a(\theta_1-p_1y)M'a(\theta_k+q_rx)\Big).
        \end{align*}
        Thus we replaced the matrix $M_1$ by a matrix with less products $S_b^pS_a^q$, with exponents of the same signs (and $a(\theta_1)$ and $a(\theta_k)$ have been replaced by matrices of the form $a(\theta)$ where $\theta$ is a linear combination with rational coefficients of Fenchel-Nielsen coordinates), and we can therefore proceed by induction.
        \item The matrix $M_1$ has the form (\ref{S3:Eq:HoloRepGen2}), i.e. we have:
        \[
            M_1=w(t)S_b^{p_1}S_a^{q_1}\cdots S_b^{p_r}S_a^{q_r}S_b^{p_{r+1}}w(t),
        \]
        with $p_1,q_1,\ldots,p_r,q_r,p_{r+1}$ of the same sign. Once again we denote $M'=M_2a(\theta_2)\cdots a(\theta_{k-1})M_k$.
        
        If $p_1=p_{r+1}$, we have:
        \begin{align*}
            \Tr\Big(M_1a(\theta_1)\cdots M_ka(\theta_k)\Big) = & \Tr\Big(w(t)S_b^{p_1}S_a^{q_1}\cdots S_b^{p_r}S_a^{q_r}S_b^{p_{r+1}}w(t)a(\theta_1)M'a(\theta_k)\Big) \\
            = & \Tr\Big(S_b^{p_1}\cdots S_a^{q_r}\Big)\Tr\Big(w(t)S_b^{p_{r+1}}w(t)a(\theta_1)M'a(\theta_k)\Big) \\
            & - \Tr\Big(w(t)(S_a^{q_r})^{-1}(S_b^{p_r})^{-1}\cdots(S_b^{p_2})^{-1}(S_a^{q_1})^{-1}w(t)a(\theta_1)M'a(\theta_k)\Big).
        \end{align*}
        The coefficient $\Tr(S_b^{p_1}\cdots S_a^{q_r})$ is an admissible coefficient for our linear combination, and in the term $\Tr(w(t)S_b^{p_{r+1}}w(t)a(\theta_1)M'a(\theta_k))$, the matrix $M_1$ has been replaced by $w(t)S_b^{p_{r+1}}w(t)$. Using (\ref{S3:Eq:MatWa}), (\ref{S3:Eq:MatWw}) and (\ref{S3:Eq:InvMatS}), we have:
        \begin{align*}
            - \Tr\Big(w(t)(S_a^{q_r})^{-1}(S_b^{p_r})^{-1}\cdots & (S_b^{p_2})^{-1}(S_a^{q_1})^{-1}w(t)a(\theta_1)M'a(\theta_k)\Big) \\
            = & \Tr\Big(a(q_rx)w(t)S_b^{-p_r}S_a^{-q_{r-1}}\cdots S_b^{-p_2}w(t)a(-q_1x)a(\theta_1)M'a(\theta_k)\Big) \\
            = & \Tr\Big(w(t)S_b^{-p_r}S_a^{-q_{r-1}}\cdots S_b^{-p_2}w(t)a(\theta_1-q_1x)M'a(\theta_k+q_rx)\Big).
        \end{align*}
        In this term the matrix $M_1$ has been replaced by a matrix of the form (\ref{S3:Eq:HoloRepGen2}) with less products $S_b^pS_a^q$ (with exponents of the same sign), thus we can conclude by induction.
        
        If the sign of $p_{r+1}-p_1$ is the same as the one of $p_{r+1}$, we have:
        \begin{align*}
            \Tr\Big(M_1a(\theta_1)\cdots M_ka(\theta_k)\Big) = & \Tr\Big(w(t)S_b^{p_1}S_a^{q_1}\cdots S_b^{p_r}S_a^{q_r}S_b^{p_{r+1}}w(t)a(\theta_1)M'a(\theta_k)\Big) \\
            = & \Tr\Big(S_b^{p_1}\cdots S_a^{q_r}\Big)\Tr\Big(w(t)S_b^{p_{r+1}}w(t)a(\theta_1)M'a(\theta_k)\Big) \\
            & - \Tr\Big(w(t)(S_a^{q_r})^{-1}(S_b^{p_r})^{-1}\cdots(S_a^{q_1})^{-1}(S_b^{p_1})^{-1}S_b^{p_{r+1}}w(t)a(\theta_1)M'a(\theta_k)\Big),
        \end{align*}
        and:
        \begin{align*}
            - \Tr\Big(w(t)(S_a^{q_r})^{-1}(S_b^{p_r})^{-1}\cdots & (S_a^{q_1})^{-1}(S_b^{p_1})^{-1}S_b^{p_{r+1}}w(t)a(\theta_1)M'a(\theta_k)\Big) \\
            = & \Tr\Big(a(q_rx)w(t)S_b^{-p_r}S_a^{-q_{r-1}}\cdots S_b^{-p_2}S_a^{-q_1}S_b^{p_{r+1}-p_1}w(t)a(\theta_1)M'a(\theta_k)\Big) \\
            = & \Tr\Big(w(t)S_b^{-p_r}S_a^{-q_{r-1}}\cdots S_b^{-p_2}S_a^{-q_1}S_b^{p_{r+1}-p_1}w(t)a(\theta_1)M'a(\theta_k+q_rx)\Big),
        \end{align*}
        The coefficients $-p_r,-q_{r-1},\cdots,-p_2,-q_1$ have the same sign and $p_{r+1}-p_1$ has the opposite sign, hence the matrix $M_1'=w(t)S_b^{-p_r}S_a^{-q_{r-1}}\cdots S_b^{-p_2}S_a^{-q_1}S_b^{p_{r+1}-p_1}w(t)$ falls into the scenario \ref{S3:Lem:TraceHoloRep1:DemoItem1} and we can conclude by induction.
        
        Finally, if $p_1-p_{r+1}$ has the same sign as $p_1$, we have:
        \begin{align*}
            \Tr\Big(M_1a(\theta_1)\cdots M_ka(\theta_k)\Big) = & \Tr\Big(a(\theta_1)M'a(\theta_k)M_1\Big) \\
            = & \Tr\Big(a(\theta_1)M'a(\theta_k)w(t)S_b^{p_1}S_a^{q_1}\cdots S_a^{q_r}S_b^{p_{r+1}}w(t)\Big) \\
            = & \Tr\Big(a(\theta_1)M'a(\theta_k)w(t)S_b^{p_1}w(t)\Big)\Tr\Big(S_a^{q_1}S_b^{p_2}\cdots S_a^{q_r}S_b^{p_{r+1}}\Big) \\
            & - \Tr\Big(a(\theta_1)M'a(\theta_k)w(t)S_b^{p_1}(S_b^{p_{r+1}})^{-1}(S_a^{q_r})^{-1}\cdots(S_b^{p_2})^{-1}(S_a^{q_1})^{-1}w(t)\Big),
        \end{align*}
        and we proceed similarly to the case $(p_{r+1}-p_1)p_{r+1}>0$ to conclude by induction.
    \end{itemize}
\end{proof}

\begin{Rem}
    In each term of the form:
    \[
        \Tr\Big(M_1a(\theta_1)\cdots M_ka(\theta_k)\Big)
    \]
    provided by Lemma \ref{S3:Lem:TraceHoloRep1}, the number of matrices $2k$ is constant and depends only on $\widetilde{\gamma}$ (and on the pants decomposition $\mathscr{P}$). Moreover the form $w(t+t_*)$ or $w(t)S_b^pw(t)$ of the matrix $M_j$ only depends on the number of matrices $S_b^p$ and $S_a^q$ appearing in the matricial representative (\ref{S3:Eq:HoloRepGen1}) or (\ref{S3:Eq:HoloRepGen2}) of the $j$-th incursion of $\widetilde{\gamma}$ inside a pair of pants of the complement $S(\gamma)\backslash\mathscr{P}$. In particular:
    \begin{itemize}
        \item if the $j$-th incursion of $\widetilde{\gamma}$ enters and exits a pair of pants by passing through different curves, then this incursion is represented by the matrix (\ref{S3:Eq:HoloRepGen1}) and each $M_j$ provided by Lemma \ref{S3:Lem:TraceHoloRep1} is of the form $w(t+t_*)$;
        \item if the $j$-th incursion of $\widetilde{\gamma}$ enters and exits a pair of pants by passing through the same curve, then this incursion is represented by the matrix (\ref{S3:Eq:HoloRepGen2}) and each $M_j$ provided by Lemma \ref{S3:Lem:TraceHoloRep1} is of the form $w(t)S_b^pw(t)$.
    \end{itemize}
\end{Rem}

\begin{Lem}\label{S3:Lem:TraceHoloRep2}
    For any $p_1,q_1,\ldots,p_r,q_r\in\mathbf{N}_{\geq1}$, there exist a finite set $G$ of linear forms on $\mathbf{R}^3$ and $\{a_T\}_{T\in G}\subset\mathbf{R}_{>0}$ such that:
    \[
        \Tr\big(S_b^{p_1}S_a^{q_1}\cdots S_b^{p_r}S_a^{q_r}\big) = \sum_{T\in G}a_Te^{T(x,y,z)}.
    \]
\end{Lem}

\begin{proof}
    Any word in the letters $A$ and $B$, where the matrices $A$ and $B$ are defined in Lemma \ref{S3:Lem:HoloRepGeneral2} (with $(\ell_{\alpha_0},\ell_{\beta_0},\ell_{\beta_1})=(x,y,z)$) is the representative in $\SL_2(\mathbf{R})$ of a loop contained inside the pair of pants bounded by $\alpha_0$, $\beta_0$ and $\beta_1$. Hence according to Proposition \ref{S3:Prop:LenPantsFill}, we have:
    \[
        \Tr\Big(B^{p_1}A^{q_1}\cdots B^{p_r}A^{q_r}\Big) = \sum_{T\in G}a_Te^{T(x,y,z)},
    \]
    for some finite set $G\subset(\mathbf{R}^3)^*$ and $\{a_T\}_{T\in G}\subset\mathbf{R}_{>0}$. We conclude by using the fact that for positive $p_1,q_1,\ldots,p_r,q_r$, we have:
    \[
        S_b^{p_1}\cdots S_a^{q_r}=R\Big(-\frac{\pi}{2}\Big)B^{p_1}\cdots A^{q_r}R\Big(\frac{\pi}{2}\Big).
    \]
\end{proof}

\begin{Cor}\label{S3:Cor:HoloRepGen2}
    Let $\gamma$ be a closed loop on $S_{g,n}$. Let $\delta_1,\ldots,\delta_N$ be the connected components of $\partial S(\gamma)$ and $\mathscr{P}$ be a pants decomposition of $S(\gamma)$. We have:
    \[
        \cosh\Big(\frac{\ell_\gamma}{2}\Big)=\frac{1}{\prod_{\alpha\in\mathscr{P}}\sinh(\frac{\ell_\alpha}{2})^{d_\alpha}}P((\ell_\alpha,\tau_\alpha)_{\alpha\in\mathscr{P}},(\ell_{\delta_i})_{1\leq i\leq N}),
    \]
    where for all $\alpha\in\mathscr{P}$, we have $d_\alpha\in\mathbf{N}_{\geq1}$, and $P$ is of the form:
    \[
        P=\sum_{i=1}^p\sqrt{\sum_{T\in G_i}a_T^ie^T},
    \]
    where $G_1,\ldots,G_p$ are finite sets of linear forms on $\mathbf{R}^{2\#\mathscr{P}+N}$, and for all $i\in\{1,\ldots,p\}$, $\{a_T^i\}_{T\in G_i}\subset\mathbf{R}_{>0}$ are positive real numbers.
\end{Cor}

\begin{proof}
    According to Lemma \ref{S3:Lem:TraceHoloRep1} and Lemma \ref{S3:Lem:TraceHoloRep2}, we have:
    \begin{equation}\label{S3:Cor:HoloRepGen2:Demo1}
        2\cosh\Big(\frac{\ell_\gamma}{2}\Big) = \sum_{j=1}^KP_j\Tr\Big(M_1^ja(\theta_1^j)\cdots M_k^ja(\theta_k^j)\Big),
    \end{equation}
    where for all $j\in\{1,\ldots,K\}$, we have:
    \[
        P_j=\sum_{T\in G_j}a_T^je^{T((\ell_\alpha)_{\alpha\in\mathscr{P}},(\ell_{\delta_i})_{1\leq i\leq N})},
    \]
    each $\theta_i^j$ is a linear combination with rational coefficients of the Fenchel-Nielsen parameters $(\ell_\alpha,\tau_\alpha)_{\alpha\in\mathscr{P}}$, $(\ell_{\delta_i})_{1\leq i\leq N}$, and each $M_i^j$ is either of the form $w(t+t_*)$ or $w(t)S_b^pw(t)$. It is important to notice that in (\ref{S3:Cor:HoloRepGen2:Demo1}), each product $M_1^ja(\theta_1^j)\cdots M_k^ja(\theta_k^j)$ is a product of exactly $2k$ terms, with $k$ depending only on $\gamma$ (and possibly on the pants decomposition $\mathscr{P}$). Let us compute the entries of the matrices $w(t+t_*)$ and $w(t)S_b^pw(t)$.
    
    Using hyperbolic trigonometry in right-angled hexagons, we have:
    \[
        \cosh(t+t_*) = \frac{\cosh(\frac{x}{2})\cosh(\frac{y}{2})+\cosh(\frac{z}{2})}{\sinh(\frac{x}{2})\sinh(\frac{y}{2})}.
    \]
    Hence:
    \[
        w(t+t_*) = \frac{1}{\sqrt{2\sinh(\frac{x}{2})\sinh(\frac{y}{2})}}\left(\begin{array}{cc}
            \sqrt{\cosh(\frac{z}{2})+\cosh(\frac{x+y}{2})} & \sqrt{\cosh(\frac{z}{2})+\cosh(\frac{x-y}{2})} \\
            \sqrt{\cosh(\frac{z}{2})+\cosh(\frac{x-y}{2})} & \sqrt{\cosh(\frac{z}{2})+\cosh(\frac{x+y}{2})}
        \end{array}\right).
    \]
    Denoting the sign of $p$ by $\varepsilon_p$, we also have:
    \[
        w(t)S_b^pw(t) = \left(\begin{array}{cc}
            \sinh(\frac{|p|y}{2})\sinh(t+t_*) & \sinh(\frac{|p|y}{2})\cosh(t+t_*)+\varepsilon_p\cosh(\frac{py}{2}) \\
            \sinh(\frac{|p|y}{2})\cosh(t+t_*)-\varepsilon_p\cosh(\frac{py}{2}) & \sinh(\frac{|p|y}{2})\sinh(t+t_*)
        \end{array}\right).
    \]
    The diagonal coefficients of $w(t)S_b^pw(t)$ are therefore given by:
    \begin{align*}
        \frac{1}{\sinh(\frac{x}{2})}\frac{\sinh(\frac{|p|y}{2})}{\sinh(\frac{y}{2})} & \sqrt{\sinh^2\Big(\frac{x}{2}\Big)+\cosh^2\Big(\frac{y}{2}\Big)+\cosh^2\Big(\frac{z}{2}\Big)+2\cosh\Big(\frac{x}{2}\Big)\cosh\Big(\frac{y}{2}\Big)\cosh\Big(\frac{z}{2}\Big)} \\
        = & \frac{1}{\sinh(\frac{x}{2})}\frac{\sinh(\frac{|p|y}{2})}{\sinh(\frac{y}{2})}\sqrt{\bigg(\cosh\Big(\frac{x+y}{2}\Big)+\cosh\Big(\frac{z}{2}\Big)\bigg)\bigg(\cosh\Big(\frac{x-y}{2}\Big)+\cosh\Big(\frac{z}{2}\Big)\bigg)},
    \end{align*}
    and the antidiagonal coefficients of $w(t)S_b^pw(t)$ are given by:
    \[
        \frac{1}{\sinh(\frac{x}{2})}\bigg[\cosh\Big(\frac{z}{2}\Big)\frac{\sinh(\frac{|p|y}{2})}{\sinh(\frac{y}{2})}+\frac{e^{\frac{x}{2}}}{2}\frac{\sinh(\frac{(|p|+1)y}{2})}{\sinh(\frac{y}{2})}+\frac{e^{-\frac{x}{2}}}{2}\frac{\sinh(\frac{(|p|-1)y}{2})}{\sinh(\frac{y}{2})}\bigg],
    \]
    and:
    \[
        \frac{1}{\sinh(\frac{x}{2})}\bigg[\cosh\Big(\frac{z}{2}\Big)\frac{\sinh(\frac{|p|y}{2})}{\sinh(\frac{y}{2})}+\frac{e^{-\frac{x}{2}}}{2}\frac{\sinh(\frac{(|p|+1)y}{2})}{\sinh(\frac{y}{2})}+\frac{e^{\frac{x}{2}}}{2}\frac{\sinh(\frac{(|p|-1)y}{2})}{\sinh(\frac{y}{2})}\bigg].
    \]
    Thus we can write:
    \[
        w(t)S_b^pw(t) = \frac{1}{\sinh(\frac{x}{2})}\left(\begin{array}{cc}
            Q_1 & Q_2 \\
            Q_3 & Q_4
        \end{array}\right),
    \]
    where $Q_1,Q_2,Q_3,Q_4$ are of the form:
    \begin{equation}\label{S3:Cor:HoloRepGen2:Demo2}
        \sum_{i=1}^q\sqrt{\sum_{T\in G_i}a_T^ie^T},
    \end{equation}
    where each $G_i$ is a finite set of linear forms and each $\{a_T^i\}_{T\in G_i}$ is a set of positive real numbers.
    
    Moreover, a matrix $M_i^j$ in (\ref{S3:Cor:HoloRepGen2:Demo1}) provides a matrix $w(t+t_*)$ after having applied Lemma \ref{S3:Lem:TraceHoloRep1} if and only if the corresponding incursion of $\widetilde{\gamma}$ in a pair of pants $P_1$ of the complement $S(\gamma)\backslash\mathscr{P}$ enters inside $P_1$ and exits $P_1$ by passing through different curves $\alpha_0$ and $\beta_0$. Hence the following incursion of $\widetilde{\gamma}$ into a neighboring pair of pants $P_2$ enters $P_2$ by passing through $\beta_0$. Two scenarios may occur.
    \begin{itemize}[label=$\bullet$]
        \item $\widetilde{\gamma}$ exits $P_2$ by crossing a curve different from $\beta_0$, hence the corresponding incursion $\widetilde{\gamma}_{P_2}$ provides a matrix of the form $w(t+t_*)$ after having applied Lemma \ref{S3:Lem:TraceHoloRep1} and $\sqrt{\sinh(\frac{\ell_{\beta_0}}{2})}$ appears twice in the denominator of the entries of $M_1^ja(\theta_1^j)\cdots M_k^ja(\theta_k^j)$.
        \item $\widetilde{\gamma}$ exits $P_2$ by crossing the same curve $\beta_0$, in which case the corresponding incursion $\widetilde{\gamma}_{P_2}$ provides a matrix of the form $w(t)S_b^pw(t)$ after having applied Lemma \ref{S3:Lem:TraceHoloRep1}. In this case, after a finite number of consecutive incursions into $P_1$ and $P_2$, $\widetilde{\gamma}$ will eventually leave $P_1$ by passing through a curve different from $\beta_0$. Such an incursion will provide a matrix of the form $w(t+t_*)$, hence $\sqrt{\sinh(\frac{\ell_{\beta_0}}{2}})$ will appear twice in the denominator of the entries of $M_1^ja(\theta_1^j)\cdots M_k^ja(\theta_k^j)$.
    \end{itemize}
    Thus we have:
    \[
        2\cosh\Big(\frac{\ell_\gamma}{2}\Big)=\frac{1}{\prod_{\alpha\in\mathscr{P}}\sinh(\frac{\ell_\alpha}{2})^{d_\alpha}}\bigg(\sum_{j=1}^K\Tr(N_j)\bigg),
    \]
    where $d_\alpha\in\mathbf{N}_{\geq1}$ for all $\alpha\in\mathscr{P}$, and the entries of each matrix $N_j$ are of the form (\ref{S3:Cor:HoloRepGen2:Demo2}).
\end{proof}

\begin{Rem}
    If $\beta$ is a boundary of $S(\gamma)$, then $\gamma$ never passes through $\beta$. Hence $\sinh(\frac{\ell_\beta}{2})$ does not appear in the denominator of $2\cosh(\frac{\ell_\gamma}{2})$.
\end{Rem}

\subsection{Length-type functions}\label{S3:Subsec:LenType}

Following Corollary \ref{S3:Cor:HoloRepGen2}, we define the class of \emph{length-type functions}.

\begin{Def}\label{S3:Def:LenType}
    A function $h\colon\mathbf{R}_{>0}^n\times\mathbf{R}^n\times\mathbf{R}_{\geq0}^k\longrightarrow\mathbf{R}_{\geq0}$ is said to be \emph{length-type} if it is proper and if we can write:
    \[
        \cosh\Big(\frac{h(\mathbf{x}_n,\mathbf{t}_n,\mathbf{y}_k)}{2}\Big)=\frac{1}{\prod_{j=1}^n\sinh(\frac{x_j}{2})^{d_j}}\sum_{i=1}^p\sqrt{\sum_{T\in G}a_T^ie^{T(\mathbf{x}_n,\mathbf{t}_n,\mathbf{y}_k)}},
    \]
    where $d_1,\ldots,d_n\in\mathbf{N}_{\geq1}$, and $G_1,\ldots,G_p$ are finite sets of linear forms on $\mathbf{R}^{2n+k}$, and for each $i\in\{1,\ldots,p\}$, $\{a_T^i\}_{T\in G}\subset\mathbf{R}_{>0}$ are positive real numbers.
    
    Let $u\colon\ell\mapsto2\log(2\cosh(\frac{\ell}{2}))$. We say that $H\colon\mathbf{R}_{>0}^n\times\mathbf{R}^n\times\mathbf{R}_{\geq0}^k\longrightarrow\mathbf{R}_{\geq0}$ is a \emph{pseudo length-type function} if $u^{-1}\circ H$ is a length-type function.
\end{Def}

\begin{Rem}
    As stated earlier, the fact that the coefficients $a_T^i$'s are positive is a crucial point for our computations. We will keep track of what would happen if these coefficients had arbitrary signs in our following computations to illustrate this fact.
\end{Rem}

If the surface $S(\gamma)\subset S_{g,n}$ filled by the loop $\gamma$ has disjoint boundaries $\delta,\delta'$ such that $\delta$ is freely homotopic to $\delta^{\pm1}$ (e.g. if $\gamma$ is an eight-shaped loop on a once-holed torus), then the expression $h_\gamma$ of the length function $\ell_\gamma$ in Fenchel-Nielsen coordinates defined along $\partial S(\gamma)$ is the restriction of a length-type function to a hyperplane of the form $\{y_i=y_j\}$. Similarly, if one of the boundaries of $\partial S(\gamma)$ is homotopic to a boundary or a cusp of $S_{g,n}$, then $h_\gamma$ is the restricition of a length-type function to a hyperplane of the form $\{y_i=L\}$.

\begin{Lem}
    The restrictions of a length-type function on hyperplanes of the form $\{y_i=y_j\}$ and $\{y_i=L\}$ are length-type functions.
\end{Lem}

It immediately follows:

\begin{Cor}
    Let $\gamma$ be a closed loop on $S_{g,n}$, $\beta_1,\ldots,\beta_M$ be the essential connected components of $\partial S(\gamma)$ and $2k$ be the dimension of the Teichmüller space of $S(\gamma)$. We extract a maximal multi-curve $\Gamma$ from $(\beta_1,\ldots,\beta_M)$, by possibly re-ordering the $\beta_i$'s we can assume without loss of generality that $\Gamma=(\beta_1,\ldots,\beta_N)$ for some $N\leq M$. Consider a pants decomposition $(\alpha,\ldots,\alpha_k)$ of $S(\gamma)$, and complete $\Gamma\cup(\alpha_1,\ldots,\alpha_k)$ into a pants decomposition $\Gamma'$ of $S_{g,n}$. The expression $h_\gamma$ of $\ell_\gamma$ into the Fenchel-Nielsen coordinates defined along $\Gamma'$ only depends on $(\ell_{\alpha_j},\tau_{\alpha_j})_{1\leq j\leq k}$ and $(\ell_{\beta_i})_{1\leq i\leq M}$, and it is a length-type function in these variables.
\end{Cor}

Mirzakhani used the fact that the length function of a closed loop behaves similarly to a piecewise linear function in a $\mathscr{C}^0$ sense to prove Theorem \ref{S1:Theo:MirzAsympt} in \cite{Mir16}. We will need to improve this statement by obtaining a $\mathscr{C}^1$ version of her argument in order to prove our main result.

\begin{Def}
    We say that $C\subset\mathbf{R}^n$ is an \emph{admissible polytope} if $C\subset\{0\leq x_1\leq\cdots\leq x_n\}$ and if there exist a set $U$ defined by a finite number of (strict or non-strict) linear inequations, $m\in\{0,\ldots,n\}$ and a linear isomorphism $\Psi\in\GL_n(\mathbf{R})$ such that:
    \[
        C=\Psi\big(U\cap([1,+\infty)^m\times\mathbf{R}^{n-m})\big).
    \]
\end{Def}

\begin{Rem}
    An admissible polytope is stable under $\mathbf{x}_n\mapsto t\mathbf{x}_n$ for all $t\geq1$.
\end{Rem}

\begin{Def}\label{S3:Def:SLT}
    Let $C\subset\{0\leq x_1\leq\cdots\leq x_n\}$ be an admissible polytope and $\Bar{C}$ be its closure. A function $h\colon \mathbf{R}_{>0}^m\times C\longrightarrow\mathbf{R}_{\geq0}$ is said to be a \emph{simplified length-type function} if $h$ is real-analytic on an open neighborhood of the closure of $C$, and if there exist a linear form $\Lambda$ on $\mathbf{R}^{m+n}$ and a function $\psi\in L^\infty(\mathbf{R}_{>0}^m\times C)$ such that:
    \begin{enumerate}[label=\textup{(\roman*)}]
        \item $h=\Lambda+\psi$;
        \item $\Lambda$ is positive and proper on $\mathbf{R}_{\geq0}^m\times\Bar{C}$, and $\mathbf{s}_m\mapsto\Lambda(\mathbf{s}_m,0)$ has positive integer coefficients;
        \item there exist a constant $c>0$ and linear forms $Y_1,\ldots,Y_k$ on $\mathbf{R}^n$ such that for every $p\in\{1,\ldots,n\}$ and every $(\mathbf{s}_m,\mathbf{x}_n)\in\mathbf{R}_{>0}^m\times C$, we have:
        \[
            \Big|\frac{\partial\psi(\mathbf{s}_m,\mathbf{x}_n)}{\partial x_p}\Big|\leq ce^{-\min\{x_p,Y_1(\mathbf{x}_n),\ldots,Y_k(\mathbf{x}_n)\}},
        \]
        and if $p_0=\max\{p\in\{1,\ldots,n\}|\frac{\partial\Lambda(\mathbf{s}_m,\mathbf{x}_n)}{\partial x_p}\neq0\}$, there exists a constant $M>0$ such that for all $\mathbf{x}_n\in C$, we have:
        \[
            \Lambda(\mathbf{0}_m,\mathbf{x}_n)\leq M\min\{x_{p_0},Y_1(\mathbf{x}_n),\ldots,Y_k(\mathbf{x}_n)\}.
        \]
    \end{enumerate}
\end{Def}

\begin{Rem}
    If $p_0=\max\{p\in\{1,\ldots,n\}|\frac{\partial\Lambda(\mathbf{s}_m,\mathbf{x}_n)}{\partial x_p}\neq0\}$, there always exists a constant $M>0$ such that:
    \[
        \Lambda(\mathbf{0}_m,\mathbf{x}_n)\leq Mx_{p_0},
    \]
    for all $\mathbf{x}_n\in C$. Indeed, let $K=\max_{1\leq p\leq n}\frac{\partial\Lambda(\mathbf{s}_m,\mathbf{x}_p)}{\partial x_p}$. Since $\Lambda$ is non-negative and proper on the closure of $C$, we have $K>0$ and for all $\mathbf{x}_n\in C$, we have:
    \[
        \Lambda(\mathbf{0}_m,\mathbf{x}_n)\leq K(x_1+\cdots+x_{p_0})\leq p_0Kx_{p_0}.
    \]
\end{Rem}

As suggested by the terminology, length-type functions and simplified length-type functions are closely linked. Before getting more precise with this statement, let us introduce some notations.

Let $I\subset\{1,\ldots,n\}$. We note:
\begin{equation}\label{S3:Eq:Partition1}
    E_I=\Big\{\mathbf{x}_n\in\mathbf{R}_{>0}^n\Big|\forall i\in I, x_i\in(0,1),\forall i\notin I,x_i\geq1\Big\},
\end{equation}
and if $\mathbf{x}_n\in E_I$, we denote $\Phi_I(\mathbf{x}_n)$ the vector $\mathbf{s}_n$ defined by:
\begin{equation}\label{S3:Eq:Partition2}
    s_i=\left\{\begin{array}{cl}
        -\log x_i & \text{ if } i\in I \\
        x_i & \text{ if } i\notin I
    \end{array}\right.,
\end{equation}
and $F_I=\Phi_I(E_I)=\{\mathbf{s}_n\in\mathbf{R}_{>0}^n|\forall i\notin I,s_i\geq1\}$.

\begin{Theo}[Simplified form of pseudo length-type function]\label{S3:Theo:SimpFormPLT}
    Let $H\colon\mathbf{R}_{>0}^n\times\mathbf{R}^n\times\mathbf{R}_{\geq0}^k\longrightarrow\mathbf{R}_{\geq0}$ be a pseudo length-type function, $I\subset\{1,\ldots,n\}$, and $\widetilde{H}\colon(\mathbf{s}_n,\mathbf{t}_n,\mathbf{y}_k)\mapsto H(\Phi_I^{-1}(\mathbf{s}_n),\mathbf{t}_n,\mathbf{y}_k)$. There exist a partition of $F_I\times\mathbf{R}^n\times\mathbf{R}_{\geq0}^k$ into cones $U_1,\ldots,U_N$ defined by a finite number of linear inequations, and linear isomorphisms $\Psi_1,\ldots,\Psi_N\in\GL_{2n+k}(\mathbf{R})$ such that each $\Psi_i$ maps $U_i$ to a product $\mathbf{R}_{>0}^{\#I}\times C_i$, where $C_i$ is an admissible polytope, and $\widetilde{H}\circ\Psi_i^{-1}\colon\mathbf{R}_{>0}^{\#I}\times C_i\longrightarrow\mathbf{R}_{\geq0}$ is a simplified length-type function.
\end{Theo}

\begin{proof}
    Let us denote:
    \[
        H\colon(\mathbf{x}_n,\mathbf{t}_n,\mathbf{y}_k)\longmapsto\log\bigg[\frac{1}{\prod_{j=1}^n\sinh(\frac{x_j}{2})^{d_j}}P(\mathbf{x}_n,\mathbf{t}_n,\mathbf{y}_k)\bigg],
    \]
    with:
    \begin{equation}\label{S3:Theo:SimpFormPLT:Demo1}
        P(\mathbf{x}_n,\mathbf{t}_n,\mathbf{y}_k) = \sum_{i=1}^p\sqrt{\sum_{T\in G_i}a_T^ie^{T(\mathbf{x}_n,\mathbf{t}_n,\mathbf{y}_k)}}.
    \end{equation}
    By possibly permuting some variables, we can assume without loss of generality that $I=\{1,\ldots,m\}$ for some $m\leq n$ and we denote $\widetilde{H} \colon (\mathbf{s}_m,\mathbf{x}_{n-m},\mathbf{t}_n,\mathbf{y}_k)\mapsto H(\Phi_I^{-1}(\mathbf{s}_m,\mathbf{x}_{n-m}),\mathbf{t}_n,\mathbf{y}_k)$. In this proof, the variable $\mathbf{x}_{n-m}$ will denote the vector $(x_{m+1},\ldots,x_n)$. With these notations, we have:
    \begin{align*}
        \widetilde{H}(\mathbf{s}_m,\mathbf{x}_{n-m},\mathbf{t}_n,\mathbf{y}_k) = & \sum_{j=1}^md_js_j-\sum_{j=1}^md_j\log\Big[e^{s_j}\sinh\Big(\frac{e^{-s_j}}{2}\Big)\Big] \\
        & - \frac{1}{2}\sum_{j=m+1}^nd_jx_j+\sum_{j=m+1}^nd_j\big(\log2-\log[1-e^{-x_j}]\big) \\
        & + \log[P(e^{-s_1},\ldots,e^{-s_m},\mathbf{x}_{n-m},\mathbf{t}_n,\mathbf{y}_k)].
    \end{align*}
    Let:
    \[
        \widetilde{H}_0(\mathbf{s}_m,\mathbf{x}_{n-m},\mathbf{t}_n,\mathbf{y}_k) =  \log[P(e^{-s_1},\ldots,e^{-s_m},\mathbf{x}_{n-m},\mathbf{t}_n,\mathbf{y}_k)].
    \]
    Let us note $G=\bigcup_{i=1}^pG_i$. If $T$ is a linear form on $\mathbf{R}^{2n+k}$, we denote:
    \begin{align*}
        \widetilde{T}(\mathbf{x}_{n-m},\mathbf{t}_n,\mathbf{y}_k) = & T(\mathbf{0}_m,\mathbf{x}_{n-m},\mathbf{t}_n,\mathbf{y}_k) \\
        R_T(\mathbf{s}_m) = & T(e^{-s_1},\ldots,e^{-s_m},\mathbf{0}_{2n-m+k})
    \end{align*}
    We also denote $\widetilde{G}_i=\{\widetilde{T}|T\in G_i\}$ and $\widetilde{G}=\{\widetilde{T}|T\in G\}$. Given $S\in G_i$, we note:
    \begin{equation}
        \label{S3:Theo:SimpFormPLT:Demo2} \alpha_S^i(\mathbf{s}_m) = \sum_{\substack{T\in G_i\\\widetilde{T}=S}}a_T^ie^{R_T(\mathbf{s}_m)}.
    \end{equation}
    With these notations, we have:
    \begin{equation}
        \label{S3:Theo:SimpFormPLT:Demo3} \widetilde{H}_0 = \log\bigg[\sum_{i=1}^p\sqrt{\sum_{T\in\widetilde{G}_i}\alpha_T^ie^T}\bigg].
    \end{equation}
    For each $i\in\{1,\ldots,p\}$, let $T_i\in \widetilde{G}_i$ be such that the vector cone:
    \[
        V_i = \Big\{(\mathbf{x}_{n-m},\mathbf{t}_n,\mathbf{y}_k)\in\mathbf{R}_{>0}^{n-m}\times\mathbf{
        R}^n\times\mathbf{R}_{\geq0}^k\Big|\forall T\in\widetilde{G}_i\backslash\{T_i\},T(\mathbf{x}_{n-m},\mathbf{t}_n,\mathbf{y}_k)\prec T_i(\mathbf{x}_{n-m},\mathbf{t}_n,\mathbf{y}_k)\Big\}
    \]
    is non empty. Here each occurrence of $\prec$ is either $\leq$ or $<$, chosen so that $\mathbf{R}_{>0}^{n-m}\times\mathbf{R}^n\times\mathbf{R}_{\geq0}^k$ is partitioned into cones of the form $V_i$. The space $\mathbf{R}_{>0}^{n-m}\times\mathbf{R}^n\times\mathbf{R}_{\geq0}^k$ is therefore partitioned into cones of the form:
    \[
        W=V_1\cap\cdots\cap V_p.
    \]
    We fix one these cones $W\neq\emptyset$. Inside $\mathbf{R}_{>0}^m\times W$, we have:
    \[
        \widetilde{H}_0 = \log\bigg[\sum_{i=1}^pe^{\frac{T_i}{2}}\sqrt{\alpha_{T_i}^i+\sum_{\substack{T\in\widetilde{G}_i\\T\neq T_i}}\alpha_T^ie^{-(T_i-T)}}\bigg].
    \]
    We fix $T_0\in\{T_1,\ldots,T_p\}$ such that the vector cone:
    \[
        W_0 = \bigcap_{\substack{i=1\\T_i\neq T_0}}^p\{T_i\prec T_0\}
    \]
    is non-empty. Here each occurrence of $\prec$ is either $\leq$ or $<$, chosen so that the cones of the form $W_0$ realize a partition of $W$. Let us fix one of these cones $W_0\neq\emptyset$. Inside $\mathbf{R}_{>0}^m\times W_0$, we have:
    \[
        \widetilde{H}_0 = \frac{T_0}{2}+ \log\bigg[\sum_{\substack{i=1\\T_i=T_0}}^p\sqrt{\alpha_{T_i}^i+\sum_{\substack{T\in\widetilde{G}_i\\T\neq T_i}}\alpha_T^ie^{-(T_i-T)}}+\sum_{\substack{i=1\\T_i\neq T_0}}^pe^{-\frac{1}{2}(T_0-T_i)}\sqrt{\alpha_{T_i}^i+\sum_{\substack{T\in\widetilde{G}_i\\T\neq T_i}}\alpha_T^ie^{-(T_i-T)}}\bigg].
    \]
    We note:
    \[
        \psi_0 = \widetilde{H}_0-\frac{T_0}{2}.
    \]
    Let:
    \begin{align*}
        \Lambda\colon(\mathbf{s}_m,\mathbf{x}_{n-m},\mathbf{t}_n,\mathbf{y}_k)\longmapsto & \sum_{j=1}^md_js_j-\frac{1}{2}\sum_{j=m+1}^nd_jx_j+\frac{1}{2}T_0(\mathbf{x}_{n-m},\mathbf{t}_n,\mathbf{y}_k) \\
        \psi\colon(\mathbf{s}_m,\mathbf{x}_{n-m},\mathbf{t}_n,\mathbf{y}_k)\longmapsto & -\sum_{j=1}^md_j\log\Big[e^{s_j}\sinh\Big(\frac{e^{-s_j}}{2}\Big)\Big]+\sum_{j=m+1}^nd_j\big(\log2-\log[1-e^{-x_j}]\big) \\
        & + \psi_0(\mathbf{s}_m,\mathbf{x}_{n-m},\mathbf{t}_n,\mathbf{y}_k).
    \end{align*}
    With these notations, $\Lambda$ is linear and we have $\widetilde{H}=\Lambda+\psi$. Moreover since the $a_T^i$'s are positive in (\ref{S3:Theo:SimpFormPLT:Demo1}), the function $\psi$ is bounded on $\mathbf{R}_{>0}^m\times W_0$. Let:
    \begin{align*}
        G_0 = & \big\{(\mathbf{x}_{n-m},\mathbf{t}_n,\mathbf{y}_k)\mapsto x_j\big|j\in\{m+1,\ldots,n\}\big\}\cup\big\{T_i-T\big|i\in\{1,\ldots,p\},T\in\widetilde{G}_i\big\} \\
        & \cup\Big\{\frac{1}{2}(T_0-T_i)\Big|i\in\{1,\ldots,p\}\Big\}\cup\Big\{\frac{1}{2}(T_0+T_i)-T\Big|i\in\{1,\ldots,p\},T\in\widetilde{G}_i\Big\}.
    \end{align*}
    We can notice that all the linear forms of $G_0$ are non-negative on the cone $W_0$. If $G_0$ is not a generating set of the dual space $(\mathbf{R}^{2n-m+k})^*$, we complete it into a finite generating set $\widetilde{G}_0$ made of non-negative linear forms on $W_0$. Let us fix an ordering of the elements of $\widetilde{G}_0$, i.e. we fix a labelling $S_1,\ldots,S_N$ of the elements of $\widetilde{G}_0$ such that the vector cone:
    \[
        U = \{0\prec S_1\prec\cdots\prec S_N\}
    \]
    is non-empty. Here each occurrence of $\prec$ is either $\leq$ or $<$, chosen so that the cones of the form $U$ realize a partition of $W_0$. Since $\widetilde{G}_0$ is a generating set of $(\mathbf{R}^{2n-m+k})^*$, we can define:
    \begin{align*}
        \begin{array}{lcl}
            K_1 & = & 1 \\
            K_2 & = & \min\{j>1|S_j\notin\langle S_{K_1}\rangle\} \\
            & \vdots & \\
            K_{2n-m+k} & = & \min\{j>K_{2n-m+k-1}|S_j\notin\langle S_{K_1},\ldots,S_{K_{2n-m+k-1}}\rangle\}.
        \end{array}
    \end{align*}
    The map $\Psi=(S_{K_1},\ldots,S_{K_{2n-m+k}})$ is therefore a linear isomorphism. Moreover $\Psi$ maps the cone $U$ into a vector cone $C\subset\{0\leq z_1\leq\cdots\leq z_{2n-m+k}\}$. Let $\widetilde{U}=U\cap\bigcap_{j=m+1}^n\{x_j\geq1\}$ and $\widetilde{C}=\Psi(\widetilde{U})$. The set $\widetilde{C}$ is an admissible polytope.
    
    Let us prove that $(\mathbf{s}_m,\mathbf{z}_{2n-m+k})\mapsto\widetilde{H}(\mathbf{s}_m,\Psi^{-1}(\mathbf{z}_{2n-m+k}))$ matches the definition of a simplified length-type function with $Y_1(\mathbf{z}_{2n-m-k})=\cdots=Y_k(\mathbf{z}_{2n-m+k})=z_{2n-m+k}$. Since $\widetilde{H}$ is non-negative and proper, the linear form $(\mathbf{s}_m,\mathbf{z}_{2n-m+k})\mapsto\Lambda(\mathbf{s}_m,\Psi^{-1}(\mathbf{z}_{2n-m+k}))$ is proper on the closure of $\mathbf{R}_{>0}^m\times \widetilde{C}$, and by definition the coefficients of $\mathbf{s}_m\mapsto\Lambda(\mathbf{s}_m,\Psi^{-1}(\mathbf{0}_{2n-m+k}))$ are positive.
    
    If $j\in\{m+1,\ldots,n\}$, we note $S^j\colon(\mathbf{x}_{n-m},\mathbf{t}_n,\mathbf{y}_k)\mapsto x_j$ and $M_j\in\{1,\ldots,2n-m+k\}$ the minimal integer such that $S^j\in\langle S_{K_1},\ldots,S_{K_{M_j}}\rangle$. By definition of $\Psi$, the linear form $S^j\circ\Psi^{-1}$ only depends on its $M_j$ first variables, hence if $M>M_j$ we have:
    \[
        \frac{\partial}{\partial z_M}\Big(\log[1-e^{-S^j\circ\Psi^{-1}}]\Big)=0.
    \]
    Since $S^j\in\widetilde{G}_0$, we have $S^j\geq S_{K_{M_j}}$ on $U$, hence if $M\leq M_j$, we have on $\widetilde{C}$:
    \[
        \frac{\partial}{\partial z_M}\Big(\log[1-e^{-S^j\circ\Psi^{-1}}]\Big) = O\Big(e^{-S^j\circ\Psi^{-1}}\Big) = O(e^{-z_{M_j}}) = O(e^{-z_M}).
    \]
    It follows that for all $M\in\{1,\ldots,2n-m+k\}$, we have:
    \[
        \frac{\partial}{\partial z_M}\bigg(-\sum_{j=1}^md_j\log\Big[e^{s_j}\sinh\Big(\frac{e^{-s_j}}{2}\Big)\Big]+\bigg(\sum_{j=m+1}^nd_j\big(\log2-\log[1-e^{-S^j}]\big)\bigg)\circ\Psi^{-1}\bigg) = O(e^{-z_M}).
    \]
    And for every $i\in\{1,\ldots,p\}$ we can write:
    \[
        \sum_{\substack{T\in\widetilde{G}_i\\ T\neq T_i}}\alpha_T^i(\mathbf{s}_m)e^{-(T_i-T)(\mathbf{x}_{n-m},\mathbf{t}_n,\mathbf{y}_k)} = \sum_{S\in\widetilde{G}_0}\rho_S^i(\mathbf{s}_m)e^{-S(\mathbf{x}_{n-m},\mathbf{t}_n,\mathbf{y}_k)},
    \]
    where $\rho_S^i$ is either $0$ or a bounded positive function on $\mathbf{R}_{>0}^m$. We have:
    \[
        \sum_{S\in\widetilde{G}_0}\rho_S^ie^{-S} = \sum_{M=1}^{2n-m+k}e^{-S_{K_M}}\bigg(\sum_{\substack{S\in\langle S_{K_1},\ldots,S_{K_M}\rangle\\S\notin\langle S_{K_1},\ldots,S_{K_{M-1}}\rangle}}\rho_S^ie^{-(S-S_{K_M})}\bigg).
    \]
    It follows that if $M'\in\{1,\ldots,2n-m+k\}$, we have:
    \begin{align*}
        \frac{\partial}{\partial z_{M'}}\bigg[\bigg(\sum_{\substack{i=1\\ T_i=T_0}}^p & \sqrt{\alpha_{T_i}^i+\sum_{\substack{T\in\widetilde{G}_i\\ T\neq T_i}}\alpha_T^ie^{-(T_i-T)}}\bigg)\circ \Psi^{-1}\bigg]\\
        = & O\bigg(\sum_{\substack{i=1\\T_i=T_0}}^p\frac{\partial}{\partial z_{M'}}\bigg(\bigg(\sum_{S\in\widetilde{G}_0}\rho_S^ie^{-S}\bigg)\circ\Psi^{-1}\bigg)\bigg) \\
        = & O\bigg(\sum_{M=1}^{2n-m+k}\frac{\partial}{\partial z_{M'}}\bigg(e^{-z_M}\bigg(\sum_{\substack{S\in\langle S_{K_1},\ldots,S_{K_M}\rangle\\S\notin\langle S_{K_1},\ldots,S_{K_{M-1}}\rangle}}\sum_{\substack{i=1\\T_i=T_0}}^p\rho_S^ie^{-(S-S_{K_M})}\bigg)\circ\Psi^{-1}\bigg)\bigg).
    \end{align*}
    By definition of the isomorphism $\Psi$, the function:
    \[
        \bigg(\sum_{\substack{S\in\langle S_{K_1},\ldots,S_{K_M}\rangle\\S\notin\langle S_{K_1},\ldots,S_{K_{M-1}}\rangle}}\sum_{\substack{i=1\\T_i=T_0}}^p\rho_S^ie^{-(S-S_{K_M})}\bigg)\circ\Psi^{-1}
    \]
    only depends on the variables $\mathbf{s}_m,z_1,\ldots,z_M$. Moreover this function and its derivatives are bounded, thus we have:
    \[
        \sum_{M=1}^{2n-m+k}\frac{\partial}{\partial z_{M'}}\bigg(e^{-z_M}\bigg(\sum_{\substack{S\in\langle S_{K_1},\ldots,S_{K_M}\rangle\\S\notin\langle S_{K_1},\ldots,S_{K_{M-1}}\rangle}}\sum_{\substack{i=1\\T_i=T_0}}^p\rho_S^ie^{-(S-S_{K_M})}\bigg)\circ\Psi^{-1}\bigg) = O\bigg(\sum_{M=M'}^{2n-m+k}e^{-z_M}\bigg)= O(e^{-z_{M'}}).
    \]
    We can also write:
    \[
        e^{-\frac{1}{2}(T_0-T_i)}\sqrt{\alpha_{T_i}^i+\sum_{\substack{T\in\widetilde{G}_i\\ T\neq T_i}}\alpha_T^ie^{-(T_i-T)}} = \sqrt{\alpha_{T_i}^ie^{-2S_0}+e^{-S_0}\sum_{S\in\widetilde{G}_0}\sigma_S^ie^S},
    \]
    where $S_0\in\widetilde{G}_0$ and the $\sigma_S^i$'s are either $0$ or positive and bounded functions that only depend on the variable $\mathbf{s}_m$. It follows:
    \begin{align*}
        \frac{\partial}{\partial z_{M'}}\bigg[\bigg(e^{-\frac{1}{2}(T_0-T_i)}\sqrt{\alpha_{T_i}^i+\sum_{\substack{T\in\widetilde{G}_i\\ T\neq T_i}}\alpha_T^ie^{-(T_i-T)}} & \bigg)\circ\Psi^{-1}\bigg] \\ = & O\bigg(\frac{\partial}{\partial z_{M'}}\Big(e^{-S_0\circ\Psi^{-1}}\Big)+\sum_{S\in\widetilde{G}_0}\sigma_S^i\frac{\partial}{\partial z_{M'}}\Big(e^{-S\circ\Psi^{-1}}\Big)\bigg).
    \end{align*}
    By repeating the same process as above, we conclude that:
    \[
        \frac{\partial}{\partial z_{M'}}\bigg[\bigg(e^{-\frac{1}{2}(T_0-T_i)}\sqrt{\alpha_{T_i}^i+\sum_{\substack{T\in\widetilde{G}_i\\ T\neq T_i}}\alpha_T^ie^{-(T_i-T)}} \bigg)\circ\Psi^{-1}\bigg]=O(e^{-z_{M'}}),
    \]
    hence $\frac{\partial}{\partial z_{M'}}\big(\psi_0(\mathbf{s}_m,\Psi^{-1}(\mathbf{z}_{2n-m+k}))\big)=O(e^{-z_{M'}})$. This makes us able to conclude that $\widetilde{H}$ is a simplified length-type function.
\end{proof}

\begin{Rem}
    If we allow the coefficients $a_T^i$ to have an arbitrary sign in the definition of length-type functions, then the coefficients $a_T^i$ would also have arbitrary signs in the expression (\ref{S3:Theo:SimpFormPLT:Demo1}). This would imply that the coefficients $\alpha_S^i$ defined by (\ref{S3:Theo:SimpFormPLT:Demo2}) can vanish at infinity. This, in turn, would imply that the leading term of the function $\widetilde{H}_0$ defined by (\ref{S3:Theo:SimpFormPLT:Demo3}) may depend on the variable $\mathbf{s}_m$, and that the linear changes of variables $\Psi$ we consider may impact the variable $\mathbf{s}_m$.
\end{Rem}

Theorem \ref{S3:Theo:SimpFormPLT} provides simplified length-type functions with trivial conditions on the linear forms $Y_1,\ldots,Y_k$. The reason why these linear forms are in Definition \ref{S3:Def:SLT} is to ensure that simplified length-type functions are stable under restriction on the boundary of their domain, as stated in the following:

\begin{Lem}\label{S3:Lem:RestBoundSLT}
    Let $h\colon \mathbf{R}_{>0}^m\times C\longrightarrow\mathbf{R}_{\geq0}$ be a simplified length-type function. The restriction of $h$ to $\mathbf{R}^m\times\partial C$ is a simplified length-type function in the following sense. For all $i\in\{1,\ldots,n\}$, there exist piecewise affine functions $X_i^0,X_i^1\colon\mathbf{R}^{n-1}\longrightarrow\mathbf{R}$ and an admissible polytope $C_i\subset\mathbf{R}^{n-1}$ such that:
    \[
        C=\Big\{\mathbf{x}_n\in\mathbf{R}^n\Big|(x_j)_{j\neq i}\in C_i, X_i^0((x_j)_{j\neq i})\prec x_i\prec X_i^1((x_j)_{j\neq i}) \Big\},
    \]
    where each occurrence of $\prec$ is either $<$ or $\leq$, depending on the choice of the inequalities used to define $C$, and when restricted to each admissible polytope on which $X_i^0$ (resp. $X_i^1$) is affine, then the function $(\mathbf{s}_m,(x_j)_{j\neq i})\mapsto h(\mathbf{s}_m,x_1,\ldots,x_{i-1},X_i^0((x_j)_{j\neq i}),x_{i+1},\ldots,x_n)$ (resp. the function $(\mathbf{s}_m,(x_j)_{j\neq i})\mapsto h(\mathbf{s}_m,x_1,\ldots,x_{i-1},X_i^0((x_j)_{j\neq i}),x_{i+1},\ldots,x_n)$) is a simplified length-type function.
\end{Lem}

\begin{proof}
    Since $C$ is an admissible polytope, there exist linear forms $T_1,\ldots,T_N$ on $\mathbf{R}^n$ and constants $b_1,\ldots,b_N\in\mathbf{R}$ such that:
    \[
        C=\{\mathbf{x}_n\in\mathbf{R}_{\geq0}^n|\forall q\in\{1,\ldots,N\}, T_q(\mathbf{x}_n)\succ b_q\}.
    \]
    Given $q\in\{1,\ldots,N\}$, let us denote $T_q\colon\mathbf{x}_n\mapsto a_{q,1}x_1+\cdots+a_{q,n}x_n$. Let:
    \begin{align*}
        I_0 = & \{q~|~a_{q,i}=0\} \\
        I_+ = & \{q~|~a_{q,i}>0\} \\
        I_- = & \{q~|~a_{q,i}<0\}.
    \end{align*}
    If $q\in I_0$, the linear form $T_q$ does not depend on the variable $x_i$, hence we denote $T_q(\mathbf{x}_n)=T_q((x_j)_{j\neq i})$ for such $q$. We denote:
    \begin{align*}
        X_i^0((x_j)_{j\neq i}) = & \max_{q\in I_+}\bigg\{\frac{b_q-a_{q,1}x_1-\cdots-a_{q,i-1}x_{i-1}-a_{q,i+1}x_{i+1}-\cdots-a_{q,n}x_n}{a_{q,i}}\bigg\} \\
        X_i^1((x_j)_{j\neq i}) = & \min_{q\in I_-}\bigg\{\frac{b_q-a_{q,1}x_1-\cdots-a_{q,i-1}x_{i-1}-a_{q,i+1}x_{i+1}-\cdots-a_{q,n}x_n}{a_{q,i}}\bigg\}.
    \end{align*}
    We finally note:
    \[
        C_i=\Big\{(x_j)_{j\neq i}\in\mathbf{R}_{\geq0}^{n-1}\Big|\forall q\in I_0,T_q((x_j)_{j\neq i})\succ b_q\Big\}\cap\Big\{(x_j)_{j\neq i}\in\mathbf{R}_{\geq0}^{n-1}\Big|0\prec X_i^0((x_j)_{j\neq i})\prec X_i^1((x_j)_{j\neq i})\Big\}.
    \]
    Since $C\subset\{0\leq x_1\leq \cdots\leq x_n\}$ we can assume without loss of generality that the linear forms $\mathbf{x}_n\mapsto x_1$ and $\mathbf{x}_n\mapsto x_{j+1}-x_j$ belong to $\{T_1,\ldots,T_N\}$ with associated constant $b_q=0$, hence we have $C_i\subset\{0\leq x_1\leq \cdots\leq x_{i-1}\leq x_{i+1}\leq \cdots\leq x_n\}$ and $C_i$ is an admissible polytope.
    
    With these notations, we have:
    \[
        C=\{\mathbf{x}_n\in\mathbf{R}_{\geq0}^n|(x_j)_{j\neq i}\in C_i,0\prec X_i^0((x_j)_{j\neq i})\prec x_i\prec X_i^1((x_j)_{j\neq i})\}.
    \]
    Let us now denote $X_i^0=Z_i^0+b_i^0$ and $X_i^1=Z_i^1+b_i^1$, where $Z_i^0$ and $Z_i^1$ are piecewise linear and $b_i^0$ and $b_i^1$ are piecewise constant, and let:
    \begin{align*}
        & h_0\colon (\mathbf{s}_m,(x_j)_{j\neq i})\mapsto h(\mathbf{s}_m,x_1,\ldots,x_{i-1},X_i^0((x_j)_{j\neq i}),x_{i+1},\ldots,x_n) \\
        & h_1\colon (\mathbf{s}_m,(x_j)_{j\neq i}\mapsto h(\mathbf{s}_m,x_1,\ldots,x_{i-1},X_i^1((x_j)_{j\neq i}),x_{i+1},\ldots,x_n).
    \end{align*}
    Let $C_1^0,\ldots,C_a^0\subset C_0$ be the subcones on which $X_i^0$ is affine and $C_1^1,\ldots,C_b^1\subset C_0$ be the subcones on which $X_i^1$ is affine.
    
    The linear part of $h_0$ on the cones $C_r^0$, for $r\in\{1,\ldots,a\}$, is given by:
    \[
        \Lambda_0\colon(\mathbf{s}_m,(x_j)_{j\neq i})\longmapsto\Lambda(\mathbf{s}_m,x_1,\ldots,x_{i-1},Z_i^0((x_j)_{j\neq i}),x_{i+1},\ldots,x_n).
    \]
    Let us note:
    \[
        \psi_0\colon(\mathbf{s}_m,(x_j)_{j\neq i})\longmapsto\Lambda(\mathbf{0}_m,0,\ldots,0,b_i^0,0,\ldots,0)+\psi(\mathbf{s}_m,x_1,\ldots,x_{i-1},X_i^0((x_j)_{j\neq i}),x_{i+1},\ldots,x_n).
    \]
    With these notations we have $h_0=\Lambda_0+\psi_0$, and $||\psi_0||_\infty<\infty$.
    
    Since $h$ is a simplified length-type function, there exist linear forms $Y_1,\ldots,Y_k$ and a constant $c>0$ such that for every $p\in\{1,\ldots,n\}$ and every $(\mathbf{s}_m,\mathbf{x}_n)\in \mathbf{R}_{>0}^m\times C$, we have:
    \[
        \Big|\frac{\partial\psi(\mathbf{s}_m,\mathbf{x}_n)}{\partial x_p}\Big|\leq ce^{-\min\{x_p,Y_1(\mathbf{x}_n),\ldots,Y_k(\mathbf{x}_n)\}},
    \]
    and if $p_0=\max\{p|\frac{\partial\Lambda(\mathbf{s}_m,\mathbf{x}_n)}{\partial x_{p_0}}\neq0\}$, there exists a constant $M>0$ such that for all $\mathbf{x}_n\in C$, we have:
    \[
        \Lambda(\mathbf{0}_m,\mathbf{x}_n)\leq M\min\{x_{p_0},Y_1(\mathbf{x}_n),\ldots,Y_k(\mathbf{x}_n)\}.
    \]
    On the one hand when denoting $Y_d^0((x_j)_{j\neq i})=Y_d(x_1,\ldots,x_{i-1},Z_i^0((x_j)_{j\neq i}),x_{i+1},\ldots,x_n)$ and $\widetilde{Y}_d^0((x_j)_{j\neq i})=Y_d(x_1,\ldots,x_{i-1},X_i^0((x_j)_{j\neq i}),x_{i+1},\ldots,x_n)$ for $d\in\{1,\ldots,k\}$, if $i=p_0$, for all $(x_j)_{j\neq i}\in C_i$ we have:
    \[
        \Lambda_0(\mathbf{0}_m,(x_j)_{j\neq i})\leq M\min\{X_i^0((x_j)_{j\neq i}),\widetilde{Y}_1^0((x_j)_{j\neq i}),\ldots,\widetilde{Y}_k^0((x_j)_{j\neq i})\}.
    \]
    Since an admissible polytope is stable under $\mathbf{x}_n\mapsto t\mathbf{x}_n$ with $t\geq1$, we have:
    \[
        \Lambda_0(\mathbf{0}_m,(x_j)_{j\neq i})\leq M\min\{Z_i^0((x_j)_{j\neq i}),Y_1^0((x_j)_{j\neq i}),\ldots,Y_k^0((x_j)_{j\neq i})\},
    \]
    and if $i\neq p_0$, we similarly have:
    \[
        \Lambda_0(\mathbf{0}_m,(x_j)_{j\neq i})\leq M\min\{x_{p_0},Y_1^0((x_j)_{j\neq i}),\ldots,Y_k^0((x_j)_{j\neq i})\}.
    \]
    Denoting $p_1=\max\{p\neq i|\frac{\partial\Lambda_0}{\partial x_p}\neq0\}$, by taking a possibly greater $M$ for all $(x_j)_{j\neq i}\in C_r^0$ we have:
    \[
        \Lambda_0(\mathbf{0}_m,(x_j)_{j\neq i})\leq Mx_{p_1}.
    \]
    On the other hand we have for all $p\neq i$ and all $(\mathbf{s}_m,(x_j)_{j\neq i})\in\mathbf{R}_{>0}^m\times(C_r^0\cap(A,+\infty)^{n-1})$:
    \begin{align*}
        \Big|\frac{\partial\psi_0}{\partial x_p}(\mathbf{s}_m,(x_j)_{j\neq i})\Big| = & \bigg|\frac{\partial \psi}{\partial x_p}(\mathbf{s}_m,x_1,\ldots,x_{i-1},X_i^0((x_j)_{j\neq i}),x_{i+1},\ldots,x_n) \\
        & + \frac{\partial X_i^0}{\partial x_p}\frac{\partial \psi}{\partial x_i}(\mathbf{s}_m,x_1,\ldots,x_{i-1},X_i^0((x_j)_{j\neq i}),x_{i+1},\ldots,x_n)\bigg|\\
        \leq & c\Big(1+\Big|\frac{\partial X_i^0}{\partial x_p}\Big|\Big)e^{-\min\{x_p,X_i^0((x_j)_{j\neq i}),\widetilde{Y}_1^0((x_j)_{j\neq i}),\ldots,\widetilde{Y}_k^0((x_j)_{j\neq i})\}} \\
        \leq & c'\Big(1+\Big|\frac{\partial X_i^0}{\partial x_p}\Big|\Big)e^{-\min\{x_p,Z_i^0((x_j)_{j\neq i}),Y_1^0((x_j)_{j\neq i}),\ldots,Y_k^0((x_j)_{j\neq i})\}}.
    \end{align*}
    Hence $h_0$ is a simplified length-type function on each $C_r^0$. By replacing $X_i^0$ by $X_i^1$ and $Z_i^0$ by $Z_i^1$, we also conclude that $h_1$ is a simplified length-type function on each $C_r^1$.
\end{proof}

\begin{Lem}\label{S3:Lem:PropLinForm}
    Let $C$ be an $n$-dimensional admissible polytope and $h\colon \mathbf{R}_{>0}^m\times C\longrightarrow\mathbf{R}$ be a simplified length-type function. Denote by $\Lambda$ the linear part of $h$. If $T$ is a linear form on $\mathbf{R}^n$, there exists a constant $\mu>0$ such that for all $\mathbf{x}_n\in C$, we have $\Lambda(\mathbf{0}_m,\mathbf{x}_n)\geq\mu T(\mathbf{x}_n)$.
\end{Lem}

\begin{proof}
    Let $\Bar{C}$ be the closure of $C$, and let us fix $a_0>0$ such that $\Bar{C}\cap\{\Lambda(\mathbf{0}_m,\mathbf{x}_n)=a_0\}\neq\emptyset$. Since this set is compact, we can define:
    \[
        \nu=\frac{1}{a_0}\max\big\{T(\mathbf{x}_n) | \mathbf{x}_n\in\Bar{C}\cap\{\Lambda(\mathbf{0}_m,\mathbf{x}_n)=a_0\}\big\}.
    \]
    Let $\mathbf{x}_n\in C$ and $a=\Lambda(\mathbf{0}_m,\mathbf{x}_n)>0$. We have $\Lambda(\mathbf{0}_m,\frac{a_0}{a}\mathbf{x}_n)=a_0$, hence $T(\frac{a_0}{a}\mathbf{x}_n)\leq a_0\nu$. If $\nu\leq 0$, then $T\leq 0$ on $C$ in which case the conclusion is trivial, and if $\nu>0$ we conclude by letting $\mu=\nu^{-1}$.
\end{proof}

\section{Pseudo-convolution of polynomials with respect to length-type functions}\label{S4}

Using Theorem \ref{S2:Theo:IntFormGeneral}, we can compute the Weil-Petersson expectation $\mathbb{E}[F^\gamma]$ of a geometric random variable as a linear combination of terms of the form:
\[
    \int_{\mathbf{R}_{>0}^k\times\mathbf{R}^k\times\mathbf{R}_{\geq0}^N}F\circ h_\gamma(\mathbf{x}_k,\mathbf{t}_k,\mathbf{y}_N,\mathbf{L}_n)y_1^{K_1}\cdots y_N^{K_N}\d\mathbf{x}_k\d\mathbf{t}_k\d\mathbf{y}_N.
\]
We therefore obtain the existence of the function $V_{\mathcal{O}_\gamma}(~.~,\mathbf{L}_n)$ of our main result by disintegrating the measure $y_1^{K_1}\cdots y_N^{K_N}\d\mathbf{x}_k\d\mathbf{t}_k\d\mathbf{y}_N$ along the level sets of the length function $h_\gamma$ (written in appropriate Fenchel-Nielsen coordinates). The function $h_\gamma$ is a length-type function, hence according to Theorem \ref{S3:Theo:SimpFormPLT}, we can approximate it with a piecewise linear function. If we disintegrate the measure $y_1^{K_1}\cdots y_N^{K_N}\d\mathbf{x}_k\d\mathbf{t}_k\d\mathbf{y}_N$ along the level sets of a linear form, we recover (up to rescaling) the usual convolution of $2k$ times 1 and $y_1^{K_1},\ldots,y_N^{K_N}$ on a subcone of $\mathbf{R}^{2k+N}$. Our strategy is to exploit this analogy to obtain our estimates on the function $V_{\mathcal{O}_\gamma}(~.~,\mathbf{L}_n)$. Our main tool is the \emph{pseudo-convolution}, developed by Anantharaman and Monk in \cite{AnMo25}.

\subsection{Disintegration and pseudo-convolution}

Let $h\colon\mathbf{R}_{\geq0}^n\longrightarrow\mathbf{R}$ be a measurable function such that for each compact subset $K\subset\mathbf{R}$, the set $h^{-1}(K)$ has finite measure, and $f\in L_{loc}^1(\mathbf{R}_{\geq0}^n)$. When integrating $f$ on the level set $h^{-1}(\ell)$ defined by the implicit relation $\ell=h(\mathbf{x}_n)$, we denote by $\frac{\d\mathbf{x}_n}{\d\ell}$ the disintegration of the Lebesgue measure on the level sets of $h$. As an example, if $\gamma$ is a closed loop on $S_{g,n}$ and $F\colon\mathbf{R}\longrightarrow\mathbf{C}$ is a test function, using the same notations as in Theorem \ref{S2:Theo:IntFormGeneral} we can therefore write:
\[
    \mathbb{E}[F^\gamma]=\int_\mathbf{R} F(\ell)\bigg(\int_{\ell=h_\gamma(\mathbf{x}_k,\mathbf{t}_k,\mathbf{y}_N,\mathbf{L}_n)}y_1\cdots y_NV_\Gamma(\mathbf{y}_N,\mathbf{L}_n)\frac{\d\mathbf{x}_k\d\mathbf{t}_k\d\mathbf{y}_N}{\d\ell}\bigg)\frac{\d\ell}{V_{g,n}(\mathbf{L}_n)}.
\]

\begin{Def}
    Let $C\subset\mathbf{R}_{\geq0}^n$ be a measurable set with non-empty interior. If $1\leq i\leq n$, we denote:
    \[
        \hat{C}_i=\{(x_j)_{j\neq i}\in\mathbf{R}_{\geq0}^{n-1}|\exists x_i\in\mathbf{R}_{\geq0},(x_1,\ldots,x_n)\in C\}.
    \]
    Let $h\colon C\longrightarrow\mathbf{R}$ be a proper map. We say that $h$ satisfies Assumption $(\mathbf{H}_n^i)$ if $h$ is $\mathscr{C}^1$ and if for any $(x_j)_{j\neq i}\in \hat{C}_i$, the function:
    \[
        h_i=h_i^{(x_j)_{j\neq i}}\colon x_i\longmapsto h(\mathbf{x}_n)
    \]
    is a $\mathscr{C}^1$-diffeomorphism from an interval of the form $(X_0((x_j)_{j\neq i}),X_1((x_j)_{j\neq i}))$ onto its image, where $X_1((x_j)_{j\neq i})$ may be infinite. Its inverse will then be denoted by:
    \[
        \ell\mapsto h_i^{-1}(\ell,(x_j)_{j\neq i}).
    \]
\end{Def}

Under Assumption $(\mathbf{H}_n^i)$, we can compute the expression of the disintegration of the Lebesgue measure on the level sets of $h$.

\begin{Prop}\label{S4:Prop:DensDesint}
    Let $h\colon \mathbf{R}_{>0}^n\longrightarrow\mathbf{R}$ be a function satisfying Assumption $(\mathbf{H}_n^i)$. For any $\ell$ belonging in the image of $h$, the function $(x_j)_{j\neq i}\mapsto h_i^{-1}(\ell,(x_j)_{j\neq i})$ is defined on:
    \[
        U_\ell^i=\{(x_j)_{j\neq i}\in\mathbf{R}_{>0}^{n-1}|h_i^{(x_j)_{j\neq i}}(X_0((x_j)_{j\neq i}))<\ell<h_i^{(x_j)_{j\neq i}}(X_1((x_j)_{j\neq i}))\}
    \]
    if $h$ is an increasing function of $x_i$ and:
    \[
        U_\ell^i=\{(x_j)_{j\neq i}\in\mathbf{R}_{>0}^{n-1}|h_i^{(x_j)_{j\neq i}}(X_1((x_j)_{j\neq i}))<\ell<h_i^{(x_j)_{j\neq i}}(X_0((x_j)_{j\neq i}))\}
    \]
    otherwise. Moreover if $f\in L_{loc}^1(\mathbf{R}^n)$, we have for almost every $\ell\in\mathbf{R}$:
    \begin{equation}\label{S4:Eq:ExpDesint}
        \int_{\ell=h(\mathbf{x}_n)}f(\mathbf{x}_n)\frac{\d\mathbf{x}_n}{\d\ell}=\int_{U_\ell^i}f\Big(x_1,\ldots,x_{i-1},h_i^{-1}\big(\ell,(x_j)_{j\neq i}\big),x_{i+1},\ldots,x_n\Big)\Big|\frac{\partial h_i^{-1}}{\partial\ell}\big(\ell,(x_j)_{j\neq i}\big)\Big|\prod_{j\neq i}\d x_j.
    \end{equation}
\end{Prop}

\begin{Lem}[Change of variables]\label{S4:Lem:ChangeVar}
    Let $C\subset\mathbf{R}^n$ be an open set, $h\colon C\longrightarrow\mathbf{R}$ be a proper and measurable function and $f\in L_{loc}^1(\mathbf{R}^n)$.
    \begin{enumerate}[label=\textup{(\roman*)}]
        \item\label{S4:Lem:ChangeVar:Item1} Let $\Phi\colon U\longrightarrow C$ be a $\mathscr{C}^1$-diffeomorphism, we have:
        \[
            \int_{\ell=h(\mathbf{y}_n)}f(\mathbf{y}_n)\mathds{1}_C(\mathbf{y}_n)\frac{\d\mathbf{y}_n}{\d\ell}=\int_{\ell=h\circ\Phi(\mathbf{x}_n)}f\circ\Phi(\mathbf{x}_n)|J_\Phi(\mathbf{x}_n)|\mathds{1}_U(\mathbf{x}_n)\frac{\d\mathbf{x}_n}{\d\ell},
        \]
        where $J_\Phi$ is the determinant of the Jacobian matrix of $\Phi$.
        \item\label{S4:Lem:ChangeVar:Item2} Let $u\colon \mathbf{R}\longrightarrow\mathbf{R}$ be a $\mathscr{C}^1$-diffeomorphism. Let us denote $H=u\circ h$ and:
        \[
            F(\ell)=\int_{\ell=h(\mathbf{x}_n)}f(\mathbf{x}_n)\mathds{1}_C(\mathbf{x}_n)\frac{\d\mathbf{x}_n}{\d\ell},
        \]
        and
        \[
            G(\ell)=\int_{\ell=H(\mathbf{x}_n)}f(\mathbf{x}_n)\mathds{1}_C(\mathbf{x}_n)\frac{\d\mathbf{x}_n}{\d\ell}.
        \]
        We have $F=|u'|G\circ u$.
    \end{enumerate}
\end{Lem}

\begin{proof}
    The point \ref{S4:Lem:ChangeVar:Item1} immediately comes from the change of variables formula in $\mathbf{R}^n$. To prove the point \ref{S4:Lem:ChangeVar:Item2}, let us consider a test function $\varphi\colon \mathbf{R}\longrightarrow\mathbf{R}$. We have on a one hand, by definition:
    \[
        \int_C\varphi(h(\mathbf{x}_n))f(\mathbf{x}_n)\d\mathbf{x}_n = \int_\mathbf{R}\varphi(\ell)F(\ell)\d\ell,
    \]
    and on the other hand, we have:
    \begin{align*}
        \int_C\varphi(h(\mathbf{x}_n))f(\mathbf{x}_n)\d\mathbf{x}_n = & \int_C\varphi\circ u^{-1}(H(\mathbf{x}_n))f(\mathbf{x}_n)\d\mathbf{x}_n \\
        = & \int_\mathbf{R}\varphi\circ u^{-1}(\ell)G(\ell)\d\ell \\
        = & \int_\mathbf{R}\varphi(L)|u'(L)|G\circ u(L)\d L,
    \end{align*}
    which makes us able to conclude.
\end{proof}

The \emph{pseudo-convolution} is an operation defined by Anantharaman and Monk in \cite{AnMo25}, which generalizes the usual convolution by disintegrating measures on level sets of non-linear maps.

\begin{Def}\cite{AnMo25}
    Let $\varphi:\mathbf{R}^n\longrightarrow\mathbf{R}$ be a measurable and locally bounded function and $h$ a measurable function defined on the support of $\varphi$ such that for every compact subset $K\subset\mathbf{R}$, the set $h^{-1}(K)\cap\{\varphi\neq0\}$ has finite measure.

    Let $f_1,\ldots,f_n\in L^1_{loc}(\mathbf{R})$. The \emph{$(h,\varphi)$-convolution of $f_1,\ldots,f_n$} is the pushforward measure $\mu$ of $\varphi\times\bigotimes_{i=1}^nf_i\d x_i$ by the map $h$, i.e. for every test function $g$:
    \[
        \int_\mathbf{R}g\d\mu = \int_{\mathbf{R}^n}g\big(h(x_1,\ldots,x_n)\big)\varphi(x_1,\ldots,x_n)f_1(x_1)\cdots f_n(x_n)\d x_1\cdots\d x_n.
    \]
    When the measure $\mu$ admits a density with respect to the Lebesgue measure, we will denote for almost every $\ell\in\mathbf{R}$:
    \[
        \frac{\d\mu}{\d\ell}(\ell)=f_1\star\cdots\star f_n(\ell)\Big|_\varphi^h.
    \]
\end{Def}

\begin{Rem}
    Under Assumption $(\mathbf{H}_n^i)$, we have:
    \begin{align*}
        f_1\star\cdots\star f_n(\ell)\Big|_\varphi^h=\int_{U_\ell^i} & f_i\Big(h_i^{-1}\big(\ell,(x_j)_{j\neq i}\big)\Big)\prod_{j\neq i}f_j(x_j) \\
        & \times\varphi\Big(x_1,\ldots,x_{i-1},h_i^{-1}\big(\ell,(x_j)_{j\neq i}\big),x_{i+1},\ldots,x_n\Big)\Big|\frac{\partial h_i^{-1}}{\partial\ell}\big(\ell,(x_j)_{j\neq i}\big)\Big|\prod_{j\neq i}\d x_j.
    \end{align*}
\end{Rem}

When re-expressing our goal in terms of pseudo-convolution, we have to obtain appropriate estimates on the pseudo-convolution of polynomials with respect to length-type functions.

\subsection{Friedman-Ramanujan functions}

The asymptotic estimates we want to prove on the density function $V_{\mathcal{O}_\gamma}(~.~,\mathbf{L}_n)$ can be reformulated by saying that $V_{\mathcal{O}_\gamma}(~.~,\mathbf{L}_n)$ belongs to the class $\mathcal{SF}_w(\lambda)$, for some $\lambda\in(0,1)$. This class of functions is closely related to the class of \emph{Friedman-Ramanujan functions}.

\begin{Def}
    Let $\lambda\in(0,1)$. A continuous function $f\colon \mathbf{R}\longrightarrow\mathbf{C}$ belongs to the class $\mathcal{R}(\lambda)$ if there exist constants $c,c_1>0$ such that:
    \[
        \forall\ell\geq0, |f(\ell)|\leq c_1(1+\ell)^ce^{\lambda\ell}.
    \]
    We denote by $\mathcal{R}_w(\lambda)$ the space of locally integrable functions $f\colon \mathbf{R}\longrightarrow\mathbf{C}$ such that for some constants $c,c_1>0$:
    \[
        \forall m\geq1,\int_0^m|f(\ell)|\d\ell\leq c_1(1+m)^ce^{\lambda m}.
    \]
    The space of \emph{$\lambda$-Friedman-Ramanujan functions} is the space $\mathcal{F}(\lambda)=e^\ell\mathbf{C}[\ell]\oplus\mathcal{R}(\lambda)$. The \emph{degree} of a Friedman-Ramanujan function is the degree of its polynomial part.

    We similarly define the space of \emph{weak $\lambda$-Friedman-Ramanujan functions} to be the space $\mathcal{F}_w(\lambda)=e^\ell\mathbf{C}[\ell]\oplus\mathcal{R}_w(\lambda)$.

    If $m\geq0$, we denote $\mathcal{F}^m(\lambda)=e^\ell\mathbf{C}_{m-1}[\ell]\oplus\mathcal{R}(\lambda)$ and $\mathcal{F}_w^m(\lambda)=\mathbf{C}_{m-1}[\ell]\oplus\mathcal{R}_w(\lambda)$.

    We denote $\mathcal{SF}(\lambda)=e^{-\ell}\mathcal{F}(1-\lambda)$ and $\mathcal{SF}^m(\lambda)=e^{-\ell}\mathcal{F}^{m+1}(1-\lambda)$. We similarly define the spaces $\mathcal{SF}_w(\lambda)$ and $\mathcal{SF}_w^m(\lambda)$.

    We finally denote $\mathcal{SR}(\lambda)=e^{-\ell}\mathcal{R}(1-\lambda)$ and $\mathcal{SR}_w(\lambda)=e^{-\ell}\mathcal{R}_w(1-\lambda)$.
\end{Def}

\begin{Rem}
    This definition extends the class of Friedman-Ramanujan functions defined by Anantharaman and Monk in \cite{AnMo24}. A Friedman-Ramanujan function as defined by Anantharaman and Monk corresponds to a $\frac{1}{2}$-Friedman-Ramanujan with this new definition. We can also notice that a $\lambda$-Friedman-Ramanujan function is the same thing as an $(e+1)$-Ramanujan function of order $e^\lambda$, as defined by Friedman in \cite{Fri03}, but with continuous variable.
\end{Rem}

Anantharaman and Monk characterized $\frac{1}{2}$-Friedman-Ramanujan functions in \cite{AnMo24}. We extend their characterization to arbitrary Friedman-Ramanujan functions.

\begin{Def}
    We denote by $\mathcal{P}$ the operator consisting in taking the primitive vanishing at 0, and $\mathcal{L}=\Id-\mathcal{P}$. When considering functions of several variables, we shall denote by $\mathcal{P}_{x_i}$ and $\mathcal{L}_{x_i}$ when we apply these operators with respect to the variable $x_i$ if there is an ambiguity on the variable.

    If $f$ is defined on an interval of the form $[x_0,+\infty)$ with $x_0>0$, we extend it to be 0 on $[0,x_0)$, and $\mathcal{P}f$ is the primitive vanishing at 0 of this extension of $f$.
\end{Def}

\begin{Lem}[Characterization of Friedman-Ramanujan functions]\label{S4:Lem:CaracFR}
    Let $\lambda\in(0,1)$ and $m\geq0$. The function $f$ belongs to $\mathcal{F}^m(\lambda)$ (respectively $\mathcal{F}_w^m(\lambda)$) if and only if $\mathcal{L}^mf\in\mathcal{R}(\lambda)$ (respectively $\mathcal{R}_w(\lambda)$).
\end{Lem}

\begin{proof}
    We proceed similarly to \cite{AnMo24}. The operator $\mathcal{L}$ maps $\mathcal{F}^m(\lambda)$ and $\mathcal{F}_w^m(\lambda)$ to $\mathcal{F}^{m-1}(\lambda)$ and $\mathcal{F}_w^{m-1}(\lambda)$ respectively, hence by induction if $f\in\mathcal{F}^m(\lambda)$ (resp. $f\in\mathcal{F}_w^m(\lambda)$), we have $\mathcal{L}^mf\in\mathcal{F}^0(\lambda)=\mathcal{R}(\lambda)$ (resp. $\mathcal{L}^mf\in\mathcal{R}_w(\lambda)$).
    
    Now if $\mathcal{L}f=p+r$ with $p\in e^\ell\mathbf{C}_{m-1}[\ell]$ and $r\in\mathcal{R}(\lambda)$ or $r\in\mathcal{R}_w(\lambda)$, we prove the result by induction on $m$. If $g=\mathcal{P}f$, then $g$ is the solution to $g'-g=\mathcal{L}f=p+r$. Hence there exists $C\in\mathbf{R}$ such that:
    \[
        \forall t\geq0,g(t)=Ce^t+\int_t^\infty r(s)e^{t-s}\d s+\int_0^tp(s)e^{t-s}\d s.
    \]
    As a direct consequence:
    \[
        f(t) = Ce^t+p(t)+r(t)+\int_t^\infty r(s)e^{t-s}\d s+\int_0^tp(s)e^{t-s}\d s,
    \]
    and the estimates on $r$ make us able to conclude that $f\in\mathcal{F}^{m+1}(\lambda)$ or $f\in\mathcal{F}_w^{m+1}(\lambda)$. We can therefore conclude by induction.
\end{proof}

In order to show that a function belongs to $\mathcal{SF}(\lambda)$, we shall multiply it by $e^\ell$ and use Lemma \ref{S4:Lem:CaracFR}.
\[
    \xymatrix@!0 @R=2cm @C=1.5cm{
    \mathcal{SF}^m(\lambda) \ar[rr] \ar[rd]_{\times e^\ell} & & \mathcal{R}(1-\lambda) \\
    & \mathcal{F}^{m+1}(1-\lambda) \ar[ru]_{\mathcal{L}^{m+1}} & 
    }
\]

\subsection{Convolution with respect to simplified length-type functions}

We introduced the class of simplified length-type functions to facilitate our computations in order to prove that the pseudo-convolution of polynomials with respect to length-type functions belongs to the class $\mathcal{SF}_w(\lambda)$. Hence we begin by proving that the pseudo-convolution of polynomials with respect to simplified length-type functions is $\mathcal{SF}_w(\lambda)$.

\subsubsection{Computations in one variable}

We assume in this paragraph that $x\mapsto h(x)$ is a $\mathscr{C}^1$ diffeomorphism from $(0,+\infty)$ to $(h_0,+\infty)$ and we denote by $h\mapsto x(h)$ its inverse.

\begin{Lem}\label{S4:Lem:OpeLOneVar}
    Let $0<x_0<x_1\leq\infty$ and let $f\colon \mathbf{R}\longrightarrow\mathbf{R}$ be a $\mathscr{C}^1$ function. We have:
    \begin{align*}
        \mathcal{L}_h\Big[e^hf\circ x(h)\mathds{1}_{(x_0,x_1)}\circ x(h)\Big] =  \sign(h')\bigg( & e^{h(x_0)}f(x_0)\mathds{1}_{(h(x_0),+\infty)}(h)-e^{h(x_1)}f(x_1)\mathds{1}_{(h(x_1),+\infty)}(h) \\ & + \mathcal{P}_h\Big[e^hf'\circ x(h)\Big|\frac{\partial x}{\partial h}\Big|\mathds{1}_{(x_0,x_1)}\circ x(h)\Big]\bigg).
    \end{align*}
\end{Lem}

\begin{proof}
    Let us assume that $h$ is increasing (the computations are very similar if $h$ is decreasing). We have:
    \begin{align*}
        \mathcal{P}_h\Big[e^h f\circ x(h)\mathds{1}_{(x_0,x_1)}\circ x(h)\Big] = & \mathds{1}_{(x_0,x_1)}\circ x(h)\int_{h(x_0)}^he^tf\circ x(t)\d t+\mathds{1}_{(x_1,+\infty)}\circ x(h)\int_{h(x_0)}^{h(x_1)}e^tf\circ x(t)\d t \\
        = & \mathds{1}_{(x_0,x_1)}\circ x(h)\bigg(e^hf\circ x(h)-e^{h(x_0)}f(x_0)-\int_{h(x_0)}^he^tf'\circ x(t)\frac{\partial x}{\partial h}(t)\d t\bigg) \\
        & + \mathds{1}_{(x_1,+\infty)}\circ x(h)\bigg(e^{h(x_1)}f(x_1)-e^{h(x_0)}f(x_0)-\int_{h(x_0)}^{h(x_1)}e^tf'\circ x(t)\frac{\partial x}{\partial h}(t)\d t\bigg) \\
        = & \mathds{1}_{(x_0,x_1)}\circ x(h)e^hf\circ x(h) - \mathds{1}_{(x_0,+\infty)}\circ x(h)e^{h(x_0)}f(x_0) \\
        & + \mathds{1}_{(x_1,+\infty)}\circ x(h)e^{h(x_1)}f(x_1) - \mathcal{P}_h\Big[\mathds{1}_{(x_0,x_1)}\circ x(h)e^hf'\circ x(h)\frac{\partial x}{\partial h}(h)\Big].
    \end{align*}
    Hence we conclude by using the fact that $\mathcal{L}_h=\Id-\mathcal{P}_h$.
\end{proof}

\subsubsection{Computations is several variables}

The key result on which we will rely is the following:

\begin{Theo}\label{S4:Theo:OpeLConvCone}
    Let $C$ be an $n$-dimensional admissible polytope, $h\colon\mathbf{R}_{>0}^m\times C\longrightarrow\mathbf{R}$ be a simplified length-type function and $\varphi\in L^\infty(\mathbf{R}_{>0}^m)$. Let $A\geq0$ and $i_0\in\{1,\ldots,n\}$ such that $h$ satisfies Assumption $(\mathbf{H}_{m+n}^{m+i_0})$ on $\mathbf{R}_{>0}^m\times C^{A,i}$ where $C^{A,i}=C\cap[0,A]^i\times(A,+\infty)^{n-i}$. There exist $(n-1)$-dimensional admissible polytopes $C_1,\ldots,C_p$, simplified length-type functions $H_1,\ldots,H_p$, linear forms $Z_1,\ldots,Z_p$ on $\mathbf{R}^{n-1}$, constants $b_1,\ldots,b_p\in\mathbf{R}$ and $1\leq q\leq p$ such that $H_k$ is defined on $\mathbf{R}_{>0}^m\times C_k$ for all $k\in\{1,\ldots,p\}$, and for any real number\footnote{The constant $\eta$ will be chosen later as a well-suited coefficient of the linear part of $h$.} $\eta\in\mathbf{R}^*$ and any $g_1,\ldots,g_m\in L^\infty(\mathbf{R}_{>0})$ and $f_1,\ldots,f_n\in\mathscr{C}^1(\mathbf{R})$, we have:
    \[
        \mathcal{L}_\ell\bigg[e^\ell\Big(g_1\star\cdots\star g_m\star f_1\star\cdots\star f_n(\ell)\Big|_{\varphi\mathds{1}_{\mathbf{R}_{>0}^m\times C^{A,i}}}^h\Big)\bigg] = \frac{1}{\eta}\bigg[\xi_{i_0}(\ell)+\chi_\eta(\ell)+\sum_{k=1}^q\theta_k(\ell)-\sum_{k=q+1}^p\theta_k(\ell)\bigg],
    \]
    where:
    \begin{align*}
        \xi_{i_0}(\ell) = \mathcal{P}_\ell\bigg[e^\ell\Big(g_1\star\cdots\star g_m\star f_1\star & \cdots\star f_{i_0-1}\star f_{i_0}'\star f_{i_0+1}\star\cdots\star f_n(\ell)\Big|_{\varphi\mathds{1}_{\mathbf{R}_{>0}^m\times C^{A,i}}}^h\Big)\bigg] \\
        \chi_\eta(\ell) = \mathcal{L}_\ell\bigg[e^\ell\Big(g_1\star\cdots\star g_m\star f_1\star & \cdots\star f_n(\ell)\Big|_{(\eta-\frac{\partial h}{\partial x_{i_0}})\varphi\mathds{1}_{\mathbf{R}_{>0}^m\times C^{A,i}}}\Big)\bigg]\\
        \theta_k(\ell) = \mathcal{P}_\ell\bigg[e^\ell\int_{\ell=H_k(\mathbf{s}_m,\mathbf{x}_{n-1})} f_{i_0}\Big( & Z_k(\mathbf{x}_{n-1})+b_k\Big)\prod_{j=1}^mg_j(s_j)\prod_{\substack{r=1\\r\neq i_0}}^nf_r(x_r) \\ 
        & \times\varphi(\mathbf{s}_m)\mathds{1}_{\mathbf{R}_{>0}^m\times (C_k\cap[0,A]^i\times(A,+\infty)^{n-1-i})}(\mathbf{s}_m,\mathbf{x}_{n-1})\frac{\d\mathbf{s}_m\d\mathbf{x}_{n-1}}{\d\ell}\bigg].
    \end{align*}
\end{Theo}

\begin{proof}
    According to Lemma \ref{S3:Lem:RestBoundSLT}, there exist an admissible polytope $C_0\subset\mathbf{R}^{n-1}$ and piecewise affine functions $X_0,X_1\colon\mathbf{R}^{n-1}\longrightarrow\mathbf{R}$ such that:
    \[
        C=\Big\{\mathbf{x}_n\in\mathbf{R}_{\geq0}^n\Big|(x_r)_{r\neq i_0}\in C_0,X_0((x_r)_{r\neq i_0})\prec x_{i_0}\prec X_1((x_r)_{r\neq i_0})\Big\}.
    \]
    We denote by $C_1,\ldots,C_q\subset C_0$ the admissible polytopes on which $X_0$ is affine and $C_{q+1},\ldots,C_p\subset C_0$ the admissible poytopes on which $X_1$ is affine. For each $k\in\{1,\ldots,q\}$, we denote by $Z_k$ the linear part of $X_0$ on $C_k$ and $b_k$ its constant part, and if $k\in\{q+1,\ldots,p\}$, $Z_k$ and $b_k$ are respectively the linear and constant parts of $X_1$ on $C_k$. Lemma \ref{S3:Lem:RestBoundSLT} ensures that each
    \[
        H_k\colon(\mathbf{s}_m,(x_r)_{r\neq i_0})\longmapsto h(\mathbf{s}_m,x_1,\ldots,x_{i_0-1},Z_k((x_r)_{r\neq i_0})+b_k,x_{i_0+1},\ldots,x_n)
    \]
    is a simplified length-type function on $\mathbf{R}_{>0}^m\times C_k$.
    
    For convenience of notations we now assume that $i_0=n$. According to Proposition \ref{S4:Prop:DensDesint}, we have:
    \begin{align*}
        g_1\star\cdots\star g_m\star f_1\star\cdots\star f_n(\ell)\Big|_{\varphi\mathds{1}_{\mathbf{R}_{>0}^m\times C^{A,i}}}^h = & \int_{U_\ell^{m+n}} \prod_{j=1}^m g_j(s_j)\prod_{r=1}^{n-1}f_r(x_r)f_n\Big(h_{m+n}^{-1}(\ell,\mathbf{s}_m,\mathbf{x}_{n-1})\Big) \\
        & \times\varphi(\mathbf{s}_m)\mathds{1}_{\mathbf{R}_{>0}^m\times C^{A,i}}(\mathbf{s}_m,\mathbf{x}_{n-1})\Big|\frac{\partial h_{m+n}^{-1}}{\partial\ell}(\ell,\mathbf{s}_m,\mathbf{x}_{n-1})\Big|\frac{\d\mathbf{s}_m\d\mathbf{x}_{n-1}}{\d\ell}.
    \end{align*}
    The function $h$ satisfies $(\mathbf{H}_{m+n}^{m+n})$, hence $\frac{\partial h_{n+m}^{-1}}{\partial\ell}\neq0$. For all $\eta\neq0$, we have:
    \[
        \Big|\frac{\partial h_{m+n}^{-1}}{\partial\ell}(\ell,\mathbf{s}_m,\mathbf{x}_{n-1})\Big|=\frac{1}{\eta}\bigg(\sign\Big(\frac{\partial h}{\partial x_n}\Big)+\Big|\frac{\partial h_{m+n}^{-1}}{\partial\ell}(\ell,\mathbf{s}_m,\mathbf{x}_{n-1})\Big|\Big(\eta-\frac{\partial h}{\partial x_n}\big(\mathbf{s}_m,\mathbf{x}_{n-1},h_{m+n}^{-1}(\ell,\mathbf{s}_m,\mathbf{x}_{n-1})\big)\Big)\bigg).
    \]
    Let $\varepsilon=\sign(\frac{\partial h}{\partial x_n})$. It follows that $e^\ell\Big(g_1\star\cdots\star g_m\star f_1\star\cdots\star f_n(\ell)\Big|_{\varphi\mathds{1}_{\mathbf{R}_{>0}^m\times C^{A,i}}}^h\Big)$ equals:
    \begin{align}
        \nonumber \varepsilon\frac{e^\ell}{\eta}\int_{U_\ell^{m+n}}\prod_{j=1}^mg_j(\mathbf{s}_m)\prod_{r=1}^{n-1}f_r & (x_r)f_n\Big(h_{m+n}^{-1}(\ell,\mathbf{s}_m,\mathbf{x}_{n-1})\Big)\\
        \label{S4:Theo:OpeLConvCone:Demo1} & \times\varphi(\mathbf{s}_m)\mathds{1}_{\mathbf{R}_{>0}^m\times C^{A,i}}\Big(\mathbf{s}_m,\mathbf{x}_{n-1},h_{m+n}^{-1}(\ell,\mathbf{s}_m,\mathbf{x}_{n-1})\Big)\d\mathbf{s}_m\d\mathbf{x}_{n-1} \\
        \label{S4:Theo:OpeLConvCone:Demo2} + \frac{e^\ell}{\eta}\int_{\ell=h(\mathbf{s}_m,\mathbf{x}_n)}\prod_{j=1}^mg_j(s_j)\prod_{r=1}^n & f_r(x_r)\Big(\eta-\frac{\partial h}{\partial x_n}(\mathbf{s}_m,\mathbf{x}_n)\Big)\varphi(\mathbf{s}_m)\mathds{1}_{\mathbf{R}_{>0}^m\times C^{A,i}}(\mathbf{s}_m\mathbf{x}_n)\frac{\d\mathbf{s}_m\d\mathbf{x}_n}{\d\ell}.
    \end{align}
    When applying the operator $\mathcal{L}_\ell$, the term (\ref{S4:Theo:OpeLConvCone:Demo2}) gives the term $\frac{1}{\eta}\chi_\eta(\ell)$. Let us now focus on the term (\ref{S4:Theo:OpeLConvCone:Demo1}). Let $C_0^{A,i}=C_0\cap[0,A]^i\times(A,+\infty)^{n-1-i}$. When applying Lemma \ref{S4:Lem:OpeLOneVar}, we obtain:
    \begin{align}
        \nonumber\mathcal{L}_\ell\bigg[\varepsilon\int_{U_\ell^{m+n}}e^\ell & \prod_{j=1}^mg_j(\mathbf{s}_m)\prod_{r=1}^{n-1}f_r(x_r)f_n\Big(h_{m+n}^{-1}(\ell,\mathbf{s}_m,\mathbf{x}_{n-1})\Big)\\
        \nonumber & \times\varphi(\mathbf{s}_m)\mathds{1}_{\mathbf{R}_{>0}^m\times C^{A,i}}\Big(\mathbf{s}_m,\mathbf{x}_{n-1},h_{m+n}^{-1}(\ell,\mathbf{s}_m,\mathbf{x}_{n-1})\Big)\d\mathbf{s}_m\d\mathbf{x}_{n-1}\bigg] \\
        \nonumber =\int_{U_\ell^{m+n}} & \prod_{j=1}^mg_j(s_j)\prod_{r=1}^{n-1}f_r(x_r)\varphi(\mathbf{s}_m)\mathds{1}_{\mathbf{R}_{>0}^m\times C^{A,i}_0}(\mathbf{s}_m,\mathbf{x}_{n-1})\\
        \nonumber & \times\varepsilon\mathcal{L}_\ell\Big[e^\ell f_n\circ h_{m+n}^{-1}(\ell,\mathbf{s}_m,\mathbf{x}_{n-1})\mathds{1}_{(X_0(\mathbf{x}_{n-1}),X_1(\mathbf{x}_{n-1}))}\circ h_{m+n}^{-1}(\ell,\mathbf{s}_m,\mathbf{x}_{n-1})\Big]\d\mathbf{s}_m\d\mathbf{x}_{n-1} \\
        \nonumber =\int_{U_\ell^{m+n}} & \prod_{j=1}^mg_j(s_j)\prod_{r=1}^{n-1}f_r(x_r)\varphi(\mathbf{s}_m)\mathds{1}_{\mathbf{R}_{>0}^m\times C^{A,i}_0}(\mathbf{s}_m,\mathbf{x}_{n-1}) \\
        \label{S4:Theo:OpeLConvCone:Demo3} & \times e^{h(\mathbf{s}_m,\mathbf{x}_{n-1},X_0(\mathbf{x}_{n-1}))}f_n(X_0(\mathbf{x}_{n-1}))\mathds{1}_{\{h(\mathbf{s}_m,\mathbf{x}_{n-1},X_0(\mathbf{x}_{n-1}))<\ell\}}\d\mathbf{s}_m\d\mathbf{x}_{n-1}\\
        \nonumber- \int_{U_\ell^{m+n}} & \prod_{j=1}^mg_j(s_j)\prod_{r=1}^{n-1}f_r(x_r)\varphi(\mathbf{s}_m)\mathds{1}_{\mathbf{R}_{>0}^m\times C^{A,i}_0}(\mathbf{s}_m,\mathbf{x}_{n-1}) \\
        \label{S4:Theo:OpeLConvCone:Demo4} & \times e^{h(\mathbf{s}_m,\mathbf{x}_{n-1},X_1(\mathbf{x}_{n-1}))}f_n(X_1(\mathbf{x}_{n-1}))\mathds{1}_{\{h(\mathbf{s}_m,\mathbf{x}_{n-1},X_1(\mathbf{x}_{n-1}))<\ell\}}\d\mathbf{s}_m\d\mathbf{x}_{n-1} \\
        \nonumber+ \int_{U_\ell^{m+n}} & \prod_{j=1}^mg_j(s_j)\prod_{r=1}^{n-1}f_r(x_r)\varphi(\mathbf{s}_m)\mathds{1}_{\mathbf{R}_{>0}^m\times C^{A,i}_0}(\mathbf{s}_m,\mathbf{x}_{n-1})\\
        \label{S4:Theo:OpeLConvCone:Demo5} & \times \mathcal{P}_\ell\Big[e^\ell f_n'\circ h_{m+n}^{-1}\Big|\frac{\partial h_{m+n}^{-1}}{\partial\ell}(\ell,\mathbf{s}_m,\mathbf{x}_{n-1})\Big|\mathds{1}_{(X_0(\mathbf{x}_{n-1}),X_1(\mathbf{x}_{n-1}))}\circ h_{m+n}^{-1}(\ell,\mathbf{s}_m,\mathbf{x}_{n-1})\Big]\d\mathbf{s}_m\d\mathbf{x}_{n-1}.
    \end{align}
    Thus, by disintegrating along the level sets of the functions $H_1,\ldots,H_p$, the terms (\ref{S4:Theo:OpeLConvCone:Demo3}) and (\ref{S4:Theo:OpeLConvCone:Demo4}) respectively give:
    \begin{align}
        \nonumber \sum_{k=1}^q\int_0^\ell e^t\bigg(\int_{t=H_k(\mathbf{s}_m,\mathbf{x}_{n-1})}f_n & \Big(X_0(\mathbf{x}_{n-1})\Big)\prod_{j=1}^mg_j(s_j)\prod_{r=1}^{n-1}f_r(x_r) \\
        \label{S4:Theo:OpeLConvCone:Demo6}& \times\varphi(\mathbf{s}_m)\mathds{1}_{\mathbf{R}_{>0}^m\times C^{A,i}_k}(\mathbf{s}_m,\mathbf{x}_{n-1})\frac{\d\mathbf{s}_m\d\mathbf{x}_{n-1}}{\d t}\bigg)\d t,
    \end{align}
    and:
    \begin{align}
        \nonumber-\sum_{k=q+1}^p\int_0^\ell e^t\bigg(\int_{t=H_k(\mathbf{s}_m,\mathbf{x}_{n-1})}f_n & \Big(X_1(\mathbf{x}_{n-1})\Big)\prod_{j=1}^mg_j(s_j)\prod_{r=1}^{n-1}f_r(x_r) \\
        \label{S4:Theo:OpeLConvCone:Demo7} & \times\varphi(\mathbf{s}_m)\mathds{1}_{\mathbf{R}_{>0}^m\times C^{A,i}_k}(\mathbf{s}_m,\mathbf{x}_{n-1})\frac{\d\mathbf{s}_m\d\mathbf{x}_{n-1}}{\d t}\bigg)\d t.
    \end{align}
    The terms (\ref{S4:Theo:OpeLConvCone:Demo6}) and (\ref{S4:Theo:OpeLConvCone:Demo7}) therefore give the term $\sum_{k=1}^q\theta_k(\ell)-\sum_{k=q+1}^p\theta_k(\ell)$, and (\ref{S4:Theo:OpeLConvCone:Demo5}) gives:
    \[
        \mathcal{P}_\ell\bigg[\int_{\ell=h(\mathbf{s}_m,\mathbf{x}_n)}\prod_{j=1}^mg_j(s_j)\prod_{r=1}^{n-1}f_r(x_r)f_n'(x_n)\varphi(\mathbf{s}_m)\mathds{1}_{\mathbf{R}_{>0}^m\times C^{A,i}}(\mathbf{s}_m,\mathbf{x}_n)\frac{\d\mathbf{s}_m\d\mathbf{x}_n}{\d\ell}\bigg].
    \]
    We conclude by putting these terms together.
\end{proof}

When considering polynomial functions $f_1,\ldots,f_n$, the conclusion of Theorem \ref{S4:Theo:OpeLConvCone} makes us able to develop the following plan.
\begin{enumerate}[label=\textup{(\roman*)}]
    \item By choosing appropriately the index $i_0$ and the constant $\eta$, the term $\eta-\frac{\partial h}{\partial x_{i_0}}$ in the definition of $\chi_\eta(\ell)$ will compensate the contribution of $e^\ell$ and we will therefore have $\chi_\eta\in\mathcal{R}_w(\lambda)$ for some $\lambda\in(0,1)$ (Lemma \ref{S4:Lem:ConvSLT1}).
    \item If we apply the operator $\mathcal{L}_\ell$ enough times, the derivatives in $\xi_{i_0}$ will eventually vanish. The only remaining term will have an expression similar to $\theta_k(\ell)$ (Lemma \ref{S4:Lem:ConvSLT2}).
    \item The $\theta_k$'s can be written as linear combinations of pseudo-convolutions of polynomials with respect to simplified length-type functions defined on products of $\mathbf{R}_{>0}^m$ and $(n-1)$-dimensional admissible polytopes. We therefore conclude by induction on $n$ that the pseudo-convolution of polynomials with respect to a simplified length-type function is a function of class $\mathcal{SF}_w(\lambda)$, for some $\lambda\in(0,1)$ (Corollary \ref{S4:Cor:ConvSLTUndbound}).
\end{enumerate}

\begin{Rem}
    If we consider a function $\varphi$ that depends on the variables $\mathbf{s}_m$ and $\mathbf{x}_n$ in Theorem \ref{S4:Theo:OpeLConvCone} instead of a function only depending on the variable $\mathbf{s}_m\in\mathbf{R}_{>0}^m$, then we would have:
    \begin{equation}\label{S4:Eq:ProbArbSign}
        \mathcal{L}_\ell\Big[e^\ell\Big(g_1\star\cdots\star g_m\star f_1\star\cdots\star f_n(\ell)\Big|_{\varphi\mathds{1}_{\mathbf{R}_{>0}^m\times C^{A,i}}}\Big)\Big] =  \frac{1}{\eta}\bigg[\widetilde{\xi}_{i_0}(\ell)+\chi_\eta(\ell)+\sum_{k=1}^q\theta_k(\ell)-\sum_{k=q+1}^p\theta_k(\ell)\bigg],
    \end{equation}
    where:
    \[
        \widetilde{\xi}_{i_0}(\ell)=\xi_{i_0}(\ell)+e^\ell\Big(g_1\star\cdots\star g_m\star f_1\star\cdots\star f_n(\ell)\Big|_{\frac{\partial\varphi}{\partial x_{i_0}}\mathds{1}_{\mathbf{R}_{>0}^m\times C^{A,i}}}\Big).
    \]
    The addition of this term would however make us unable to perform the induction described above, which would be problematic for the proof of Theorem \ref{S4:Theo:ConvLenType} if we allow the length-type functions to be defined with coefficients of arbitrary signs.
\end{Rem}

\begin{Lem}\label{S4:Lem:ConvSLT1}
    Let $C$ be an $n$-dimensional admissible polytope, $i\in\{0,\ldots,n\}$ and $h\colon\mathbf{R}_{>0}^m\times C\longrightarrow\mathbf{R}$ be a simplified length-type function. We note $\varphi\colon\mathbf{s}_m\in\mathbf{R}_{>0}^m\mapsto e^{-(s_1+\cdots+s_m)}$. Let $\Lambda$ be the linear part of $h$ and $i_0=\max\{r\in\{1,\ldots,n\}|\frac{\partial\Lambda}{\partial x_r}\neq0\}$.
    If $i_0>i$, there exist $A\geq0$, $\eta\neq0$ and $\lambda\in(0,1)$ such that $h$ satisfies Assumption $(\mathbf{H}_{m+n}^{m+i_0})$ on $\mathbf{R}_{>0}^m\times C^{A,i}$ and for every $g_1,\ldots,g_m\in L^\infty(\mathbf{R}_{>0})$ and for every polynomial functions $f_1,\ldots,f_n$, we have:
    \[
        g_1\star\cdots\star g_m\star f_1\star\cdots\star f_n\Big|_{(\eta-\frac{\partial h}{\partial x_{i_0}})\varphi\mathds{1}_{\mathbf{R}_{>0}^m\times C^{A,i}}}^h\in\mathcal{SR}_w(\lambda).
    \]
\end{Lem}

\begin{proof}
    We note:
    \[
        B(\ell) = g_1\star\cdots\star g_m\star f_1\star\cdots\star f_n(\ell)\Big|_{(\eta-\frac{\partial h}{\partial x_{i_0}})\varphi\mathds{1}_{\mathbf{R}_{>0}^m\times C^{A,i}}}^h.
    \]
    Let us denote $h=\Lambda+\psi$, where $\Lambda$ is linear, and $\eta=\frac{\partial\Lambda}{\partial x_{i_0}}$ (which is not necessarily positive). Since $h$ is a simplified length-type function, there exists a constant $c_0>0$ such that:
    \[
        \Lambda-c_0\leq h\leq \Lambda+c_0,
    \]
    and there exist linear forms $Y_1,\ldots,Y_k$ on $\mathbf{R}^n$ and constants $c,M>0$ such that for all $r\in\{1,\ldots,n\}$ and all $(\mathbf{s}_m,\mathbf{x}_{n-m})\in\mathbf{R}_{>0}^m\times C$, we have:
    \begin{align*}
        \Big|\frac{\partial\psi(\mathbf{s}_m,\mathbf{x}_n)}{\partial x_r}\Big| & \leq ce^{-\min\{x_r,Y_1(\mathbf{x}_n),\ldots,Y_k(\mathbf{x}_n)\}} \\
        \Lambda(\mathbf{0}_m,\mathbf{x}_n) & \leq M\min\{x_{i_0},Y_1(\mathbf{x}_n),\ldots,Y_k(\mathbf{x}_n)\}.
    \end{align*}
    According to Lemma \ref{S3:Lem:PropLinForm}, there exists a constant $\mu>0$ such that for all $r\in\{1,\ldots,n\}$ and all $\mathbf{x}_n\in C$, we have $\Lambda(\mathbf{0}_m,\mathbf{x}_n)\geq\mu x_r$. Hence if $A\geq0$ and $\mathbf{x}_n\in C^{A,i}$, we have:
    \[
        \min\{x_{i_0},Y_1(\mathbf{x}_n),\ldots,Y_k(\mathbf{x}_n)\}\geq \frac{1}{M}\Lambda(\mathbf{0}_m,\mathbf{x}_n)\geq\frac{\mu}{M}A.
    \]
    Thus for $A\geq0$ large enough, the function $h$ satisfies Assumption $(\mathbf{H}_{m+n}^{m+i_0})$ on $\mathbf{R}_{>0}^m\times C^{A,i}$ (and more generally, $h$ satisfies Assumption $(\mathbf{H}_{m+n}^{m+r})$ as soon as $r>i$ and $\frac{\partial\Lambda}{\partial x_r}\neq0$).
    
    Let $a\geq0$, we have:
    \[
        \int_0^ae^\ell|B(\ell)|\d\ell \leq e^{c_0}\prod_{j=1}^m||g_j||_\infty \int_{\mathbf{R}_{>0}^m\times C^{A,i}}\mathds{1}_{\{h(\mathbf{s}_m,\mathbf{x}_n)\leq a\}}e^{\Lambda(\mathbf{s}_m,\mathbf{x}_n)}\varphi(\mathbf{s}_m)\Big|\frac{\partial\psi(\mathbf{s}_m,\mathbf{x}_n)}{\partial x_{i_0}}\Big|\prod_{r=1}^n|f_r(x_r)|\d\mathbf{s}_m\d\mathbf{x}_n,
    \]
    and:
    \[
        e^{\Lambda(\mathbf{s}_m,\mathbf{x}_n)}\varphi(\mathbf{s}_m)\Big|\frac{\partial\psi(\mathbf{s}_m,\mathbf{x}_n)}{\partial x_{i_0}}\Big| \leq e^{\Lambda(\mathbf{s}_m,\mathbf{x}_n)-(s_1+\cdots+s_m+\min\{x_{i_0},Y_1(\mathbf{x}_n),\ldots,Y_k(\mathbf{x}_n)\})}.
    \]
    We have:
    \begin{align*}
        \Lambda(\mathbf{s}_m,\mathbf{x}_n)-(s_1+\cdots+s_m+\min\{x_{i_0},Y_1(\mathbf{x}_n),\ldots,Y_k(\mathbf{x}_n)\}) = &  \Lambda(\mathbf{s}_m,\mathbf{0}_n)-(s_1+\cdots+s_m) \\
        & + \Lambda(\mathbf{0}_m,\mathbf{x}_n)-\min\{x_{i_0},Y_1(\mathbf{x}_n),\ldots,Y_k(\mathbf{x}_n)\}.
    \end{align*}
    On the one hand, let:
    \[
        D=\max_{1\leq j\leq m}\frac{\partial\Lambda}{\partial s_j}.
    \]
    If $m=0$ we note $D=0$, in which case we have $\Lambda(\mathbf{s}_m,\mathbf{0}_n)-(s_1+\cdots+s_m)=0$. If $m>0$ we have $D>0$ and for all $\mathbf{s}_m\in\mathbf{R}_{>0}^m$ we have:
    \[
        \Lambda(\mathbf{s}_m,\mathbf{0}_n) \leq D(s_1+\cdots+s_m),
    \]
    hence:
    \[
        \Lambda(\mathbf{s}_m,\mathbf{0}_n)-(s_1+\cdots+s_m)\leq (1-\lambda_1)\Lambda(\mathbf{s}_m,\mathbf{0}_n),
    \]
    with $\lambda_1=\frac{1}{D}$.
    
    On the other hand, for all $\mathbf{x}_n\in C$, we have:
    \[
        \Lambda(\mathbf{0}_m,\mathbf{x}_n)\leq M\min\{x_{i_0},Y_1(\mathbf{x}_k),\ldots,Y_k(\mathbf{x}_n)\},
    \]
    hence:
    \[
        \Lambda(\mathbf{0}_m,\mathbf{x}_n)-\min\{x_{i_0},Y_1(\mathbf{x}_k),\ldots,Y_k(\mathbf{x}_n)\} \leq (1-\lambda_2)\Lambda(\mathbf{0}_m,\mathbf{x}_n),
    \]
    with $\lambda_2=\frac{1}{M}$.
    
    Let $\lambda=\min\{\lambda_1,\lambda_2\}$. By possibly taking a larger $M>0$, we can assume that $\lambda\in(0,1)$ and we have:
    \[
        \Lambda(\mathbf{s}_m,\mathbf{x}_n)-(s_1+\cdots+s_m+\min\{x_{i_0},Y_1(\mathbf{x}_n),\ldots,Y_k(\mathbf{x}_n)\}) \leq (1-\lambda)\Lambda(\mathbf{s}_m,\mathbf{x}_n).
    \]
    In addition to this, Lemma \ref{S3:Lem:PropLinForm} provides a constant $\mu>0$ such that for all $(\mathbf{s}_m,\mathbf{x}_n)\in\mathbf{R}_{>0}^m\times C$, we have $\Lambda(\mathbf{s}_m,\mathbf{x}_n)\geq s_1+\cdots+s_m+\mu x_n$. This yields:
    \[
        \int_0^ae^\ell|B(\ell)|\d\ell\leq ce^{c_0+(1-\lambda)(a+c_0)}\prod_{j=1}^m||g_j||_\infty\int_{\mathbf{R}_{>0}^m\times C^{A,i}}\mathds{1}_{\{s_1+\cdots+s_m+\mu x_n\leq a+c_0\}}\prod_{r=1}^n|f_r(x_r)|\d\mathbf{s}_m\d\mathbf{x}_n.
    \]
    Thus there exist constants $K,c_1>0$ such that:
    \[
        \int_0^ae^\ell|B(\ell)|\d\ell\leq K(1+a)^{c_1}e^{(1-\lambda)a},
    \]
    which proves that $B\in\mathcal{SR}_w(\lambda)$.
\end{proof}

\begin{Lem}\label{S4:Lem:ConvSLT2}
    Let $C$ be an $n$-dimensional admissible polytope, $i\in\{0,\ldots,n\}$ and $h\colon\mathbf{R}_{>0}^m\times C\longrightarrow\mathbf{R}_{\geq0}$ be a simplified length-type function. We denote $\varphi\colon\mathbf{s}_m\in\mathbf{R}_{>0}^m\mapsto e^{-(s_1+\cdots+s_m)}$. Let $i_0\in\{1,\ldots,n\}$, $A\geq0$ and $\eta\neq0$ be chosen as in Lemma \ref{S4:Lem:ConvSLT1} and $C_1,\ldots,C_p$, $H_1,\ldots,H_p$, $Z_1,\ldots,Z_p$, $b_1,\ldots,b_p$ and $q\leq p$ be chosen as in Theorem \ref{S4:Theo:OpeLConvCone}. Let $g_1,\ldots,g_m\in L^\infty(\mathbf{R}_{>0})$ and let $f_1,\ldots,f_n$ be polynomial functions. For $d\geq0$, we denote:
    \begin{align*}
        \Theta_{i_0}^d(\ell) =  \sum_{k=1}^qe^\ell\int_{\ell=H_k(\mathbf{s}_m,\mathbf{x}_{n-1})}f_{i_0}^{(d)}\Big(Z_k(\mathbf{x}_{n-1})+ & b_k\Big)\prod_{j=1}^mg_j(s_j)\prod_{\substack{r=1\\r\neq i_0}}^nf_r(x_r) \\
        & \times\varphi(\mathbf{s}_m)\mathds{1}_{\mathbf{R}_{>0}^m\times C^{A,i}_k}(\mathbf{s}_m,\mathbf{x}_{n-1})\frac{\d\mathbf{s}_m\d\mathbf{x}_{n-1}}{\d\ell}\\
        -\sum_{k=q+1}^pe^\ell\int_{\ell=H_k(\mathbf{s}_m,\mathbf{x}_{n-1})}f_{i_0}^{(d)}\Big(Z_k(\mathbf{x}_{n-1}) & +b_k\Big)\prod_{j=1}^mg_j(s_j)\prod_{\substack{r=1\\r\neq i_0}}^nf_r(x_r) \\
        & \times\varphi(\mathbf{s}_m)\mathds{1}_{\mathbf{R}_{>0}^m\times C^{A,i}_k}(\mathbf{s}_m,\mathbf{x}_{n-1})\frac{\d\mathbf{s}_m\d\mathbf{x}_{n-1}}{\d\ell}.
    \end{align*}
    If $K_{i_0}=\deg f_{i_0}$, there exist $\lambda\in(0,1)$, which depends only on $h$, and $\rho\in\mathcal{R}_w(\lambda)$ such that for all $\ell\geq0$, we have:
    \begin{align*}
        \mathcal{L}_\ell^{K_{i_0}+1}\bigg[e^\ell\Big(g_1\star\cdots\star g_m\star f_1\star\cdots\star f_n(\ell)\Big|_{\varphi\mathds{1}_{\mathbf{R}_{>0}^m\times C^{A,i}}}^h\Big)\bigg] = \sum_{d=0}^{K_{i_0}}\frac{1}{\eta^{d+1}}\mathcal{L}_\ell^{K_{i_0}-d}\mathcal{P}_\ell^{d+1}\Theta_{i_0}^d(\ell)+\rho(\ell).
    \end{align*}
\end{Lem}

\begin{proof}
    We prove by induction on $K\geq0$ that there exists $\lambda\in(0,1)$ such that for any $K$, there exists $\rho_K\in\mathcal{R}_w(\lambda)$ with:
    \begin{equation}\label{S4:Lem:ConvSLT2:Demo}
        \begin{aligned}
            \mathcal{L}_\ell^{K+1}\bigg[e^\ell & \Big(g_1\star\cdots\star g_m\star f_1\star\cdots\star f_n(\ell)\Big|_{\varphi\mathds{1}_{\mathbf{R}_{>0}^m\times C^{A,i}}}^h\Big)\bigg] \\
            = & \frac{1}{\eta^{K+1}}\mathcal{P}_\ell^{K+1}\bigg[e^\ell\Big(g_1\star\cdots\star g_m\star f_1\star\cdots\star f_{i_0-1}\star f_{i_0}^{(K+1)}\star f_{i_0+1}\star\cdots\star f_n(\ell)\Big|_{\varphi\mathds{1}_{\mathbf{R}_{>0}^m\times C^{A,i}}}^h\Big)\bigg] \\
            & + \sum_{d=0}^K\frac{1}{\eta^{d+1}}\mathcal{L}_\ell^{K-d}\mathcal{P}_\ell^{d+1}\Theta_{i_0}^k(\ell) + \rho_K(\ell).
        \end{aligned}
    \end{equation}
    According to Theorem \ref{S4:Theo:OpeLConvCone}, this result is true for $K=0$. Now if we assume the result to be true for some $K\geq0$, according to Theorem \ref{S4:Theo:OpeLConvCone} we have:
    \begin{align*}
        \mathcal{L}_\ell^{K+2}\bigg[e^\ell & \Big(g_1\star\cdots\star g_m\star f_1\star\cdots\star f_n(\ell)\Big|_{\varphi\mathds{1}_{\mathbf{R}_{>0}^m\times C^{A,i}}}^h\Big)\bigg] \\
        = & \frac{1}{\eta^{K+1}}\mathcal{P}_\ell^{K+1}\mathcal{L}_\ell\bigg[e^\ell\Big(g_1\star\cdots\star g_m\star f_1\star\cdots\star f_{i_0-1}\star f_{i_0}^{(K+1)}\star f_{i_0+1}\star\cdots\star f_n(\ell)\Big|_{\varphi\mathds{1}_{\mathbf{R}_{>0}^m\times C^{A,i}}}^h\Big)\bigg] \\
        & + \sum_{d=0}^K\frac{1}{\eta^{d+1}}\mathcal{L}_\ell^{K+1-d}\mathcal{P}_\ell^{d+1}\Theta_{i_0}^k(\ell) + \mathcal{L}_\ell \rho_K(\ell) \\
        = & \frac{1}{\eta^{K+2}}\mathcal{P}_\ell^{K+2}\bigg[e^\ell\Big(g_1\star\cdots\star g_m\star f_1\star\cdots\star f_{i_0-1}\star f_{i_0}^{(K+2)}\star f_{i_0+1}\star\cdots\star f_n(\ell)\Big|_{\varphi\mathds{1}_{\mathbf{R}_{>0}^m\times C^{A,i}}}^h\Big) \\
        & + e^\ell\Big(g_1\star\cdots\star g_m\star f_1\star\cdots\star f_{i_0-1}\star f_{i_0}^{(K+1)}\star f_{i_0+1}\star\cdots\star f_n(\ell)\Big|_{(\eta-\frac{\partial h}{\partial x_i})\varphi\mathds{1}_{\mathbf{R}_{>0}^m\times C^{A,i}}}^h\Big)\bigg] \\
        & + \frac{1}{\eta^{K+2}}\mathcal{P}_\ell^{K+2}\Theta_{i_0}^{K+1}(\ell) + \sum_{k=0}^K\frac{1}{\eta^{d+1}}\mathcal{L}_\ell^{K+1-d}\mathcal{P}_\ell^{d+1}\Theta_{i_0}^d(\ell) + \mathcal{L}_\ell \rho_K(\ell)
    \end{align*}
    According to Lemma \ref{S4:Lem:ConvSLT1}, we have:
    \[
        \ell\mapsto e^\ell\Big(g_1\star\cdots\star g_m\star f_1\star\cdots\star f_{i_0-1}\star f_{i_0}^{(K+1)}\star f_{i_0+1}\star\cdots\star f_n(\ell)\Big|_{(\eta-\frac{\partial h}{\partial x_{i_0}})\varphi\mathds{1}_{\mathbf{R}_{>0}^m\times C^{A,i}}}^h\Big)\in\mathcal{R}_w(\lambda),
    \]
    hence if we define $\rho_{K+1}$ to be:
    \[
        \mathcal{L}_\ell \rho_K+\frac{1}{\eta^{K+2}}\mathcal{P}_\ell^{K+2}\bigg[e^\ell\Big(g_1\star\cdots\star g_m\star f_1\star\cdots\star f_{i_0-1}\star f_{i_0}^{(K+1)}\star f_{i_0+1}\star\cdots\star f_n(\ell)\Big|_{(\eta-\frac{\partial h}{\partial x_{i_0}})\varphi\mathds{1}_{\mathbf{R}_{>0}^m\times C^{A,i}}}^h\Big)\bigg],
    \]
    the formula (\ref{S4:Lem:ConvSLT2:Demo}) is true. We prove the result by taking $K=K_{i_0}$.
\end{proof}

\begin{Cor}\label{S4:Cor:ConvSLTUndbound}
    Let $C$ be an $n$-dimensional admissible polytope, $i\in\{0,\ldots,n\}$ and $h\colon\mathbf{R}_{>0}^m\times C\longrightarrow\mathbf{R}$ be a simplified length-type function. Let $\varphi\colon\mathbf{s}_m\in\mathbf{R}_{>0}^m\longrightarrow e^{-(s_1+\cdots+s_m)}$. There exists $\lambda\in(0,1)$ such that for every $g_1,\ldots,g_m\in L^\infty(\mathbf{R}_{>0})$ and every polynomial functions $f_1,\ldots,f_n$ of degree $K_1,\ldots,K_n$ respectively, we have:
    \[
        g_1\star\cdots\star g_m\star f_1\star\cdots\star f_n\Big|_{\varphi\mathds{1}_{\mathbf{R}_{>0}^m\times C^{A,i}}}^h\in\mathcal{SF}_w^{K_{i+1}+\cdots+K_n+n-i-1}(\lambda).
    \]
\end{Cor}

\begin{proof}
    We denote:
    \[
        B(\ell)=g_1\star\cdots\star g_m\star f_1\star\cdots\star f_n(\ell)\Big|_{\varphi\mathds{1}_{\mathbf{R}_{>0}^m\times C^{A,i}}}^h
    \]
    We prove the result by induction on $n-i$. Denote $h=\Lambda+\psi$, where $\Lambda$ is linear. If $n=i$ and $m=0$, then $h$ is bounded on $C^{A,i}$ and the pseudo-convolution $B$ is compactly supported. Else, if $m>0$, then $\mathbf{x}_n\mapsto\Lambda(\mathbf{0}_m,\mathbf{x}_n)$ is bounded on $C^{A,n}=C\cap[0,A]^n$. Hence there exist constants $c_0,K>0$ such that for every $a\geq0$, we have:
    \begin{align*}
        \int_0^ae^\ell|B(\ell)|\d\ell\leq K\int_{\mathbf{R}_{>0}^m}e^{\Lambda(\mathbf{s}_m,\mathbf{0}_n)-(s_1+\cdots+s_m)}\mathds{1}_{\{\Lambda(\mathbf{s}_m,\mathbf{0}_n)\leq a+c_0\}}\d\mathbf{s}_m.
    \end{align*}
    Let $D$ be the maximal coefficient of $\mathbf{s}_m\mapsto\Lambda(\mathbf{s}_m,\mathbf{0}_n)$. If $D=1$, we let $\lambda$ be any number in $(0,1)$, and if $D>1$, let $\lambda=\frac{1}{D}$. In both cases, for all $\mathbf{s}_m$, we have:
    \[
        \Lambda(\mathbf{s}_m,\mathbf{0}_n)-(s_1+\cdots+s_m)\leq (1-\lambda)\Lambda(\mathbf{s}_m,\mathbf{0}_n).
    \]
    This yields that there exist constants $c,c'>0$ such that:
    \[
        \int_0^ae^\ell|B(\ell)|\d\ell\leq Ke^{(1-\lambda)a}\int_{\mathbf{R}_{>0}^m}\mathds{1}_{\{\Lambda(\mathbf{s}_m,\mathbf{0}_n)\leq a+c_0\}}\d\mathbf{s}_m\leq c'(1+a)^ce^{(1-\lambda)a}.
    \]
    Hence $B\in\mathcal{SR}_w(\lambda)$.
    
    Let us now assume that $n-i>0$. Let $i_0=\max\{r\in\{1,\ldots,n\}|\frac{\partial\Lambda}{\partial x_r}\neq0\}$. If $i_0\leq i$, then $\mathbf{x}_n\mapsto\Lambda(\mathbf{0}_m,\mathbf{x}_n)$ is bounded on $C^{A,i}$, hence we conclude in this case that $B\in\mathcal{SR}_w(\lambda)$ for some $\lambda\in(0,1)$. We can therefore assume that $i_0>i$, in which case we can use Lemma \ref{S4:Lem:ConvSLT2}. For convenience of notations, we assume that $i_0=n$. Thus $\mathcal{L}_\ell^{K_n+1}[e^\ell B(\ell)]$ is a linear combination of elements of $\mathcal{R}_w(\nu)$, for some $\nu\in(0,1)$, and of terms of the form:
    \begin{align}\label{S4:Cor:ConvSLTUndbound:Demo1}
        \nonumber \mathcal{P}_\ell^{d+1}\mathcal{L}_\ell^{K_n-d}\bigg[ e^\ell\int_{\ell=H(\mathbf{s}_m,\mathbf{x}_{n-1})}f_n^{(d)}\Big(Z(\mathbf{x}_{n-1})+ & b\Big)\prod_{j=1}^mg_j(s_j)\prod_{r=1}^{n-1}f_r(x_r) \\
        & \times\varphi(\mathbf{s}_m)\mathds{1}_{\mathbf{R}_{>0}^m\times \widetilde{C}^{A,i}}(\mathbf{s}_m,\mathbf{x}_{n-1})\frac{\d\mathbf{s}_m\d\mathbf{x}_{n-1}}{\d\ell}\bigg],
    \end{align}
    where $d\leq K_n$, $\widetilde{C}$ is an $(n-1)$-dimensional admissible polytope, $H\colon\mathbf{R}_{>0}^m\times\widetilde{C}\longrightarrow\mathbf{R}$ is a simplified length-type function, $Z$ is a linear form on $\mathbf{R}^{n-1}$ and $b\in\mathbf{R}$ is a constant. The function $\mathbf{x}_{n-1}\mapsto f_n^{(d)}(Z(\mathbf{x}_{n-1})+b)\prod_{r=1}^{n-1}f_r(x_r)$ is therefore a polynomial function of the variable $\mathbf{x}_{n-1}$ of total degree $K_1+\cdots+K_n-d$. Hence (\ref{S4:Cor:ConvSLTUndbound:Demo1}) is a linear combination of elements of the form:
    \[
        \mathcal{P}_\ell^{d+1}\mathcal{L}_\ell^{K_n-d}\Big[e^\ell\Big(g_1\star\cdots\star g_m\star\widetilde{f}_1\star\cdots\star\widetilde{f}_{n-1}(\ell)\Big|_{\mathds{1}_{\mathbf{R}_{>0}^m\times\widetilde{C}^{A,i}}}^H\Big)\Big],
    \]
    where $\widetilde{f}_1,\cdots,\widetilde{f}_{n-1}$ are polynomial functions of degree $K_1+d_1,\ldots,K_{n-1}+d_{n-1}$ respectively, where $d_1+\ldots+d_{n-1}=K_n-d$. By induction, there exists a constant $\lambda\in(0,1)$ such that:
    \begin{align*}
        e^\ell\Big(g_1\star\cdots\star g_m\star\widetilde{f}_1\star\cdots\star\widetilde{f}_{n-1}(\ell)\Big|_{\mathds{1}_{\mathbf{R}_{>0}^m\times\widetilde{C}^{A,i}}}^H\Big) \in & \mathcal{F}_w^{K_{i+1}+d_{i+1}+\cdots+K_{n-1}+d_{n-1}+n-1-i}(1-\lambda) \\
        \subset & \mathcal{F}_w^{K_{i+1}+\cdots+K_n-d+n-1-i}(1-\lambda).
    \end{align*}
    It follows that (\ref{S4:Cor:ConvSLTUndbound:Demo1}) belongs to $\mathcal{F}_w^{K_{i+1}+\cdots+K_{n-1}+n-1-i}(1-\lambda)$, which yields:
    \[
        \mathcal{L}_\ell^{K_n+1}[e^\ell B(\ell)]\in\mathcal{F}_w^{K_{i+1}+\cdots+K_{n-1}+n-1-i}(1-\lambda),
    \]
    and $B\in\mathcal{SF}_w^{K_{i+1}+\cdots+K_n+n-1-i}(\lambda)$.
\end{proof}

\begin{Cor}\label{S4:Cor:ConvSLTGeneral}
    Let $C$ be an $n$-dimensional admissible polytope and $h\colon\mathbf{R}_{>0}^m\times C\longrightarrow\mathbf{R}$ be a simplified length-type function. We denote $\varphi\colon\mathbf{s}_m\mapsto e^{-(s_1+\cdots+s_m)}$. There exists $\lambda\in(0,1)$ such that for all $g_1,\ldots,g_m\in L^\infty(\mathbf{R}_{>0})$ and all polynomial functions $f_1,\ldots,f_n$ of degree $K_1,\ldots,K_n$ respectively, we have:
    \[
        g_1\star\cdots\star g_m\star f_1\star\cdots\star f_n\Big|_{\varphi\mathds{1}_{\mathbf{R}_{>0}^m\times C}}^h\in\mathcal{SF}_w^{K_1+\cdots+K_n+n-1}(\lambda).
    \]
\end{Cor}

\begin{proof}
    Since $C\subset\{0\leq x_1\leq\cdots\leq x_n\}$, it is partitionned into the sets $C\cap[0,A]^i\times(A,+\infty)^{n-i}$ for $i\in\{0,\ldots,n\}$, and we have:
    \[
        g_1\star\cdots\star g_m\star f_1\star\cdots\star f_n\Big|_{\varphi\mathds{1}_{\mathbf{R}_{>0}^m\times C}}^h=\sum_{i=0}^ng_1\star\cdots\star g_m\star f_1\star\cdots\star f_n\Big|_{\varphi\mathds{1}_{\mathbf{R}_{>0}^m\times (C\cap[0,A]^i\times(A,+\infty)^{n-i})}}^h.
    \]
    We conclude by using Corollary \ref{S4:Cor:ConvSLTUndbound}.
\end{proof}

\subsection{Convolution with respect to length-type functions}

\begin{Theo}\label{S4:Theo:ConvLenType}
    Let $h\colon\mathbf{R}_{>0}^n\times\mathbf{R}^n\times\mathbf{R}_{\geq0}^k\longrightarrow\mathbf{R}$ be a length-type function. There exists $\lambda\in(0,1)$ such that for every polynomial functions $f_1,\ldots,f_{2n+k}$ of degree $K_1,\ldots,K_{2n+k}$ respectively, we have:
    \[
        f_1\star\cdots\star f_{2n+k}\Big|_{\mathds{1}_{\mathbf{R}_{>0}^n\times\mathbf{R}^n\times\mathbf{R}_{\geq0}^k}}^h\in\mathcal{SF}_w^{K_1+\ldots+K_{2n+k}+2n+k-1}(\lambda).
    \]
\end{Theo}

\begin{proof}
    Let $u\colon\ell\mapsto2\log(2\cosh(\frac{\ell}{2}))$ and $H=u\circ h$. Then $H$ is a pseudo length-type function, and according to Lemma \ref{S4:Lem:ChangeVar}, we have:
    \[
        f_1\star\cdots\star f_{2n+k}(\ell)\Big|_{\mathds{1}_{\mathbf{R}_{>0}^n\times\mathbf{R}^n\times\mathbf{R}_{\geq0}^k}}^h = \Big(1+O(e^{-\ell})\Big)\Big(f_1\star\cdots\star f_{2n+k}\big(\ell+O(e^{-\ell})\big)\Big|_{\mathds{1}_{\mathbf{R}_{>0}^n\times\mathbf{R}^n\times\mathbf{R}_{\geq0}^k}}^H\Big),
    \]
    hence it is sufficient to prove that:
    \[
        f_1\star\cdots\star f_{2n+k}\Big|_{\mathds{1}_{\mathbf{R}_{>0}^n\times\mathbf{R}^n\times\mathbf{R}_{\geq0}^k}}^H\in\mathcal{SF}_w^{K_1+\ldots+K_{2n+k}+2n+k-1}(\lambda).
    \]
    We partition the space $\mathbf{R}_{>0}^n\times\mathbf{R}^n\times\mathbf{R}_{\geq0}^k$ into the sets $E_I\times\mathbf{R}^n\times\mathbf{R}_{\geq0}^k$, for $I\subset\{1,\ldots,n\}$, defined by (\ref{S3:Eq:Partition1}). Given a subset $I\subset\{1,\ldots,n\}$, we note:
    \[
        \begin{array}{cccc}
            \widetilde{H}_I\colon & F_I\times\mathbf{R}^n\times\mathbf{R}_{\geq0}^k & \longrightarrow & \mathbf{R} \\
             & (\mathbf{s}_m,\mathbf{t}_n,\mathbf{y}_k) & \longmapsto & H(\Phi_I^{-1}(\mathbf{s}_n),\mathbf{t}_n,\mathbf{y}_k)
        \end{array},
    \]
    where the diffeomorphism $\Phi_I\colon E_I\longrightarrow F_I$ is defined by (\ref{S3:Eq:Partition2}).
    
    By possibly performing a permutation of variables, we can assume without loss of generality for convenience of notations that $I=\{1,\ldots,m\}$. In this case we have $E_I=(0,1)^m\times[1,+\infty)^{n-m}$, $F_I=\mathbf{R}_{>0}^m\times[1,+\infty)^{n-m}$ and if $(\mathbf{s}_m,\mathbf{x}_n)\in F_I$, we have:
    \[
        \Phi_I^{-1}(\mathbf{s}_m,\mathbf{x}_{n-m}) = (e^{-s_1},\ldots,e^{-s_m},\mathbf{x}_{n-m}).
    \]
    The absolute value of the Jacobian determinant of $\Phi_I^{-1}$ is therefore given by the function $\varphi\colon(\mathbf{s}_m,\mathbf{x}_{n-m})\mapsto e^{-(s_1+\cdots+s_m)}$. According to Lemma \ref{S4:Lem:ChangeVar}, it follows:
    \[
        f_1\star\cdots\star f_{2n+k}\Big|_{\mathds{1}_{E_I\times\mathbf{R}^n\times\mathbf{R}_{\geq0}^k}}^H = g_1\star\cdots\star g_m\star f_{m+1}\star\cdots\star f_{2n+k}\Big|_{\varphi\mathds{1}_{F_I\times\mathbf{R}^n\times\mathbf{R}_{\geq0}^k}}^{H_I},
    \]
    where for all $j\in\{1,\ldots,m\}$, $g_j\colon s_j\mapsto f_j(e^{-s_j})$.
    
    Using Theorem \ref{S3:Theo:SimpFormPLT}, we partition $F_I\times\mathbf{R}^n\times\mathbf{R}_{\geq0}^k$ into cones $U_1,\ldots,U_{N_I}$ such that for each $i\in\{1,\ldots,N_I\}$, we have a linear isomorphism $\Psi_i\colon U_i\longrightarrow\mathbf{R}_{>0}^m\times C_i$, where $C_i$ is an admissible polytope, and $\widetilde{H}_{I,i}=\widetilde{H}_I\circ\Psi_i^{-1}\colon\mathbf{R}_{>0}^m\times C_i\longrightarrow\mathbf{R}$ is a simplified length-type function. For each $i\in\{1,\ldots,N_I\}$, we note $\Psi_i^{-1}=(\Id_{\mathbf{R}_{>0}^m},S_{i,1},\ldots,S_{i,2n-m+k})$. Since each $\Psi_i^{-1}$ is linear, its Jacobian determinant is constant, hence $f_1\star\cdots\star f_{2n+k}|_{\mathds{1}_{E_I\times\mathbf{R}^n\times\mathbf{R}_{\geq0}^k}}^{H_I}$ is a linear combination of terms of the form:
    \begin{equation}\label{S4:Theo:ConvLenType:Demo1}
        \int_{\ell=H_{I,i}(\mathbf{s}_m,\mathbf{z}_{2n-m+k})}\varphi(\mathbf{s}_m)\prod_{j=1}^mg_j(s_j)\prod_{r=1}^{2n-m+k}f_r(S_{i,r}(\mathbf{z}_{2n-m+k}))\mathds{1}_{\mathbf{R}_{>0}^m\times C_i}(\mathbf{s}_m,\mathbf{z}_{2n-m+k})\frac{\d\mathbf{s}_m\d\mathbf{z}_{2n-m+k}}{\d\ell}.
    \end{equation}
    Since the $S_{i,r}$'s are linear forms, (\ref{S4:Theo:ConvLenType:Demo1}) is a linear combination of terms of the form:
    \[
        g_1\star\cdots\star g_m\star z_1^{d_1}\star\cdots\star z_{2n-m+k}^{d_{2n-m+k}}\Big|_{\varphi\mathds{1}_{\mathbf{R}_{>0}^m\times C_i}}^{H_{I,i}},
    \]
    where $d_1+\cdots+d_{2n-m+k}\leq K_{m+1}+\cdots+K_{2n+k}$. The conclusion follows from Corollary \ref{S4:Cor:ConvSLTGeneral}
\end{proof}

\begin{Rem}
    Let us check what would have happened if we had allowed the length-type functions to be defined using coefficients with arbitrary signs. Recall that we noticed earlier that if it had been the case, the linear changes of variables provided by Theorem \ref{S3:Theo:SimpFormPLT} would impact the variable $\mathbf{s}_m$. Therefore, after having applied such a linear change of variable $\Psi_i^{-1}$, the Jacobian determinant $\varphi\colon(\mathbf{s}_m,\mathbf{x}_{n-m})\mapsto e^{-(s_1+\cdots+s_m)}$ would be re-expressed as:
    \[
        \varphi\circ\Psi_i^{-1}\colon\mathbf{z}_{2n+k}\mapsto e^{-T(\mathbf{z}_{2n+k})},
    \]
    for some positive linear form $T$. The issue is that we cannot control anymore the variables on which the Jacobian determinant $\varphi\circ\Psi_i^{-1}$ depends. Hence we have to replace the conclusion of Theorem \ref{S4:Theo:OpeLConvCone} by the expression (\ref{S4:Eq:ProbArbSign}), which makes us unable to perform the induction we used to prove Corollary \ref{S4:Cor:ConvSLTUndbound}.
    
    This is the reason why we insisted on the fact that we require the length-type functions to be defined with positive coefficients.
\end{Rem}

\section{Application to the integration of geometric random variables and geodesic counting}

Using the integration formula from Theorem \ref{S2:Theo:IntFormGeneral} together with Theorem \ref{S4:Theo:ConvLenType}, we are now able to prove the following.

\begin{Theo}\label{S5:Theo:IntRandVarPTS}
    Let $\gamma$ be a loop in minimal position on $S_{g,n}$. There exist $\lambda_\gamma\in(0,1)$ and a function $V_{\mathcal{O}_\gamma}(~.~,\mathbf{L}_n)\in\mathcal{SF}_w^{6g-7+2n}(\lambda_\gamma)$ such that for all test function $F\colon\mathbf{R}\longrightarrow\mathbf{C}$, we have:
    \[
        \mathbb{E}[F^\gamma]=\int_0^\infty F(\ell)V_{\mathcal{O}_\gamma}(\ell,\mathbf{L}_n)\frac{\d\ell}{V_{g,n}(\mathbf{L}_n)}.
    \]
\end{Theo}

\begin{proof}
    We denote $\beta_1,\ldots,\beta_N$ the connected components of $S(\gamma)$, and we extract a maximal multi-curve $\Gamma$ from $(\beta_1,\ldots,\beta_N)$. Let $N'$ be the number of curves of $\Gamma$. The volume $V_\Gamma(\mathbf{y}_{N'},\mathbf{L}_n)$ of the moduli space of $S_{g,n}\backslash S(\gamma)$ with boundary lengths given by $y_1,\ldots,y_{N'},L_1,\ldots,L_n$ is a polynomial and the length function $h_\gamma$ (written in appropriate Fenchel-Nielsen coordinates) is a length-type function. Hence, by linearity and using Theorem \ref{S2:Theo:IntFormGeneral} and Theorem \ref{S4:Theo:ConvLenType} we can conclude.
\end{proof}

As a direct consequence of Theorem \ref{S5:Theo:IntRandVarPTS} we obtain an asymptotic expansion of the average counting functions $a\mapsto\mathbb{E}[N_\gamma(a)]$ for arbitrary loops.

\begin{Cor}\label{S5:Cor:CountRandVarPTS}
    Let $\gamma$ be a loop in minimal position on $S_{g,n}$. Then $a\mapsto\mathbb{E}[N_\gamma(a)]\in\mathcal{SF}^{6g-6+2n}(\lambda_\gamma)$.
\end{Cor}

\begin{proof}
    Let $P_\gamma^{\mathbf{L}_n}$ be the polynomial part of $V_{\mathcal{O}_\gamma}(~.~,\mathbf{L}_n)$ and $r_\gamma^{\mathbf{L}_n}=V_{\mathcal{O}_\gamma}(~.~,\mathbf{L}_n)-P_\gamma^{\mathbf{L}_n}$. We have $V_{\mathcal{O}_\gamma}(~.~,\mathbf{L}_n)\in\mathcal{SF}_w^{6g-7+2n}(\lambda_\gamma)$, hence $r_\gamma^{\mathbf{L}_n}\in L^1(\mathbf{R}_{\geq0})$ and:
    \[
        \int_a^\infty|r_\gamma^{\mathbf{L}_n}(\ell)|\d\ell=\underset{a\to\infty}{O}\big((1+a)^ce^{-\lambda_\gamma a}\big).
    \]
    Denoting by $Q_\gamma^{\mathbf{L}_n}$ the primitive of $P_\gamma^{\mathbf{L}_n}$ such that $Q_\gamma(0)=\int_0^\infty r_\gamma^{\mathbf{L}_n}(\ell)\d\ell$, we have:
    \[
        \mathbb{E}[N_\gamma(a)]=Q_\gamma^{\mathbf{L}_n}(a)+\underset{a\to\infty}{O}\big((1+a)^ce^{-\lambda_\gamma a}\big).
    \]
\end{proof}

Another consequence of Theorem \ref{S5:Theo:IntRandVarPTS} is that it enables us to estimate the period of the Hamiltonian flow of length functions of filling loops on small hyperbolic surfaces.

By \emph{small} hyperbolic surfaces, we mean surfaces whose Teichmüller space has dimension 2. These hyperbolic surfaces are homeomorphic to either the one-holed torus $S_{1,1}$ or the four-holed sphere $S_{0,4}$.

Given a filling loop $\gamma$ on $S_{g,n}$, its length function $\ell_\gamma$ is a proper function on $\mathcal{T}_{g,n}(\mathbf{L}_n)$, and due to the geodesic convexity of the length functions proved by Wolpert in \cite[Corollary 4.7]{Wol87}, $\ell_\gamma$ has a unique minimum on $\mathcal{T}_{g,n}(\mathbf{L}_n)$ that we denote by $\ell_\gamma^{\min}(\mathbf{L}_n)$. Thus if $(g,n)\in\{(1,1),(0,4)\}$, on the one hand for every $\ell>\ell_\gamma^{\min}(\mathbf{L}_n)$ the level set $\ell_\gamma^{-1}(\ell)$ is a compact manifold of dimension 1, i.e. it is homeomorhic to a circle. On the other hand the orbits of the Hamiltonian flow $(\varphi_\gamma^t)$ of $\ell_\gamma$ (defined using the Weil-Petersson symplectic structure of $\mathcal{T}_{g,n}(\mathbf{L}_n)$) are contained in the level sets of $\ell_\gamma$, hence for every $Y=[X,f]\in\mathcal{T}_{g,n}(\mathbf{L}_n)$ the function $t\mapsto\varphi_\gamma^t(Y)$ is periodic and its period only depends on $\ell=\ell_\gamma(Y)$. Let us denote by $T_\gamma(\ell,\mathbf{L}_n)$ this period.

\begin{Theo}\label{S5:Theo:PeriodHamilton}
    With the same notations as above, the function $\ell\mapsto T_\gamma(\ell,\mathbf{L}_n)$ is smooth on $(\ell_\gamma^{\min}(\mathbf{L}_n),+\infty)$ and we have $T_\gamma(\ell,\mathbf{L}_n)=V_{\mathcal{O}_\gamma}(\ell,\mathbf{L}_n)$, i.e. for every test function $F\colon\mathbf{R}\longrightarrow\mathbf{C}$ we have:
    \[
        \mathbb{E}[F^\gamma]=\int_{\ell_\gamma^{\min}(\mathbf{L}_n)}^\infty F(\ell)T_\gamma(\ell,\mathbf{L}_n)\frac{\d\ell}{V_{g,n}(\mathbf{L}_n)}.
    \]
     In particular $T_\gamma(~.~,\mathbf{L}_n)$ only depends on $\mathcal{O}_\gamma$ and is of class $\mathcal{SF}_w^1(\lambda_\gamma)$. 
\end{Theo}

\begin{proof}
    Since in this proof $\mathbf{L}_n$ is fixed, for convenience of notations we will only denote $T_\gamma(\ell)=T_\gamma(\ell,\mathbf{L}_n)$ and $\ell_\gamma^{\min}=\ell_\gamma^{\min}(\mathbf{L}_n)$. Let $Y_0=\ell_\gamma^{-1}(\ell_\gamma^{\min})$.
    
    Let us first prove that $T_\gamma$ is smooth on $(\ell_\gamma^{\min},+\infty)$. Let $Y\colon[\ell_\gamma^{\min},+\infty)\longrightarrow\mathcal{T}_{g,n}(\mathbf{L}_n)$ be a smooth right inverse of $\ell_\gamma$ and :
    \[
        \begin{array}{cccc}
            \Phi\colon & [\ell_\gamma^{\min},+\infty)\times\mathbf{R} & \longrightarrow & \mathcal{T}_{g,n}(\mathbf{L}_n) \\
             & (\ell,t) & \longmapsto & \varphi_\gamma^t(Y(\ell))
        \end{array},
    \]
    where $(\varphi_\gamma^t)$ is the Hamiltonian flow of $\ell_\gamma$.
    
    We first prove that $\Phi$ is a local symplectomorphism $((\ell_\gamma^{\min},+\infty)\times\mathbf{R},\omega_0)\longrightarrow(\mathcal{T}_{g,n}(\mathbf{L}_n),\omega_\WP)$, where $\omega_0$ is the standard symplectic form on $\mathbf{R}^2$. Let $X_\gamma$ be the symplectic gradient of $\ell_\gamma$ (hence $(\varphi_\gamma^t)$ is the flow of the vector field $X_\gamma$). Given $(\ell,t)\in(\ell_\gamma^{\min},+\infty)\times\mathbf{R}$ and $(u,v)\in\mathbf{R}^2$, we have:
    \[
        \d\Phi(\ell,t)\cdot(u,v) = u\d\varphi_\gamma^t(Y(\ell))\cdot Y'(\ell)+vX_\gamma(\varphi_\gamma^t(Y(\ell))).
    \]
    Thus:
    \begin{align*}
        (\Phi^*\omega_\WP)_{(\ell,t)}\big((u_1,v_1),(u_2,v_2)\big) = & (u_1v_2-u_2v_1)\omega_\WP\Big(\d\varphi_\gamma^t(Y(\ell))\cdot Y'(\ell),X_\gamma(\varphi_\gamma^t(Y(\ell)))\Big) \\
        = & \omega_0\big((u_1,v_1),(u_2,v_2)\big)\d\ell_\gamma(\varphi_\gamma^t(Y(\ell)))\cdot\d\varphi_\gamma^t(Y(\ell))\cdot Y'(\ell) \\
        = & \omega_0\big((u_1,v_1),(u_2,v_2)\big)\frac{\d}{\d\ell}\Big(\ell_\gamma\circ\varphi_\gamma^t\circ Y(\ell)\Big) \\
        = & \omega_0\big((u_1,v_1),(u_2,v_2)\big).
    \end{align*}
    The map $\Phi$ is therefore a smooth local symplectomorphism. Let us now prove that $T_\gamma$ is smooth on $(\ell_\gamma^{\min},+\infty)$. Let:
    \[
        U_0 = \Big\{(\ell,t)\in\mathbf{R}^2\Big|\ell>\ell_\gamma^{\min},t\in[0,T_\gamma(\ell))\Big\}
    \]
    When restricted to $U_0$, $\Phi$ is a bijection onto $\mathcal{T}_{g,n}(\mathbf{L}_n)\backslash\{Y_0\}$. Let $\pi\colon\mathcal{T}_{g,n}\backslash\{Y_0\}\longrightarrow\mathbf{R}$ be such that $(\ell_\gamma,\pi)\colon\mathcal{T}_{g,n}\backslash\{Y_0\}\longrightarrow U_0$ is the inverse of $\Phi_{|U_0}$. Let $\ell>\ell_\gamma^{\min}$ and $\varepsilon>0$ be small enough so that $\Phi$ is injective on $(\ell-\varepsilon,\ell+\varepsilon)\times(-\varepsilon,\varepsilon)$. For every $x\in(\ell-\varepsilon,\ell+\varepsilon)$, we have $T_\gamma(x)\geq2\varepsilon$, otherwise $\Phi$ would not be injective on $(\ell-\varepsilon,\ell+\varepsilon)\times(-\varepsilon,\varepsilon)$. Let:
    \[
        \begin{array}{cccc}
            \widetilde{Y}\colon & (\ell-\varepsilon,\ell+\varepsilon) & \longrightarrow & \mathcal{T}_{g,n}(\mathbf{L}_n) \\
            & x & \mapsto & \Phi(x,-\frac{\varepsilon}{2})
        \end{array}.
    \]
    For all $x\in(\ell-\varepsilon,\ell+\varepsilon)$, we have $\pi(\widetilde{Y}(x)) = T_\gamma(x)-\frac{\varepsilon}{2}$, hence $T_\gamma$ is smooth on $(\ell-\varepsilon,\ell+\varepsilon)$. Since this is true for any $\ell>\ell_\gamma^{\min}$ and any small enough $\varepsilon>0$, $T_\gamma$ is smooth on $(\ell_\gamma^{\min},+\infty)$.
    
    It follows that the set $U_1=\{(\ell,t)|\ell>\ell_\gamma^{\min},0<t<T_\gamma(\ell)\}$ is an open set of $\mathbf{R}^2$ and when restricted to $U_1$, $\Phi$ becomes a symplectomorphism $(U_1,\omega_0)\longrightarrow(\mathcal{T}_{g,n}(\mathbf{L}_n)\backslash Y([\ell_\gamma^{\min},+\infty)),\omega_\WP)$. Hence if $A\subset\mathbf{R}$ is a measurable set, we have:
    \begin{align*}
        (\ell_\gamma)_*\mu_\WP(A) = & \int_{U_1}\mathds{1}_A(\ell_\gamma\circ\Phi(\ell,t))\d\ell\d t \\
        = & \int_{\ell_\gamma^{\min}}^\infty\bigg(\int_0^{T_\gamma(\ell)}\mathds{1}_A(\ell_\gamma(\varphi_\gamma^t(Y(\ell))))\d t\bigg)\d\ell \\
        = & \int_{\ell_\gamma^{\min}}^\infty\mathds{1}_A(\ell)T_\gamma(\ell)\d\ell.
    \end{align*}
    We therefore have $\d\big((\ell_\gamma)_*\mu_\WP\big)=T_\gamma(\ell)\d\ell$, which yields $T_\gamma(\ell)=V_{\mathcal{O}_\gamma}(\ell,\mathbf{L}_n)$ according to Theorem \ref{S5:Theo:IntRandVarPTS}.
\end{proof}

\begin{Rem}
    It follows from Mirzakhani's estimates from Corollary \ref{S1:Cor:MirzAsympt} that the leading term of $T_\gamma(\ell,\mathbf{L}_n)$ is $2n_\gamma\ell$.
\end{Rem}

Another consequence is that for an arbitrary topological signature $(g,n)$, if $\gamma$ is a loop on $S_{g,n}$ such that the surface $S(\gamma)$ filled by $\gamma$ is homeomorphic to a one-holed torus or a four-holed sphere, using Theorems \ref{S2:Theo:IntFormGeneral}, \ref{S5:Theo:PeriodHamilton} and Fubini Theorem, we have:
\[
    V_{\mathcal{O}_\gamma}(\ell,\mathbf{L}_n)=\frac{\mathds{1}_{\{\ell\geq\ell_\gamma^{\min}(\mathbf{0}_N)\}}}{m(\gamma)2^{M_\gamma}}\int_{\mathbf{R}_{>0}^N}\mathds{1}_{\{\ell_\gamma^{\min}(\mathbf{y}_N)\leq\ell\}}y_1\cdots y_NV_\Gamma(\mathbf{y}_N,\mathbf{L}_n)T_\gamma(\ell,\mathbf{y}_N)\d\mathbf{y}_N
\]

\printbibliography

\end{document}